\documentclass[a4paper,11pt]{article}
\usepackage{amsmath,amsthm,amsfonts,amssymb,enumerate,graphicx,float,wrapfig,calc,microtype,mathtools,thmtools,underscore}
\usepackage[T1]{fontenc}
\usepackage[usenames,dvipsnames,svgnames,table]{xcolor}
\usepackage[unicode=true,
pdftitle={Nonrepetitive Graph Colouring},
pdfauthor={David R. Wood},
colorlinks =true,
linkcolor={black},
urlcolor={blue!80!black},
filecolor={blue!80!black},
citecolor={black},
pagebackref,
anchorcolor=blue,
filecolor=blue, 
menucolor=blue]{hyperref}

\renewcommand*{\backref}[1]{}
\renewcommand*{\backrefalt}[4]{\footnotesize\hspace*{0pt}\hfill \ifcase #1 \mbox{[not cited]} \or  \mbox{[p.\,#2]}  \else \mbox{[pp.\,#2]} \fi}

\usepackage[noabbrev,capitalise]{cleveref}
\crefname{lem}{Lemma}{Lemmas}
\crefname{thm}{Theorem}{Theorems}
\crefname{cor}{Corollary}{Corollaries}
\crefname{prop}{Proposition}{Propositions}
\crefname{conj}{Conjecture}{Conjectures}
\crefname{open}{Open Problem}{Open Problems}

\crefformat{equation}{(#2#1#3)}
\Crefformat{equation}{Equation #2(#1)#3}

\usepackage[shortlabels]{enumitem}
\setlist[itemize]{topsep=0ex,itemsep=0ex,parsep=0.4ex}
\setlist[enumerate]{topsep=0ex,itemsep=0ex,parsep=0.4ex}

\usepackage[longnamesfirst,numbers,sort&compress]{natbib}
\usepackage{hypernat}
\makeatletter
\def\NAT@spacechar{~}
\makeatother

\usepackage[lmargin=30mm,rmargin=30mm,tmargin=30mm,bmargin=30mm]{geometry}
\renewcommand{\baselinestretch}{1.1}
\allowdisplaybreaks

\DeclarePairedDelimiter{\ceil}{\lceil}{\rceil}
\DeclarePairedDelimiter{\floor}{\lfloor}{\rfloor}

\newcommand{\half}{\ensuremath{\protect\tfrac{1}{2}}}

\renewcommand{\ge}{\geqslant}
\renewcommand{\le}{\leqslant}
\renewcommand{\geq}{\geqslant}
\renewcommand{\leq}{\leqslant}
\newcommand{\blah}[1]{\ensuremath{\protect\langle#1\rangle}}

\DeclareMathOperator{\tw}{tw}

\DeclareMathOperator{\pw}{pw}

\DeclareMathOperator{\dist}{dist}
\DeclareMathOperator{\col}{col}
\newcommand{\acy}{\chi_{\text{a}}}
\newcommand{\st}{\chi_{\text{s}}}

\newcommand{\WLLL}{Weighted Local Lemma}
\newcommand{\GLLL}{General Local Lemma}
\newcommand{\ELLL}{Optimised Weighted Local Lemma}

\renewcommand{\thefootnote}{\fnsymbol{footnote}}	
\allowdisplaybreaks

\newcommand{\doi}[1]{doi:\,\href{https://dx.doi.org/#1}{#1}}

\theoremstyle{plain}
\newtheorem{thm}{Theorem}[section]
\newtheorem{lem}[thm]{Lemma}
\newtheorem{cor}[thm]{Corollary}
\newtheorem{prop}[thm]{Proposition}
\newtheorem{claim}{Claim}
\theoremstyle{definition}
\newtheorem{open}[thm]{Open Problem}

\newcommand{\wish}[1]{\framebox{\textcolor{red}{#1}}}
\renewcommand{\wish}[1]{}

\newcommand{\GG}{\mathcal{G}}
\newcommand{\DD}{\mathcal{D}}

\newcommand{\NN}{\mathbb{N}}
\newcommand{\EE}{\mathcal{E}}
\newcommand{\prob}{\mathbb{P}}

\newcommand{\pich}{\pi_{\textup{ch}}}

\begin{document}

\author{David~R.~Wood\footnotemark[2]}

\footnotetext[2]{School of Mathematics, Monash University, Melbourne, Australia  (\texttt{david.wood@monash.edu}).\\ Research supported by the Australian Research Council.}

\sloppy

\title{\boldmath\bf Nonrepetitive Graph Colouring}
\maketitle


\begin{abstract}
A vertex colouring of a graph $G$ is \emph{nonrepetitive} if $G$ contains no path for which the first half of the path is assigned the same sequence of colours as the second half. Thue's famous theorem says that every path is nonrepetitively 3-colourable. This paper surveys results about nonrepetitive colourings of graphs. The goal is to give a unified and comprehensive presentation of the major results and proof methods, as well as to highlight numerous open problems. 
\end{abstract}

\renewcommand{\thefootnote}{\arabic{footnote}}

\newpage
\tableofcontents
\newpage

\section{Introduction}
\label{Introduction}

In 1906, \citet{Thue06} constructed arbitrarily long words $w_1w_2\ldots$ on an alphabet of three symbols with no repeated consecutive blocks; that is, there are no integers $i,k\in\NN$ such that 
$w_iw_{i+1}\dots w_{i+k-1}=w_{i+k}w_{i+k+1}\dots w_{i+2k-1}$. Such a word is called \emph{square-free}. 
Thue's Theorem is a foundational result in the combinatorics of words (see the surveys \citep{BK03,Berstel05,BP07,Currie-TCS05,Lothaire02,Lothaire05,Currie93,Berstel95})
and has also found applications in semi-group theory~\citep{BN71}, dynamics~\citep{MH44,Queffler10}, and most famously in the solution of the Burnside problem for groups by \citet{NA68a,NA68b,NA68c}.

In 2002, \citet{AGHR02} introduced a graph-theoretic generalisation of square-free words. They defined a vertex colouring of a graph to be \emph{nonrepetitive} if there is no path for which the first half of the path is assigned the same sequence of colours as the second half.  Thue's Theorem says that every path is nonrepetitively 3-colourable. Nonrepetitive graph colouring is interesting for several reasons: 
\begin{itemize}
\item It is a natural marriage of two major areas of combinatorial mathematics, combinatorics of words and graph colouring. 
\item Several advanced techniques have been used to obtain results in nonrepetitive graph colouring, such as the Lov\'asz Local Lemma, entropy compression, layered treewidth, and product structure theorems. Indeed, in some cases,  nonrepetitive graph colouring has motivated the development of these general-purpose tools that have then been applied to other areas.
\item Nonrepetitive graph colouring  is one of the most illustrative examples of the use of the Lov\'asz Local Lemma, since it requires the Lov\'asz Local Lemma in its full generality. I recommend teaching \cref{2DeltaSquared} in any course on the probabilistic method. 
\item Nonrepetitive graph colouring turns out to be a central concept in graph sparsity. Indeed, graph classes with bounded expansion can be characterised in terms of nonrepetitive colorings (see \cref{BoundedExpansionIffNonRep}). 
\item One of the most important recent developments in algorithmic graph theory has been the constructive proof of the Lov\'asz Local Lemma due to \citet{MoserTardos}. This lead to what Terry Tao dubbed the `entropy compression' method. Nonrepetitive graph colouring was one of the first applications of this method that showed that entropy compression can give better results than those obtained using the Lov\'asz Local Lemma~\citep{DJKW16,GKM11}. 
\item Several recent papers have presented generalised Moser--Tardos frameworks improving the original work in various ways~\citep{AIK19,Harris19,HSS11,HS17,HV20}. Nonrepetitive graph colouring has been a key test case here. One reason for this is that when modelling nonrepetitive colouring using the Lov\'asz Local Lemma the number of bad events (one for each even path) grows exponentially with the size of the graph, which is an obstacle for  polynomial-time algorithms. 
\end{itemize}

This paper surveys results about nonrepetitive colourings of graphs. The goal is to give a unified and comprehensive presentation of the major results and proof methods from the literature, to highlight numerous open problems, and to present a couple of original theorems. For previous surveys, see \citep{AG05,Gryczuk-IJMMS07,Grytczuk-DM08,CSZ,Grytczuk-Berge,Skrabu15}. 

\subsection{Path-Nonrepetitive Colourings}

We consider finite undirected graphs with no loops or parallel edges. A \emph{colouring} of a graph $G$ is a function $\phi$ that assigns a `colour' to each vertex of $G$. A colouring $\phi$ of $G$ is a \emph{$k$-colouring} if $|\{\phi(v):v\in V(G)\}|\leq k$. A colouring $\phi$ of $G$ is \emph{proper} if $\phi(v)\neq\phi(w)$ for each edge $vw\in E(G)$. The \emph{chromatic number} $\chi(G)$ is the minimum integer $k$ for which there exists a proper $k$-colouring of $G$. If $\phi$ is a colouring of $G$, then a sequence $(v_1,v_2,\dots,v_{2t})$ of vertices in $G$ is \emph{$\phi$-repetitive} if $\phi(v_i)=\phi(v_{t+i})$ for each $i\in\{1,\dots,t\}$. A $\phi$-repetitive sequence is also said to be \emph{repetitively coloured by $\phi$}. A \emph{walk} in a graph $G$ is a sequence $(v_1,v_2,\dots,v_{t})$ of vertices in $G$ such that $v_iv_{i+1}\in E(G)$ for each $i\in\{1,\dots,t-1\}$. A \emph{path} in a graph $G$ is a walk $(v_1,v_2,\dots,v_{t})$ in $G$ such that $v_i\neq v_j$ for all distinct $i,j\in\{1,\dots,t\}$. 

A colouring $\phi$ of a graph $G$ is \emph{path-nonrepetitive}, or simply \emph{nonrepetitive}, if no path of $G$ is $\phi$-repetitive. The \emph{(path-)nonrepetitive chromatic number} $\pi(G)$ is the minimum integer $k$ such that $G$ admits a nonrepetitive $k$-colouring. Note that $\pi(G)$ is also called the \emph{Thue chromatic number} or \emph{square-free chromatic number} of $G$. Thue's theorem mentioned above says that paths are nonrepetitively 3-colourable. Every path-nonrepetitive colouring  is proper, as otherwise like-coloured adjacent vertices would form a repetitively coloured path on 2 vertices. Moreover, every nonrepetitive colouring has no $2$-coloured $P_4$ (a path on four vertices). A proper colouring with no $2$-coloured $P_4$ is called a \emph{star colouring} since each bichromatic subgraph is a star forest; see \citep{ACKKR04,Grunbaum73,FRR04,Borodin-DM79,Wood-DMTCS05,NesOss03}. The \emph{star chromatic number} $\st(G)$ is the minimum number of colours in a proper colouring of $G$ with no $2$-coloured $P_4$. Thus
\begin{equation}
\chi(G)\leq \st(G)\leq \pi(G).
\end{equation}

Starting with the seminal work of \citet{AGHR02}, nonrepetitive colourings of graphs have now been widely studied, including for the following graph classes: cycles \citep{Currie-EJC02}, trees \citep{BGKNP07,KP-DM08,FOOZ11}, outerplanar graphs \citep{KP-DM08,BV07}, graphs with bounded treewidth \citep{KP-DM08,BV07}, graphs with bounded degree~\citep{AGHR02,Gryczuk-IJMMS07,HJ-DM11,DJKW16,HS17,Rosenfeld20}, graphs excluding a fixed immersion~\citep{WW19}, planar graphs~\citep{HJSS11,Przybyo14,Przybyo13,BC13,JS12,JS09,JS09,BDMR17,DFJW13,DEJWW20}, graphs embeddable on a fixed surface~\citep{DEJWW20,DMW17}, graphs excluding a fixed minor~\citep{DEJWW20,DMW17}, graphs excluding a fixed topological minor~\citep{DEJWW20,DMW17}, and graph subdivisions \citep{Gryczuk-IJMMS07,BW08,MS09,PZ09,NOW11,GPZ11,DJKW16}. \cref{ResultsTable} summarises many of these results. 

	\begin{table}
		\setlength{\tabcolsep}{1ex}
		\begin{center}
			\caption{Lower and upper bounds on $\pi$ and $\sigma$ for various graph classes.\label{ResultsTable}}
			\begin{tabular}{lccr}
				\hline
				graph class & $\pi$  & $\sigma$ & reference\\
				\hline
				paths & $3$ &  $4$ & \S\ref{Paths}\\
				cycles & $3\dots4$ &  $4\dots5$ & \S\ref{Cycles}\\
				pathwidth $k$ & $k+1\dots 2k^2+6k+1$ &$(2k^2+6k+1)(k\,\Delta+1)$ & \S\ref{Pathwidth} \\
				trees & $4$ &  $\Delta+1\dots4\Delta$ & \S\ref{Trees}\\
				outerplanar & $7\dots12$ &  $\Theta(\Delta)$ & \S\ref{Outerplanar}\\
				treewidth $k$ & $\binom{k+2}{2}\ldots4^k$ & $O(\min\{4^kk\Delta,k^2\Delta^2\})$ & \S\ref{Treewidth},\ref{TreewidthDegree}\\
				planar & $11\dots768$ &   $\Theta(\Delta)$ & \S\ref{Planar}\\
				Euler genus $g$ & $\Omega(g^{3/5}/\log^{1/5}g)\dots O(g)$ &  $O(g\Delta)$ & \S\ref{Surfaces}\\
				excluded minor & $\Theta(1)$  & $\Theta(\Delta)$ & \S\ref{MinorClosedClass}\\
				excluded topo.~minor & $\Theta(1)$ &   $\Theta(\Delta)$ & \S\ref{MinorClosedClass}\\
				max degree $\Delta$ & $\Omega(\Delta^2/\log \Delta)\dots(1+o(1))\Delta^2$ & ? & \S\ref{BoundedDegree}\\
				\hline
			\end{tabular}
		\end{center}
	\end{table}
	
\subsection{Walk-Nonrepetitive Colourings}

While path-nonrepetitive colourings are the focus of this survey, we also present results about colourings of graphs that are nonrepetitive on walks, previously studied in \citep{BV08,BW08,Aprile14}. A walk $(v_1,\dots,v_{2t})$ in a graph is \emph{boring} if $v_i=v_{t+i}$ for each $i\in\{1,\dots,t\}$. Every colouring of a boring walk is repetitive. So \citet{BW08} defined a colouring to be \emph{walk-nonrepetitive} if every repetitively coloured walk is boring. For a graph $G$, the \emph{walk-nonrepetitive chromatic number} $\sigma(G)$ is the minimum number of colours in a walk-nonrepetitive colouring of $G$. Bounds on $\sigma$ for various classes are presented in \cref{ResultsTable}.

\subsection{Stroll-Nonrepetitive Colourings}

The following notion sits between paths and walks, and is important for many proofs that follow. A \emph{stroll} in a graph $G$ is a walk  $(v_1,\dots,v_{2t})$ such that $v_i\neq v_{t+i}$ for each $i\in\{1,\dots,t\}$. A colouring of $G$ is \emph{stroll-nonrepetitive} if no stroll is repetitively coloured. For a graph $G$, the \emph{stroll-nonrepetitive chromatic number} $\rho(G)$ is the minimum number of colours in a stroll-nonrepetitive colouring of $G$. Every walk-nonrepetitive colouring is stroll-nonrepetitive and every stroll-nonrepetitive colouring is path-nonrepetitive. Thus every graph $G$ satisfies $$\pi(G) \leq \rho(G) \leq \sigma(G).$$

At first glance the definition of stroll-nonreptitive may seem arbitrary. However, stroll-nonreptitive colourings play a central role.  First, they appear in the characterisation of walk-nonrepetitive colourings (\cref{SigmaCharacterisation}). Second, several results for path-nonrepetitive colourings can be strengthened for stroll-nonrepetitive colourings, and this strengthening is sometimes needed in the proof. This includes the breakthrough result for planar graphs (\cref{PlanarRho} using \cref{ProductRho}). Despite the importance of stroll-nonreptitive colourings, the following fundamental questions remain unsolved. 

\begin{open}
\label{PiRho}
Is there a function $f$ such that $\rho(G)\leq f(\pi(G))$ for every graph $G$? It is possible that $\pi(G)=\rho(G)$ for every graph $G$. 
\end{open}

\subsection{List Colourings}

Some results about nonrepetitive colouring hold in the stronger setting of list colourings, which we now introduce. Let $G$ be a graph. A \emph{list-assignment} of $G$ is a function $L$ that assigns each vertex $v$ of $G$ a set $L(v)$, whose elements are called \emph{colours}. If $|L(v)|\geq k$ for each vertex $v$ of $G$, then $L$ is a \emph{$k$-list-assignment}. An \emph{$L$-colouring} of $G$ is a function $\phi$ such that $\phi(v)\in L(v)$ for each vertex $v$ of $G$. The \emph{list chromatic number} $\chi_{\textup{ch}}(G)$ (also called \emph{choosability} of $G$) is the minimum integer $k$ such that $G$ has a proper $L$-colouring for every $k$-list-assignment $L$ of $G$. List chromatic number is widely studied in the literature. These notions naturally extend to nonrepetitive colourings. The \emph{(path-)nonrepetitive list chromatic number} $\pich(G)$ is the minimum integer $k$ such that $G$ has a nonrepetitive $L$-colouring for every $k$-list-assignment $L$ of $G$.  

\subsection{Structure}

This survey is structured as follows. \cref{Tools} presents definitions and tools that will be used for the main results that follow. \cref{BoundedDegreeGraphs} contains results about nonrepetitive colourings of graphs with bounded degree. Most of the material here is based on the Lov\'asz Local Lemma and related methods. \cref{TreesTreewidth} begins our study of nonrepetitive colourings of structured graph classes by looking at trees and graphs of bounded treewidth. This study continues in \cref{PlanarAndBeyond}, where we consider planar graphs and other minor-closed classes.  \cref{Subdivisions} considers nonrepetitive colourings of graph subdivisions. This material is important for \cref{BoundedExpansion} which looks at connections between nonrpetitive colourings and graph expansion. 


This survey aims to present most of the main results about nonrepetitive graph colouring. Nevertheless, several relevant areas have been omitted, including game-theoretic generalisations~\citep{GKM11a,GKZ15,Pegden11,GSZ13}, anagram-free colouring~\citep{CJS17,WW18a,KLS18,WW18b,CDM18,Carpi98,Pleasants70,Keranen09}, geometric variants \citep{GKZ16,WW16,DGNPSTWW20}, $k$-power-free colourings~\citep{KM13,BreakingRhythm}, the Thue sequence of a graph \citep{KSN16}, and computational complexity issues \citep{MS09,Manin} (testing whether a given colouring of a graph is nonrepetitive is co-NP-complete, even for 4-colourings \citep{MS09}). 






\section{Tools}
\label{Tools}

\subsection{Definitions}

We use standard graph-theoretic terminology and notation~\citep{Diestel5}. 

Let $\dist_G(u,v)$ be the distance between vertices $u$ and $v$ in a graph $G$. For a vertex $v$ in a graph $G$ and $r\in\mathbb{N}$, let $N^r_G(v)$ be the set of vertices of $G$ at distance exactly $r$ from $v$, and let $N^r_G[v]$ be the set of vertices at distance at most $r$ from $v$. The set $N^r_G[v]$ is called an \emph{$r$-ball}.

The \emph{cartesian product} of graphs $A$ and $B$, denoted by $A\square B$, is the graph with vertex set $V(A)\times V(B)$, where distinct vertices $(v,x),(w,y)\in V(A)\times V(B)$ are adjacent if:
$v=w$ and $xy\in E(B)$; or
$x=y$ and $vw\in E(A)$.
The \emph{direct product} of $A$ and $B$, denoted by $A\times  B$, is the graph with vertex set $V(A)\times V(B)$, where distinct vertices $(v,x),(w,y)\in V(A)\times V(B)$ are adjacent if $vw\in E(A)$ and $xy\in E(B)$.
The \emph{strong product} of $A$ and $B$, denoted by $A\boxtimes B$, is the union of $A\square B$ and $A\times B$. 
If $X$ is a subgraph of some product $A \ast B$, then the \emph{projection} of $X$ into $A$ is the set of vertices $v\in V(A)$ such that $(v,w)\in V(X)$ for some $w\in V(B)$.


A \emph{subdivision} of a graph $G$ is a graph $G'$ obtained from $G$ by replacing each edge $vw$ of $G$ by a path $P_{vw}$ with endpoints $v$ and $w$, where the $P_{vw}$ are pairwise internally disjoint. If each path $P_{vw}$ has exactly $d$ internal vertices, then $G'$ is the \emph{$d$-subdivision} of $G$, denoted by $G^{(d)}$. If each path $P_{vw}$ has at least $d$ internal vertices, then $G'$ is a \emph{$(\geq d)$-subdivision}. If each path $P_{vw}$ has at most $d$ internal vertices, then $G'$ is a \emph{$(\leq d)$-subdivision}.

A graph $H$ is a \emph{minor} of a graph $G$ if a graph isomorphic to $H$ can be obtained from a subgraph of $G$ by contracting edges. A graph class $\GG$ is \emph{minor-closed} if for every graph $G\in\GG$, every minor of $G$ is in $\GG$. A graph $H$ is a \emph{topological minor} of $G$ if some subgraph of $G$ is isomorphic to a subdivision of $H$. 

A \emph{graph parameter} is a function $\lambda$ such that $\lambda(G)$ is a non-negative real number for every graph $G$, and $\lambda(G_1)=\lambda(G_2)$ for all isomorphic graphs $G_1$ and $G_2$. Examples include the chromatic number $\chi$, the nonrepetitive chromatic number $\pi$, etc. If $\lambda$ and $\mu$ are graph parameters, then $\lambda$ is \emph{bounded by} $\mu$ if for some function $f$,  we have $\lambda(G)\leq f(\mu(G))$ for every graph $G$. Parameters $\lambda$ and $\mu$ are \emph{tied}  if $\lambda$ is bounded by $\mu$ and $\mu$ is bounded by $\lambda$.  

\subsection{Naive Upper Bound}

Consider the following naive upper bound, where $\alpha(G)$ is the size of the largest independent set in $G$. 

\begin{lem}
	\label{NaiveUpperBound}
	For every graph $G$,
	$$\st(G) \leq \pi(G) \leq \rho(G) \leq |V(G)| - \alpha(G) + 1.$$
	For every complete multipartite graph $G$,
	$$\st(G) = \pi(G) = \rho(G) = |V(G)| - \alpha(G) + 1.$$
\end{lem}

\begin{proof}
	Let $X$ be an independent set in $G$ with $|X|=\alpha(G)$. Assign each vertex in $V(G)\setminus X$ a unique colour. Assign the vertices in $X$ one further colour. Suppose that there is a repetitively coloured stroll $P=(v_1,\dots,v_{2t})$ in $G$. Since $X$ is an independent set, some vertex $v_i$ is in $V(G)\setminus X$. Without loss of generality, $i\in\{1,\dots,t\}$. Since $v_{t+i}$ is assigned the same colour as $v_i$ and $v_i$ is the only vertex assigned its colour, $v_i=v_{t+i}$. Thus $P$ is not a stroll, and $G$ is stroll-nonrepetitively coloured. Thus $\st(G) \leq \pi(G) \leq \rho(G) \leq |V(G)| - \alpha(G) + 1$. 
	
	Now consider a complete multipartite graph $G$ with colour classes $X_1,\dots,X_k$ with $\alpha(G)=|X_1|\geq\dots\geq |X_k|$. Consider a star colouring of $G$. Distinct sets $X_i$ and $X_j$ are assigned disjoint sets of colours. Say $X_i$ is \emph{rainbow} if $X_i$ is assigned $|X_i|$ colours. If distinct sets $X_i$ and $X_j$ are both not rainbow, then two vertices in $X_i$ are assigned the same colour, and two vertices in $X_j$ are assigned the same colour, implying there is monochromatic 4-vertex path. Thus at least $k-1$ of $X_1,\dots,X_k$ are rainbow. The remaining set is assigned at least one colour. Thus for some $i\in\{1,\dots,k\}$, the total number of colours is at least $1+\sum_{j\neq i}|X_j|\geq 1+ n-|X_1|=n-\alpha(G)+1$. Thus
	$\rho(G) \geq \pi(G) \geq \st(G) \geq n - \alpha(G)+1$. 
\end{proof}

\subsection{Extremal Questions}

This section studies the maximum number of edges in a nonrepetitively coloured graph. \citet{BW08} determined the answer precisely for path-nonrepetitive colouring. 

\begin{prop}[\citep{BW08}]
	\label{ExtremalPi}
	For all integer $n\geq c\geq 1$ the maximum number of edges in an $n$-vertex graph $G$ with $\pi(G)\leq c$ equals $(c-1)n-\binom{c}{2}$.
\end{prop}

\begin{proof}
	Say $G$ is an $n$-vertex graph with $\pi(G)\leq c$. Fix a path-nonrepetitive $c$-colouring of $G$. Say there are $n_i$ vertices in the $i$-th colour class. Every cycle receives at least three colours. Thus the subgraph induced by the vertices coloured $i$ and $j$ is a forest, and has at most $n_i+n_j-1$ edges. Hence the number of edges in $G$ is at most 
	\begin{equation*}
	\sum_{1\leq i<j\leq c}(n_i+n_j-1)=\sum_{1\leq i\leq c}(c-1)n_i-\binom{c}{2}
	=(c-1)n-\binom{c}{2}.
	\end{equation*}
	This bound is attained by the graph consisting of a complete graph $K_{c-1}$ completely joined to an independent set of $n-(c-1)$ vertices, which obviously has a path-nonrepetitive $c$-colouring.
\end{proof}

The same answer applies for stroll-nonrepetitive colourings. 

\begin{prop}
	\label{ExtremalRho}
	For all integer $n\geq c\geq 1$ the maximum number of edges in an $n$-vertex graph $G$ with $\rho(G)\leq c$ equals $(c-1)n-\binom{c}{2}$.
\end{prop}

\begin{proof}
	If $\rho(G)\leq c$ then $\pi(G)\leq c$, implying $E(G)| \leq (c-1)|V(G)| -\binom{c}{2}$ by \cref{ExtremalPi}. This bound is tight since the example given in the proof of \cref{ExtremalPi} obviously has a stroll-nonrepetitive $c$-colouring. 
\end{proof}

%
%
%

%

Now consider the maximum number of edges in a walk-nonrepetitive coloured graph. First note that the example in the proof of \cref{ExtremalPi} is walk-repetitive. Since $\sigma(G)\geq\Delta(G)+1$ and $|E(G)|\leq\half\Delta(G)|V(G)|$, we have the trivial upper bound, 
$$|E(G)|\leq\half(\sigma(G)-1)|V(G)|.$$
This bound is tight for $\sigma=2$ (matchings) and $\sigma=3$ (cycles), but is not known to be tight for $\sigma\geq 4$. 

We have the following lower bound. 

\begin{prop}[\citep{BW08}]
	\label{Example}
	For all $p\geq1$, there are infinitely many graphs $G$ with $\sigma(G)\leq 4p$ and $$|E(G)|\geq\tfrac18 (3\sigma(G)-4)|V(G)|- \tfrac19 \sigma(G)^2.$$
\end{prop}


\begin{proof}
	Let $G:=P_n\boxtimes K_\ell$. 
	By \cref{ProductSigma,PathSigma}, $\sigma(G)\leq 4\ell$. 
	Note that $|E(G)|=\half(3\ell-1)|V(G)|-\ell^2$. 
	As a lower bound, $\sigma(G)\geq\Delta(G)+1=3\ell$. 
	Thus $|E(G)|\geq\half(3\sigma(G)/4-1)|V(G)|-(\sigma(G)/3)^2$.
\end{proof}

\begin{open}
	What is the maximum number of edges in an $n$-vertex graph $G$ with $\sigma(G)\leq c$?
\end{open}

\subsection{Walk-nonrepetitive Colourings}

The following result (implicit in \citep{BW08} and explicit in \citep{Aprile14}) characterises walk-nonrepetitive colourings. It provides our first example of the value of considering stroll-nonrepetitive colourings. Let $G^2$ be the \emph{square} graph of $G$. That is, $V(G^2)=V(G)$, and $vw\in E(G^2)$ if and only if the distance between $v$ and $w$ in $G$ is at most $2$. A proper colouring of $G^2$ is called a \emph{distance-$2$} colouring of $G$.

\begin{lem}[\citep{BW08}]
\label{SigmaCharacterisation}
A colouring of a graph is walk-nonrepetitive if and only if it is stroll-nonrepetitive and distance-2. 
\end{lem}

\begin{proof}
It follows from the definition that every walk-nonrepetitive colouring is stroll-nonrepetitive. Consider a walk-nonrepetitive colouring of a graph $G$. Adjacent vertices $v$ and $w$ receive distinct colours, as otherwise $v,w$ would be a repetitively coloured path. If $u,v,w$ is a path, and $u$ and $w$ receive the same colour, then the non-boring walk $u,v,w,v$ is repetitively coloured. Thus vertices at distance at most $2$ receive distinct colours. 

Now we prove the converse. Let $c$ be a stroll-nonrepetitive distance-2 colouring of $G$. Suppose for the sake of contradiction that $G$ contains a non-boring repetitively coloured walk $W=(v_1,\dots,v_{2t})$. Since $c$ is stroll-nonrepetitive, $v_i = v_{t+i}$ for some $i \in \{1,\dots,t\}$. Since $W$ is not boring, $v_j \neq  v_{t+j}$ for some $j \in \{1,\dots,t\}$. Choose such $i$ and $j$ to minimise $|i-j|$. Then $|j-i|=1$. Thus $v_i\in N[v_j] \cap N[v_{t+j}]$ and $\dist_G(v_j,v_{t+j})\leq 2$. Hence $v_j$ and $v_{t+j}$ are assigned distinct colours, and $W$ is not repetitively coloured. This contradiction shows that $G$ contains no non-boring repetitively coloured walk. That is, $c$ is walk-nonrepetitive. 
\end{proof}

\cref{SigmaCharacterisation} implies the following bounds on $\sigma(G)$.

\begin{cor}
\label{SigmaBounds}
For every graph $G$, 
$$\max\{\rho(G),\Delta(G)+1\} 
\leq \max\{\rho(G),\chi(G^2) \} 
\leq \sigma(G) 
\leq \rho(G)\, \chi(G^2)
\leq \rho(G)\,(\Delta(G)^2+1).$$
\end{cor}

\begin{proof}
The lower bounds on $\sigma(G)$ follow directly from  \cref{SigmaCharacterisation} and since $G^2$ has a clique on $\Delta(G)+1$ vertices. The upper bound
$\sigma(G) \leq \rho(G)\, \chi(G^2)$ is proved by considering the product of a stroll-nonrepetitive colouring and a distance-2 colouring. The final upper bound follows since $\chi(G^2)\leq\Delta(G^2)+1\leq \Delta(G)^2+1$.
\end{proof}

A graph $G$ is \emph{$d$-degenerate} if every subgraph of $G$ has minimum degree at most $d$. A greedy algorithm shows that every $d$-degenerate graph is $(d+1)$-colourable. For a $d$-degenerate graph $G$ with maximum degree $\Delta$, the square $G^2$ is $d\Delta$-degenerate and $(d\Delta+1)$-colourable. Thus \cref{SigmaBounds} implies:

\begin{cor}
\label{SigmaPiDegen}
For every $d$-degenerate graph $G$, 
$$\sigma(G) \leq \rho(G)\,(d\,\Delta(G)+1).$$
\end{cor}

It is not obvious that there is a finite algorithm to test if a given colouring of a graph is walk-nonrepetitive. However, the following lemma by \citet{BW08} implies that to test if a colouring of an $n$-vertex  graph is walk-nonrepetitive, one need only test whether walks of length at most $2n^2$ are nonrepetitive. A similar result for edge-colourings was previously proved by \citet{BV08}. 

\begin{prop}[\citep{BW08}]
	\label{SmallWalks}
	Suppose that in some coloured graph, there is a repetitively coloured non-boring walk. Then there is a repetitively coloured non-boring walk of order $k$ and length at most $2k^2$.
\end{prop}

\begin{proof}
Let $k$ be the minimum order of a repetitively coloured non-boring walk. Let $W=(v_1,v_2,\dots,v_{2t})$ be a repetitively coloured non-boring walk of order $k$ and with $t$ minimum. If $2t\leq2k^2$, then we are done. Now assume that $t>k^2$. By the pigeonhole principle, there is a vertex $x$ that appears at least $k+1$ times in $v_1,v_2,\dots,v_t$. Thus there is a vertex $y$ that appears at least twice in the set $\{v_{t+i}:v_i=x,i\in[t]\}$. As illustrated in \cref{Shorten}, $W=AxBxCA'yB'yC'$ for some walks $A,B,C,A',B',C'$ with $|A|=|A'|$, $|B|=|B'|$, and $|C|=|C'|$. Consider the walk $U:=AxCA'yC'$. If $U$ is not boring, then it is a repetitively coloured non-boring walk of order at most $k$ and length less than $2t$, which contradicts the minimality of $W$. Otherwise $U$ is boring, implying $x=y$, $A=A'$, and $C=C'$. Thus $B\ne B'$ since $W$ is not boring, implying $xBxB'$ is a repetitively coloured non-boring walk of order at most $k$ and length less than $2t$, which contradicts the minimality of $W$.
\end{proof}

\begin{figure}[!ht]
	\centering
	\includegraphics{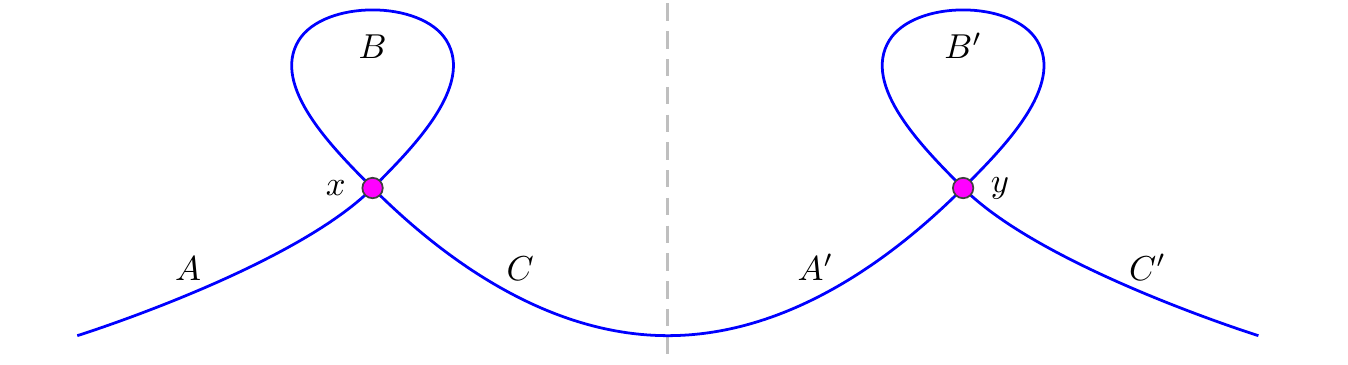}
	\caption{Illustration for the proof of \cref{SmallWalks}.}
	\label{Shorten} 
\end{figure}

\subsection{Lazy Considerations}

Many results that follow depend on the following definitions and lemmas. Two vertices in a graph are said to \emph{touch} if they are adjacent or equal. The following definition is commonly used in the theory of random walks. A \emph{lazy walk} in a graph $G$ is a sequence $(v_1,\dots,v_t)$ of vertices in $G$ such that $v_i$ and $v_{i+1}$ touch for each $i\in\{1,\dots,t-1\}$. Equivalently, a lazy walk in $G$ is a walk in the pseudograph obtained from $G$ by adding a loop at each vertex. Lazy walks were introduced in the context of nonrepetitive colourings by \citet{DEJWW20}, although the idea was implcit in a lemma of \citet{KP-DM08} about walk-nonrepetitive colourings of paths. 

\begin{lem}
\label{LazyWalks}
Every walk-nonrepetitive colouring is nonrepetitive on  non-boring lazy walks. 
\end{lem}

\begin{proof}
Let $c$ be  a walk-nonrepetitive colouring of a graph $G$. Suppose that $G$ contains a repetitively coloured non-boring lazy walk. Choose such a walk $W=(v_1,\dots,v_{2t})$ with minimum length $2t$. Since no non-lazy non-boring walk is repetitively coloured, by symmetry, $v_i=v_{i+1}$ for some $i\in\{1,\dots,t\}$. 

First suppose that $i=t$. Let $W'$ be the walk $(v_1,\dots,v_{t-1},v_{t+1},\dots,v_{2t-1}\}$. Then $W'$ is a repetitively coloured lazy walk of length $2t-2$. If $W'$ is not boring, then $W'$ contradicts the choice of $W$. So $W'$ is boring. In particular, $v_1=v_{t+1}=v_t$ and $v_{t-1}=v_{2t-1}$. Since $W$ is not boring, $v_t \neq v_{2t}$. Thus $(v_t,v_{t-1},v_{2t},v_{2t-1})$ is a non-boring repetitively coloured walk, which is a contradiction. 

Now assume that $i\in\{1,\dots,t-1\}$. Since $W$ is repetitively coloured, $c(v_{t+i})=c(v_i)$ and  $c(v_{t+i+1})=c(v_{i+1})$, implying $c(v_{t+i})=c(v_{t+i+1})$ since $v_i=v_{i+1}$. If $v_{t+i}\neq v_{t+i+1}$ then $(v_{t+i},v_{t+i+1})$ is a repetitively coloured non-boring non-lazy walk, which is a contradiction. So $v_{t+i}= v_{t+i+1}$. Let $W''$ be the walk 
$(v_1,\dots,v_{i},v_{i+2},\dots,v_{t+i},v_{t+i+2},\dots,v_{2t-1})$. 
Then $W''$ is a repetitively coloured lazy walk of length $2t-2$. 
If $W''$ is boring, then $v_i=v_{i+1}=v_{t+i}=v_{t+i+1}$, implying that $W$ is boring as well. Thus $W''$ is not boring. Hence $W''$ contradicts the choice of $W$. 
\end{proof}

The following similar definition was implicitly introduced in the context of nonrepetitive colourings by  \citet{DEJWW20}\footnote{\citet{DEJWW20} defined a colouring of a graph to be \emph{strongly nonrepetitive} if for every repetitively coloured lazy walk $(v_1,\dots,v_{2t})$ in $G$, we have $v_i=v_{t+i}$ for some $i\in\{1,\dots,t\}$. This is equivalent to saying that every lazy stroll is nonrepetitively coloured. They defined $\pi^*(G)$ to be the minimum number of colours in a strongly nonrepetitive colouring of a graph $G$. By \cref{LazyStroll}, $\pi^*(G)=\rho(G)$.}. A \emph{lazy stroll} in a graph $G$ is a lazy walk $(v_1,\dots,v_{2t})$ in $G$ such that $v_i\neq v_{t+i}$ for each $i\in\{1,\dots,t-1\}$.

\begin{lem}
\label{LazyStroll}
Every stroll-nonrepetitive colouring $c$ of a graph $G$ is nonrepetitive on lazy strolls. 
\end{lem}

\begin{proof}
Suppose that there is a repetitively coloured lazy stroll in $G$. Choose such a stroll $W=(v_1,\dots,v_{2t})$ with minimum length $2t$. Since no (non-lazy) stroll is repetitively coloured, without loss of generality, $v_i=v_{i+1}$ for some $i\in\{1,\dots,t\}$. 

First suppose that $i=t$. Then $(v_1,\dots,v_{t-1},v_{t+1},\dots,v_{2t-1}\}$ is a repetitively coloured lazy stroll of length $2t-2$, which contradicts the choice of $W$. 

Now assume that $i\in\{1,\dots,t-1\}$. Since $W$ is repetitively coloured, $c(v_{t+i})=c(v_i)$ and  $c(v_{t+i+1})=c(v_{i+1})$, implying $c(v_{t+i})=c(v_{t+i+1})$ since $v_i=v_{i+1}$. If $v_{t+i}\neq v_{t+i+1}$ then $(v_{t+i},v_{t+i+1})$ is a repetitively coloured (non-lazy) stroll, which is a contradiction. So $v_{t+i}= v_{t+i+1}$. Then 
$(v_1,\dots,v_{i},v_{i+2},\dots,v_{t+i},v_{t+i+2},\dots,v_{2t-1})$ is a repetitively coloured lazy stroll of length $2t-2$, which contradicts the choice of $W$. 
\end{proof}

Finally, we have a similar definition and lemma for lazy paths. A \emph{lazy path} in a graph $G$ is a lazy walk $(v_1,\dots,v_t)$ such that if $v_i=v_j$ and $1\leq i<j\leq t$ then $v_i=v_{i+1}=\dots=v_j$, and $v_1\neq v_t$. The last condition says that at least two distinct vertices occur in a lazy path, which is essential for the next lemma to hold. 

\begin{lem}
\label{LazyPath}
Every path-nonrepetitive colouring $c$ of a graph $G$ is nonrepetitive on lazy paths. 
\end{lem}

\begin{proof}
Suppose that there is a repetitively coloured lazy path in $G$. Choose such a lazy path $P=(v_1,\dots,v_{2t})$ with the minimum number of vertices. Since no (non-lazy) path is repetitively coloured, without loss of generality, $v_i=v_{i+1}$ for some $i\in\{1,\dots,t\}$. 

First suppose that $i=t$. Let  $P':=(v_1,\dots,v_{t-1},v_{t+1},\dots,v_{2t-1}\}$.
We claim that $P'$ is a lazy path. 
This is the case unless $v_1=v_{2t-1}$, so assume that $v_1=v_{2t-1}$. 
Since $P$ is a lazy path, $v_1=v_2=\dots=v_{2t-1}$ and $v_{2t-1}\neq v_{2t}$. Since $c(v_t)=c(v_{2t})$ and $v_t=v_{2t-1}$, we have $c(v_{2t-1})=c(v_{2t})$. Thus $(v_{2t-1},v_{2t})$ is a repetitively coloured (non-lazy) path, which is a contradiction. Thus $P'$ is a lazy path with $2t-2$ vertices, which contradicts the choice of $P$. 

Now assume that  $i\in\{1,2,\dots,t-1\}$. 
Since $P$ is repetitively coloured, $c(v_{t+i})=c(v_i)$ and $c(v_{t+i+1})=c(v_{i+1})$, implying $c(v_{t+i})=c(v_{t+i+1})$ since $v_i=v_{i+1}$. If $v_{t+i}\neq v_{t+i+1}$ then $(v_{t+i},v_{t+i+1})$ is a repetitively coloured (non-lazy) path, which is a contradiction. So $v_{t+i}= v_{t+i+1}$. Let 
$P':=(v_1,\dots,v_{i},v_{i+2},\dots,v_{t+i},v_{t+i+2},\dots,v_{2t-1})$.
We claim that $P'$ is a lazy path. 
This is the case unless $v_1=v_{2t-1}$, so assume that $v_1=v_{2t-1}$. 
Since $P$ is a lazy path, $v_1=v_2=\dots=v_{2t-1}$ and $v_{2t-1}\neq v_{2t}$. Since $c(v_t)=c(v_{2t})$ and $v_t=v_{2t-1}$, we have $c(v_{2t-1})=c(v_{2t})$. Thus $(v_{2t-1},v_{2t})$ is a repetitively coloured (non-lazy) path, which is a contradiction. Thus $P'$ is a lazy path with $2t-2$ vertices, which contradicts the choice of $P$. 
%
%
\end{proof}

\subsection{Shadow-Complete Layerings}

This section presents results about shadow-complete layerings. This  tool was first introduced in the context of nonrepetitive colourings by \citet{KP-DM08}. It will be used to obtain results for trees (\cref{Trees}), graphs of bounded treewidth (\cref{Treewidth}), and graphs excluding a fixed minor or topological minor (\cref{MinorClosedClass}).

A \emph{layering} of a graph $G$ is a partition $(V_0,V_1,\dots)$ of $V(G)$ such that for every edge $vw\in E(G)$, if $v\in V_i$ and $w\in V_j$, then $|i-j| \leq 1$. Vertices in $V_i$ are said to be at \emph{depth} $i$. For example, if $r$ is a vertex in a connected graph $G$ and $V_i$ is the set of vertices at distance exactly $i$ from $r$ in $G$ for all $i\geq 0$, then  layering $(V_0,V_1,\dots)$ is a layering of $G$, called a \emph{BFS layering} of $G$.

Consider a layering $(V_0,V_1,\dots)$ of a graph $G$. Let $H$ be a connected component of $G[V_i\cup V_{i+1}\cup \cdots]$, for some $i\ge1$. The \emph{shadow} of  $H$ is the set of vertices in $V_{i-1}$ adjacent to  some vertex in $H$. The layering is \emph{shadow-complete} if every shadow is a clique. This concept was introduced by \citet{KP-DM08}, who showed the utility of  shadow-complete layerings for nonrepetitive colourings by the next lemma. 

\begin{lem}[\citep{KP-DM08}]
\label{ShadowCompletePi}
If a graph $G$ has a shadow-complete layering  $(V_0,V_1,\dots,V_n)$, then 
\begin{align*}
\pi(G) \leq 4\max_i \pi( G[V_i] ).
\end{align*}
\end{lem}

\begin{proof}
Let $c:= \max_i \pi( G[V_i] )$. 
Let $\beta_i$ be a nonrepetitive $c$-colouring of $G[V_i]$ for each $i\in\{1,\dots,n\}$. 
By \cref{PathSigma} there is a walk-nonrepetitive 4-colouring $\alpha$ of the path $P=(x_1,\dots,x_n)$. 
Colour each vertex $v$ in $V_i$ by the pair $\phi(v):=(\alpha(x_i),\beta_i(v))$. 
Suppose for the sake of contradiction that $G$ contains a repetitively coloured path $W=(v_1,\ldots,v_{2k})$. 
Let $d$ be the minimum depth of a vertex in $W$.
Let $W'$ be the sequence of vertices obtained from $W$ by removing all vertices at depth greater than $d$. 
The projection of $W$ on $P$ is an $\alpha$-repetitive lazy walk in $P$, and is thus boring  by \cref{LazyWalks}. Thus the vertices $v_j$ and $v_{j+k}$ of $W$ have the same depth for every $j \in \{1, \dots, k\}$. In particular, $v_j$ is in $W'$ if and only if $v_{j+k}$ is. Hence, there are indices $1 \leq i_1 < i_2 < \cdots < i_{\ell} \leq k$ such that $W'=(v_{i_1}, v_{i_2}, \dots, v_{i_{\ell}}, v_{i_1+k}, v_{i_2+k}, \dots, v_{i_{\ell}+k})$. 
For each pair of consecutive vertices $v_a$ and $v_b$ in $W'$, the vertices strictly between $v_a$ and $v_b$ in $W$ are in a single connected component of the graph induced by the vertices of depth greater than $d$. By shadow-completeness, $v_a$ and $v_b$ are adjacent. Hence $W'$ is a path in $G[V_d]$. Since $W$ is $\phi$-repetitive, for each $j\in\{1,\dots,\ell\}$  we have $\phi(v_{i_j})=\phi(v_{i_j+k})$, implying $\beta_d(v_{i_j})=\beta_d(v_{i_j+k})$. Hence $W'$ is a $\beta_d$-repetitively coloured path in $G[V_d]$, which is the desired contradiction. 
\end{proof}

\citet{DEJWW20} implicitly proved an analogous result for $\rho$.

\begin{lem}[\citep{DEJWW20}]
\label{ShadowCompleteRho}
If a graph $G$ has a shadow-complete layering  $(V_0,V_1,\dots,V_n)$, then 
\begin{align*}
\rho(G) & \leq 4\max_i \rho( G[V_i] ) 
\end{align*}
\end{lem}

\begin{proof}
Let $c:= \max_i \pi( G[V_i] )$. 
Let $\beta_i$ be a nonrepetitive $c$-colouring of $G[V_i]$.
By \cref{PathSigma} there is a walk-nonrepetitive 4-colouring $\alpha$ of the path $(x_1,\dots,x_n)$. 
Colour each vertex $v$ in $V_i$ by the pair $\phi(v):=(\alpha(x_i),\beta_i(v))$. 

We now prove that $\phi$ is path-nonrepetitive. 
Let $W$ be a $\phi$-repetitive walk $v_1,\ldots,v_{2k}$.
Our goal is to prove that $v_j=v_{j+k}$ for some $j\in\{1,\ldots,k\}$. 
Let $d$ be the minimum depth of a vertex in $W$.
Let $W'$ be the sequence of vertices obtained from $W$ by removing all vertices at depth greater than $d$. 
We claim that $W'$ is a lazy walk. To see this, consider vertices $v_i,v_{i+1}, \ldots, v_{i+t}$ of $W$ such that $v_i$ and $v_{i+t}$ have depth
$d$ but $v_{i+1}, \ldots, v_{i+t-1}$ all have depth greater than $d$; thus, $v_{i+1}, \ldots, v_{i+t-1}$ were removed when constructing $W'$.
Then, the vertices $v_{i+1}, \ldots, v_{i+t-1}$ lie
in a connected component of the graph induced by the vertices at depth
greater than $d$. Since the layering is shadow-complete, $v_i$
and $v_{i+t}$ are adjacent or equal. This shows that $W'$ is a lazy walk in $G[V_d]$.

The projection of $W$ into $P$ is an $\alpha$-repetitive lazy walk in $P$, and is thus boring by \cref{LazyWalks}. Thus the vertices $v_j$ and $v_{j+k}$ of $W$ have the same depth for every $j \in \{1, \dots, k\}$. In particular, $v_j$ was removed from $W'$ if and only if $v_{j+k}$ was. Hence, there are indices $1 \leq i_1 < i_2 < \cdots < i_{\ell} \leq k$ such that $W'=v_{i_1}, v_{i_2}, \dots, v_{i_{\ell}}, v_{i_1+k}, v_{i_2+k}, \dots, v_{i_{\ell}+k}$. Since $W$ is $\phi$-repetitive, it follows that $W'$ is also $\phi$-repetitive and in particular $W'$ is $\beta_d$-repetitive. Hence there is an index $i_r$ such that $v_{i_r}=v_{i_r+k}$, which completes the proof.
\end{proof}

\citet{BW08} proved an analogous result for  $\sigma$, which we refine as follows. 

\begin{lem}
\label{ShadowCompleteWalk}
Let $G$ be a graph that has a shadow-complete layering $(V_0,V_1,\dots,V_n)$. 
Assume that $G$ has a $k$-colouring $\beta$ in which $G[V_i]$ is stroll-nonrepetitively coloured for each $i\in\{0,1,\dots,n\}$, and 
distinct vertices $v,w$ of $G$ are assigned distinct colours whenever $v,w\in V_i$ for some $i\in\{1,\dots,n\}$ and $v,w\in N(u)$ for some vertex $u\in V_{i-1}\cup V_i$. Then 
\begin{align*}
\sigma(H) & \leq 4k.
\end{align*}
\end{lem}

\begin{proof}
By \cref{PathSigma} there is a walk-nonrepetitive 4-colouring $\alpha$ of the path $P=(x_0,x_1,\dots,x_n)$. 
Colour each vertex $v$ in $V(G)\cap V_i$ by the pair  $\phi(v):=(\alpha(x_i),\beta(v))$. 
We claim that $\phi$ is a walk-nonrepetitive colouring of $G$.

Suppose on the contrary that $W=(v_1,\dots,v_{2t})$ is a $\phi$-repetitive non-boring walk in $G$. The projection of $W$ to $P$ is a lazy walk, which is repetitively coloured by $\alpha$, and is therefore boring by \cref{LazyWalks}. Thus, for $i\in\{1,\dots,t\}$ the vertices $v_i$ and $v_{t+i}$ are in the same layer. 

Let $k$ be the minimum layer containing a vertex in $W$. 
Let $W'$ be the sequence of vertices obtained from $W$ by deleting all vertices not in $V_k$. Since $v_i\in W'$ if and only if $v_{t+i}\in W'$, 
the sequence $W'$ is repetitively coloured. Let $v_i$ and $v_j$ be consecutive vertices in $W'$ with $i<j$. Then there is walk from $v_i$ to $v_j$ with all its internal vertices in $V_{k+1}\cup\dots\cup V_n$ (since $k$ was chosen minimum), implying $v_i=v_j$ or $v_iv_j$ is an edge of $H$ (since the layering is shadow-complete). Thus $W'$ forms a repetitively coloured lazy walk in $G[V_k]$. Since $G[V_i]$ is stroll-repetitively coloured by $\beta$, by \cref{LazyStroll}, some vertex $v_i=v_{t+i}$ is in $W'$. Since $W$ is not boring, $v_j\ne v_{t+j}$ for some $j\in[t]$. Choose such $i$ and $j$ to minimise $|i-j|$. Thus $|i-j|=1$. Hence $v_j,v_{t+j}\in V_k$ or $v_j,v_{t+j}\in V_{k+1}$. Moreover, $v_j$ and $v_{t+j}$ have a common neighbour $v_i=v_{t+i}$. By assumption, $\beta(v_j)\neq \beta(v_{t+j})$, which contradicts the assumption that $W$ is repetitively coloured.
\end{proof}

\subsection{Strong Products}
\label{StrongProducts}

Nonrepetitive colourings of graph products have been studied in \citep{KP-DM08,KPZ14,BW08,PSFT18,DEJWW20,BK-AC04}. Here we focus on strong products because doing so has applications to numerous graph classes, such as planar graphs (\cref{Planar}) and graphs excluding a minor (\cref{MinorClosedClass}). 

\begin{lem}[\citep{DEJWW20}]
\label{ProductRho}
For all graphs $G$ and $H$, 
$$\pi(G\boxtimes H) \leq \rho(G \boxtimes H)\le \rho(G) \cdot \sigma(H).$$ 
\end{lem}

\begin{proof}
Let $\alpha$ be a stroll-nonrepetitive colouring of $G$ with $\rho(G)$ colours.  Let $\beta$ be a walk-nonrepetitive colouring of $H$ with $\sigma(H)$ colours. By \cref{LazyWalks}, $\beta$ is nonrepetitive on non-boring lazy walks in $H$. For any two vertices $u \in V(G)$ and $v \in V(H)$,  colour vertex $(u,v)$ of $G\boxtimes H$ by $\phi(u,v) :=(\alpha(u),\beta(v))$. We claim that $\phi$  is a stroll-nonrepetitive colouring of $G\boxtimes H$. To see this, consider a $\phi$-repetitive lazy walk $W=(u_1,v_1), \ldots, (u_{2k},v_{2k})$ in $G\boxtimes H$. By the definition of the strong product and the definition of $\phi$, the projection
$W_G=(u_1,u_2,\ldots,u_{2k})$ of $W$ into $G$ is an $\alpha$-repetitive lazy walk in $G$ and the projection $W_H=(v_1,v_2,\ldots,v_{2k})$ of $W$ into $H$ is a $\beta$-repetitive lazy walk in $H$. Since $\alpha$ is stroll-nonrepetitive, by \cref{LazyStroll}, $u_i=u_{i+k}$ for some $i\in\{1,\dots,k\}$. Since $\beta$ is
nonrepetitive on non-boring lazy walks, $v_j=v_{j+k}$ for every $j \in \{1,\dots,k\}$. In particular, $v_i=v_{i+k}$ and $(u_i,v_i)=(u_{i+k},v_{i+k})$. This shows that $\phi$ is a stroll-nonrepetitive colouring with at most $\rho(G) \cdot \sigma(H)$ colours. 
\end{proof}

Several notes about \cref{ProductRho} are in order:

\begin{itemize}
	\item There is no known upper bound on $\pi(G\boxtimes H)$ that avoids stroll-nonrepetitive colouring. This is an important reason for considering strolls. 
	
	\item $\rho(G \boxtimes H)$ is not bounded by any function of $\rho(G)$ or $\rho(H)$. For example, if $G=H=K_{1,n}$ then $\rho(G)=\rho(H)=2$, but $G\boxtimes H$ contains the complete bipartite graph $K_{n,n}$, and thus $\rho(G\boxtimes H)\geq \pi(G\boxtimes H)\geq \pi(K_{n,n}) \geq n+1$ by \cref{NaiveUpperBound}.


\item As pointed out by Kevin Hendrey [personal communication, 2020], dependence on $\Delta$ and $\sigma$ is unavoidable in \cref{ProductRho}. Since the complete bipartite graph $K_{\Delta(G),\Delta(H)}$ is a subgraph of $G \boxtimes H$, \cref{NaiveUpperBound} implies:
$$\rho(G\boxtimes H) \geq \pi(G \boxtimes H) \geq  \pi( K_{\Delta(G),\Delta(H)} ) \geq \min\{\Delta(G),\Delta(H)\}+1.$$
In particular, $\rho(H\boxtimes H) \geq \Delta(H)+1$ and $\rho(H\boxtimes H) \geq \rho(H)$, implying
$\rho(H\boxtimes H)^3 \geq \rho(H) ( \Delta(H)^2+1) \geq \sigma(H)$ (by \cref{SigmaBounds}) and $\rho(H\boxtimes H) \geq \sigma(H)^{1/3}$.
\end{itemize}

%
%

%

%
%
%

Since $\sigma(K_\ell)=\ell$, \cref{ProductRho} implies:

\begin{cor}[\citep{DEJWW20}]
\label{CompleteProductRho}
For every graph $G$ and integer $\ell\in\mathbb{N}$, 
$$\rho(G \boxtimes K_\ell)\leq \ell\, \rho(G).$$
\end{cor}

\citet{BV07} proved an analogous result for walk-nonrepetitive colourings of strong products.

\begin{lem}[\citep{BV07}]
\label{ProductSigma}
For all graphs $G$ and $H$, 
$$\sigma(G \boxtimes H)\le \sigma(G) \cdot \sigma(H).$$ 
\end{lem}

\begin{proof}
Let $\alpha$ be a walk-nonrepetitive colouring of $G$ with $\sigma(G)$ colours. Let $\beta$ be a walk-nonrepetitive colouring of $H$ with $\sigma(H)$ colours. By \cref{LazyWalks}, $\alpha$ is nonrepetitive on non-boring lazy walks in $G$, and $\beta$ is nonrepetitive on non-boring lazy walks in $H$. 
For any two vertices $u \in V(G)$ and $v \in V(H)$,  colour vertex $(u,v)$ of $G\boxtimes H$ by $\phi(u,v) :=(\alpha(u),\beta(v))$.
We claim that $\phi$ is a walk-nonrepetitive colouring of
$G\boxtimes H$. To see this, consider a $\phi$-repetitive walk $W=(u_1,v_1), \ldots, (u_{2k},v_{2k})$ in $G\boxtimes H$. By the definition of the strong product and the definition of $\phi$, the projection $W_G=(u_1,u_2,\ldots,u_{2k})$ of $W$ into $G$ is an $\alpha$-repetitive lazy walk in $G$ and the projection $W_H=(v_1,v_2,\ldots,v_{2k})$ of $W$ in $H$ is a $\beta$-repetitive lazy walk in $H$. Since $\alpha$ is nonrepetitive on non-boring lazy walks, $u_i=u_{i+k}$ for all $i\in\{1,\dots,k\}$. Similarly,
since $\beta$ is nonrepetitive on non-boring lazy walks, $v_i=v_{i+k}$ for all $i\in\{1,\dots,k\}$. Thus $(u_i,v_i)=(u_{i+k},v_{i+k})$ for all $i\in\{1,\dots,k\}$, implying $W$ is boring. Therefore $\phi$ is a  walk-nonrepetitive colouring of $G\boxtimes H$ with 
at most $\sigma(G)\cdot \sigma(H)$ colours. Hence 
$\sigma(G \boxtimes H)\le \sigma(G) \cdot \sigma(H)$. 
\end{proof}

The above results about strong products are applied for several graph classes in \cref{PlanarAndBeyond}. Here we give one more application. A graph class $\GG$ has \emph{polynomial growth} if for some constant $c$, for every graph $G\in\GG$, for each $r\geq 2$ every $r$-ball in $G$ has at most $r^c$ vertices. For example, every $r$-ball in an $n\times n$ grid graph is contained in a $(2r+1)\times(2r+1)$ subgrid, which has size $(2r+1)^2$; therefore the class of grid graphs has polynomial growth. More generally, let $\mathbb{Z}^d$ be the strong product of $d$ infinite two-way paths. That is, $V(\mathbb{Z}^d)=\{(x_1,\dots,x_d): x_1,\dots,x_d\in\mathbb{Z}\}$ where distinct vertices $(x_1,\dots,x_d)$ and $(y_1,\dots,y_d)$ are adjacent in $\mathbb{Z}^d$ if and only if $|x_i-y_i|\leq 1$ for each $i\in\{1,\dots,d\}$. Then every $r$-ball in $\mathbb{Z}^d$ has size at most $(2r+1)^d$. \citet{KL07} characterised the graph classes with polynomial growth as the subgraphs of $\mathbb{Z}^d$; see \citep{DHJLW20} for an alternative characterisation.

\begin{thm}[\citep{KL07}]
	\label{KL}
	Let $G$ be a graph such that for some constant $c$ and for every integer $r\geq 2$, every $r$-ball in $G$ has at most $r^c$ vertices. Then $G\subseteq \mathbb{Z}^{O(c\log c)}$. 
\end{thm}

\cref{KL,ProductRho,ProductSigma} imply:

\begin{thm}
	Let $G$ be a graph such that for some $c\in\NN$ and for every integer $r\geq 2$, every $r$-ball in $G$ has at most $r^c$ vertices. Then 
	$$\rho(G)\leq \sigma(G) \leq c^{O(c)}.$$ 
\end{thm}

Our focus has been on strong products. The other two main graph products are also of interest. 

\begin{open}
What can be said about $\pi(G\square H)$ and $\pi(G\times H)$? This is related to \cref{LineGraphCompleteBipartite} since $K_n\square K_n \cong L(K_{n,n})$. 
\end{open}

\section{Bounded Degree Graphs}
\label{BoundedDegreeGraphs}

\subsection{Paths}
\label{Paths}

As mentioned in \cref{Introduction}, \citet{Thue06} proved the following:

\begin{thm}[\citep{Thue06}] 
\label{PathPi}
For every path $P$, 
\begin{equation}
\label{Path}
\pi(P) \leq 3,
\end{equation}
with equality if $|V(P)|\geq 4$. 
\end{thm}

\newcommand{\A}[1]{\textcolor{red}{#1}}
\renewcommand{\AA}[1]{\textcolor{brown}{#1}}
\newcommand{\AAA}[1]{\textcolor{blue}{#1}}

\begin{proof} 
The following construction is due to \citet{Leech57}. Consider the following three blocks:
\begin{align*}
A_0: &\; \A{0\,1\,2\,1\,0\,2\,1\,2\,0\,1\,2\,1\,0} \\
A_1: &\; \AA{1\,2\,0\,2\,1\,0\,2\,0\,1\,2\,0\,2\,1} \\
A_2: &\; \AAA{2\,0\,1\,0\,2\,1\,0\,1\,2\,0\,1\,0\,2}.
\end{align*}
First observe that $A_0,A_1,A_2$ are symmetric in the sense that a cyclic permutation of $0,1,2$ also permutes $A_0,A_1,A_2$. Say $W$ is a nonrepetitive word on $\{0,1,2\}$. Let $W'$ be obtained from $W$ by replacing each element $i$ in $W$ by $A_i$. We now prove that $W'$ is also nonrepetitive. Suppose on the contary that $W'$ contains a repetition $x_1,\dots,x_{2t}$. First suppose that $t\leq 7$. Then $x_1,\dots,x_{2t}$ is contained in two consecutive blocks $A_i\,A_j$, and $i\neq j$ since $W$ is nonrepetitive. By symmetry, we may assume that $i=0$. But $A_0\,A_j$ is nonrepetitive:
\begin{align*}
A_0A_1: &\; \A{0\,1\,2\,1\,0\,2\,1\,2\,0\,1\,2\,1\,0}\,\AA{1\,2\,0\,2\,1\,0\,2\,0\,1\,2\,0\,2\,1} \\
A_0A_2: &\; \A{0\,1\,2\,1\,0\,2\,1\,2\,0\,1\,2\,1\,0}\,\AAA{2\,0\,1\,0\,2\,1\,0\,1\,2\,0\,1\,0\,2} .
\end{align*}
Now assume that $t\geq 8$. Any sequence of 8 characters in any block or in any two consecutive blocks is uniquely determined by the block or blocks involved and the starting character. The following cases confirm this, since by symmetry one only needs to check sequences  beginning with 0:
\begin{align*}
A_0: &\; \framebox{\A{0\,1\,2\,1\,0\,2\,1\,2}}\,\A{0\,1\,2\,1\,0} \\
A_0: &\; \A{0\,1\,2\,1}\,\framebox{\A{0\,2\,1\,2\,0\,1\,2\,1}}\,\A{0} \\
A_1: &\; \AA{1\,2}\,\framebox{\AA{0\,2\,1\,0\,2\,0\,1\,2}}\,\AA{0\,2\,1} \\
A_1: &\; \AA{1\,2\,0\,2\,1}\,\framebox{\AA{0\,2\,0\,1\,2\,0\,2\,1}} \\
A_2: &\; \AAA{2}\,\framebox{\AAA{0\,1\,0\,2\,1\,0\,1\,2}}\,\AAA{0\,1\,0\,2}\\
A_2: &\; \AAA{2\,0\,1}\,\framebox{\AAA{0\,2\,1\,0\,1\,2\,0}}\,\AAA{1\,0\,2}\\
A_0A_1: 
&\; \A{0\,1\,2\,1\,0\,2\,1\,2\,0\,1\,2\,1}\,\framebox{\A{0}\,\AA{1\,2\,0\,2\,1\,0\,2}}\,\AA{0\,1\,2\,0\,2\,1} \\
&\; \A{0\,1\,2\,1\,0\,2\,1\,2}\,\framebox{\A{0\,1\,2\,1\,0}\,\AA{1\,2\,0}}\AA{\,2\,1\,0\,2\,0\,1\,2\,0\,2\,1} \\
A_0A_2: 
&\; \A{0\,1\,2\,1\,0\,2\,1\,2\,0\,1\,2\,1}\,\framebox{\A{0}\,\AAA{2\,0\,1\,0\,2\,1\,0}}\,\AAA{1\,2\,0\,1\,0\,2} \\
&\; \A{0\,1\,2\,1\,0\,2\,1\,2}\,\framebox{\A{0\,1\,2\,1\,0}\,\AAA{2\,0\,1}}\,\AAA{0\,2\,1\,0\,1\,2\,0\,1\,0\,2} \\
A_1A_0: 
&\; \AA{1\,2\,0\,2\,1\,0\,2\,0\,1\,2}\,\framebox{\AA{0\,2\,1}\,\A{0\,1\,2\,1\,0}}\,\A{2\,1\,2\,0\,1\,2\,1\,0} \\
&\; \AA{1\,2\,0\,2\,1\,0\,2}\,\framebox{\AA{0\,1\,2\,0\,2\,1}\,\A{0\,1}}\,\A{2\,1\,0\,2\,1\,2\,0\,1\,2\,1\,0} \\
A_1A_2: 
&\; \AA{1\,2\,0\,2\,1\,0\,2\,0\,1\,2}\,\framebox{\AA{0\,2\,1}\,\AAA{2\,0\,1\,0\,2}}\,\AAA{1\,0\,1\,2\,0\,1\,0\,2} \\
&\; \AA{1\,2\,0\,2\,1\,0\,2}\,\framebox{\AA{0\,1\,2\,0\,2\,1}\,\AAA{2\,0}}\,\AAA{1\,0\,2\,1\,0\,1\,2\,0\,1\,0\,2} \\
A_2A_0: 
&\; \AAA{2\,0\,1\,0\,2\,1\,0\,1\,2\,0\,1}\,\framebox{\AAA{0\,2}\,\A{0\,1\,2\,1\,0\,2}}\,\A{1\,2\,0\,1\,2\,1\,0}\\
&\; \AAA{2\,0\,1\,0\,2\,1\,0\,1\,2}\,\framebox{\AAA{0\,1\,0\,2}\,\A{0\,1\,2\,1}}\,\A{0\,2\,1\,2\,0\,1\,2\,1\,0}\\
&\; \AAA{2\,0\,1\,0\,2\,1}\,\framebox{\AAA{0\,1\,2\,0\,1\,0\,2}\,\A{0}}\,\A{1\,2\,1\,0\,2\,1\,2\,0\,1\,2\,1\,0}\\
A_2A_1: 
&\; \AAA{2\,0\,1\,0\,2\,1\,0\,1\,2\,0\,1}\,\framebox{\AAA{0\,2}\,\AA{1\,2\,0\,2\,1\,0}}\,\AA{2\,0\,1\,2\,0\,2\,1}\\
&\; \AAA{2\,0\,1\,0\,2\,1\,0\,1\,2}\,\framebox{\AAA{0\,1\,0\,2}\,\AA{1\,2\,0\,2}}\,\AA{1\,0\,2\,0\,1\,2\,0\,2\,1}\\
&\; \AAA{2\,0\,1\,0\,2\,1}\,\framebox{\AAA{0\,1\,2\,0\,1\,0\,2}\,\AA{1}}\,\AA{2\,0\,2\,1\,0\,2\,0\,1\,2\,0\,2\,1}
\end{align*}
Thus, for $\ell\in\{1,\dots,t-7\}$, the subsequences $x_\ell,x_{\ell+1},\dots,x_{\ell+7}$ and $x_{t+\ell},x_{t+\ell+1},\dots,x_{t+\ell+7}$ appear in copies of the same block $A_i$ or in the same block pair $A_i \,A_j$, and moreover, $x_\ell$ appears in the same position as $x_{t+\ell}$ in the corresponding block. This implies that the starting word $W$ contains a repetitive subsequence. This contradiction shows that $W'$ is nonrepetitive. 
Arbitrarily long paths can be nonrepetitively $3$-coloured by this substitution rule. 
\end{proof}

There is a large body of literature on substitution rules like that used in the above proof; see \citep{Crochemore82,
Currie19,Currie13,ARS09} for example. 

We also briefly mention another proof of \cref{PathPi} via the Thue--Morse sequence, which is the binary sequence
$\underline{0}\,\underline{1}\,\underline{10}\,\underline{1001}\,\underline{10010110}\,\dots$
where the each underlined block is the negation of the entire preceding subsequence. 
See \citep{AS99} for a survey about the Thue--Morse sequence.
%
Given a path $(v_1,v_2,\dots)$, colour each vertex $v_i$ by the difference of the $(i+1)$-th and $i$-th entries in the Thue--Morse sequence. So the sequence of colours is $(1, 0, -1, 1, -1, 0, 1, 0, \dots)$. The Thue--Morse sequence contains no $0X0X0$ or $1X1X1$ pattern. It follows that the above 3-colouring of the path is nonrepetitive.

\begin{open}
\label{PathRhoOpen}
Is $\rho(P)\leq 3$ for every path $P$?
\end{open}

\citet{KP-DM08} showed that paths are walk-nonrepetitively 4-colourable. 

\begin{lem}[\citep{KP-DM08}]
\label{PathSigma}
For every path $P$, 
\begin{equation*}
\sigma(P) \leq 4 .
\end{equation*}
with equality if $|V(P)|\geq 6$. 
\end{lem}

\begin{proof}
Given a nonrepetitive sequence on $\{1,2,3\}$, insert the symbol $4$ between consecutive block of length two. For example, from the sequence $123132123$ we obtain $1243143241243$. Any three consecutive elements are distinct. Thus this sequence corresponds to a distance-2 colouring $\phi$ of a path. We now show that $\phi$ is stroll-nonrepetitieve. Suppose for the sake of contradiction that there is a repetitively coloured stroll. Let $W=(v_1,\dots,v_{2t})$ be a repetitively coloured stroll with $t$ minimum. Then $t\geq 2$. 

First suppose that $W$ is a subpath. Since $\phi(v_i)=4$ if and only if $\phi(v_{t+i})=4$, removing the vertices coloured 4 in $W$ gives a repetition in the original sequence on $\{1,2,3\}$. Now assume that $W$ has a repeated vertex. Thus $v_i=v_{i+2}$ for some $i\in\{1,\dots,2t-2\}$. (The stroll must turn around somewhere.)\ By symmetry, we may assume that $i\in\{1,\dots,t-1\}$. 

Suppose that $i\in\{1,\dots,t-2\}$. Thus $\phi(v_i)=\phi(v_{i+2})=\phi(v_{t+i})=\phi(v_{t+i+2})$. Since $\phi$ is a distance-2 colouring, $v_{t+i}=v_{t+i+2}$. 
Then $(v_1,\dots,v_{i-1},v_{i+2},\dots,v_{t+i-1},v_{t+i+2},\dots,v_{2t})$ is a repetitively coloured stroll on $2t-4$ vertices, contradicting the choice of $W$. 

Thus $i=t-1$. Then $(v_1,\dots,v_{t-2},v_{t+1},\dots,v_{2t-2})$ is a repetitively coloured stroll on $2t-4$ vertices, contradicting the choice of $W$. 

Hence $\phi$ is stroll-nonrepetitive. By \cref{SigmaCharacterisation}, $\phi$ is walk-nonrepetitive. 

Suppose on the contrary that some path $P$ on at least six vertices is walk-nonrepetitive 3-colourable. Since the colouring is distance~2, without loss of generality, the colouring begins $123123$, which is a repetitively coloured path. Thus $\sigma(P)\geq 4$. 
\end{proof}

\cref{ProductRho,PathSigma} imply:

\begin{cor}[\citep{DEJWW20}]
\label{PathProductRho}
For every graph $G$ and every path $P$, 
$$\rho(G \boxtimes P)\le 4 \rho(G).$$
\end{cor}

Now consider nonrepetitive list colourings of paths. \citet{GPZ11} first proved that $\pich(P)\leq 4$. Their proof uses the Lov\'asz Local Lemma in conjunction with a deterministic colouring rule that ensures that short paths are not repetitively coloured. We present two proofs of this result. The first, due to \citet{GKM11}, uses entropy compression, which is a technique based on the algorithmic proof of the Lov\'asz Local Lemma by \citet{MoserTardos}. 

\begin{thm}[\citep{GPZ11}]
\label{PathChoose}
Every path is nonrepetitively list 4-colourable.
\end{thm}

\begin{proof}
Let $L$ be a 4-list assignment of the path $(v_1,\dots,v_n)$. We may assume that $|L(v_i)|=4$ for each $i\in\{1,\dots,n\}$. Apply the following algorithm, where $R$ is a binary sequence called the \emph{record}. At the start of the while loop, vertices $v_1,\dots,v_{i-1}$ are coloured and vertices $v_i,\dots,v_n$ is uncoloured.

\begin{samepage}
\framebox{	\begin{minipage}{\textwidth-5mm}
\hspace*{10mm} let $i:=1$\\
\hspace*{10mm} let $R:= ()$\\
\hspace*{10mm} \textbf{while} $i \leq n$ \textbf{do}\\
\hspace*{20mm} randomly colour $v_i$ from $L(v_i)$\\
\hspace*{20mm} append one 0 to $R$\\
\hspace*{20mm} \textbf{if} some repetitively coloured subpath $P$ appears \textbf{then}\\
\hspace*{30mm} let $k := \frac12 |V(P)|$\\
\hspace*{30mm} uncolour the last $k$ vertices of $P$\\
\hspace*{30mm} let $i:=  i - k + 1$\\
\hspace*{30mm} append $k$ 1's to $R$\\
\hspace*{20mm} \textbf{else}\\
\hspace*{30mm} let $i:= i+1$\\
\hspace*{20mm} \textbf{end-if}\\
\hspace*{10mm} \textbf{end-while}
\end{minipage}}
\end{samepage}

\bigskip
Each iteration of the while loop is called a \emph{step}. Let $R_t$ be the record $R$ at the end of step $t$. Let $\phi_t$ be the current colouring at the end of step $t$. A key property is that $(R_t,\phi_t)$ is a `lossless encoding' of the actions of the algorithm. That is, given $(R_t,\phi_t)$ one can determine $(R_{t-1}, \phi_{t-1})$ because whenever a repetitively coloured subpath $P$ appears, the colours on the second half of $P$ (which is uncoloured by the algorithm) are determined by the colours on the first half of $P$.

Consider the status of the algorithm at the end of some time step $t\geq 1$. Let $a_t$ and $b_t$ respectively be the number of 0's and the number of 1's in $R_t$. Observe that $a_t=t$ and the algorithm maintains the invariant that $a_t-b_t$ equals the number of coloured vertices under $\phi_t$. Call $a_t-b_t$ the \emph{type} of $R_t$, which is an element of $\{1,\dots,n\}$. Let $\widetilde{R_t}$ be the binary sequence obtained by adding $a_t-b_t$ 1's at the end of $R_t$. Thus $\widetilde{R_t}$ is a Dyck word of length $|R_t|+(a_t-b_t) = |R_t| + a_t - (|R_t|-a_t) = 2a_t = 2t$. Here a \emph{Dyck word} is a binary sequence with an equal number of 0's and 1's, such that every prefix has at least as many 0's as 1's. The number of Dyck words of length $2t$ equals the $t$-th Catalan number $C_t:=\frac{1}{t+1}\binom{2t}{t}$. Thus the number of distinct $R_t$'s is at most $C_t \times \# \text{types} = nC_t$. Since each vertex has four possible colours or is uncoloured, the number of distinct $\phi_t$'s is at most $(4 + 1)^n$. 

Consider the $4^t$ possible executions of the algorithm up to time $t$. For each such execution, the algorithm either finds a nonrepetitive colouring of the whole path or `fails' and produces a pair $(R_t, \phi_t)$. By the lossless  encoding property, distinct fail executions produce distinct pairs $(R_t, \phi_t)$. Thus the number of fail executions is at most the number of pairs ($R_t,\phi_t)$, which is at most $n5^n C_t \approx n 5^n \pi^{-1/2}t^{-3/2} 4^t$, which is less than $4^t$ for $t\gg n$. Thus, there exists an execution that does not fail. Therefore $(v_1,\dots,v_n)$ is $L$-colourable, and every path is nonrepetitively list 4-colourable. 
\end{proof}

Our second proof that $\pich(P)\leq 4$ uses a  simple counting argument of \citet{Rosenfeld20}.

\begin{thm}[\citep{Rosenfeld20}]
\label{RosenfeldPath}
Every path is nonrepetitively list 4-colourable. In fact, for every 4-list assignment $L$ of an $n$-vertex path, there are at least $2^{n+1}$ nonrepetitive $L$-colourings. 
\end{thm}

\begin{proof}
Let $L$ be a 4-list assignment of a path $P=(v_1,v_2,\dots,v_n)$. For $m\in\{1,\dots,n\}$, let $C_m$ be the number of nonrepetitive $L$-colourings of the subpath $P_m:=(v_1,\dots,v_m)$. We now prove that $C_{m+1} \geq 2C_m$ by induction on $m\geq 1$. The base case holds since $C_1=4$ and $C_2\geq 3C_1$. Let $m\in\{3,\dots,n\}$ and assume the claim holds for all values less than $m$. Thus $C_{m-1}\geq 2^i\, C_{m-i-1}$ for all $i\in\{1,\dots,m-2\}$. Let $F$ be the set of repetitive $L$-colourings of $P_m$ that induce a nonrepetitive colouring of $P_{m-1}$. Then $C_m = 4C_{m-1} - |F|$. For $i\in\NN$, let $F_i$ be the colourings in $F$ that contain a repetitively coloured path on $2i$ vertices, which must end at vertex $v_m$. Then $|F| \leq \sum_{i\geq 1}|F_i|$. For each colouring in $F_i$, the colours of vertices $v_{m-i+1},\dots,v_m$ are determined by the colours of vertices $v_{m-2i+1},\dots,v_{m-i}$. Since $v_1,\dots,v_{m-i}$ induce a nonrepetitively coloured path, $|F_i| \leq C_{m-i} \leq 2^{-i+1} C_{m-1}$. Thus $|F| \leq \sum_{i\geq 1}|F_i| \leq \sum_{i\geq 1} 2^{-i+1} C_{m-1} \leq 2C_{m-1}$. Hence $C_{m} = 4C_{m-1} - |F| \geq 2C_{m-1}$, as claimed. It follows that there exist at least $2^{n+1}$ nonrepetitive $L$-colourings of $P$.  
\end{proof}

\begin{open}
\label{Path3Choose}
Is every path nonrepetitively list $3$-colourable?~\citep{Gryczuk-IJMMS07,CG07,MS15}? Note that a simple adaptation to the proof of \cref{RosenfeldPath} shows that every path is list 3-colourable such that every subpath with at least four vertices is nonrepetitively coloured; that is, the only repetitively coloured subpaths have two vertices. The results of \citet{ZZ16} are also relevant here. 
\end{open}

The following multi-colour generalisation of \cref{RosenfeldPath} will be useful for the study of nonrepetitive colourings of subdivisions in \cref{Subdivisions}. \citet{Shur10} established precise asymptotic bounds on the number of distinct nonrepetitive $r$-colourings in a path. So that our presentation is self-contained, we present the slightly weaker result with a simple proof. 

\begin{thm}
\label{MultiColourPath}
Fix $r\geq 4$. For every $r$-list assignment $L$ of an $n$-vertex path, there are at least $k^n$ nonrepetitive $L$-colourings, where $k:= \tfrac12 ( r + \sqrt{r^2-4r} ) \geq r-2$. 
\end{thm}

\begin{proof}
Let $L$ be an $r$-list assignment of a path $P=(v_1,v_2,\dots,v_n)$. For $m\in\{1,\dots,n\}$, let $C_m$ be the number of nonrepetitive $L$-colourings of the subpath $P_m:=(v_1,\dots,v_m)$. We now prove that $C_{m+1} \geq k\,C_m$ by induction on $m\geq 1$. The base case holds since $C_1=r>k$ and $C_2\geq (r-1)C_1 >k\,C_1$. Let $m\in\{3,\dots,n\}$ and assume the claim holds for all values less than $m$. Thus $C_{m-1}\geq k^i\, C_{m-i-1}$ for all $i\in\{1,\dots,m-2\}$. Let $F$ be the set of repetitive $L$-colourings of $P_m$ that induce a nonrepetitive colouring of $P_{m-1}$. Then $C_m = r\, C_{m-1} - |F|$. For $i\in\NN$, let $F_i$ be the colourings in $F$ that contain a repetitively coloured path on $2i$ vertices, which must end at vertex $v_m$. Then $|F| \leq \sum_{i\geq 1}|F_i|$. For each colouring in $F_i$, the colours of vertices $v_{m-i+1},\dots,v_m$ are determined by the colours of vertices $v_{m-2i+1},\dots,v_{m-i}$. Since $v_1,\dots,v_{m-i}$ induce a nonrepetitively coloured path, $|F_i| \leq C_{m-i} \leq k^{-i+1}\, C_{m-1}$. Thus $|F| \leq \sum_{i\geq 1}|F_i| \leq \sum_{i\geq 1} k^{-i+1}\, C_{m-1} \leq \frac{k}{k-1}\, C_{m-1}$. Hence $C_{m} = r\,C_{m-1} - |F| \geq (r-\frac{k}{k-1})C_{m-1} = k \,C_{m-1}$, as claimed. It follows that there exist at least $k^n$ nonrepetitive $L$-colourings of $P$.  
\end{proof}

\subsection{Cycles}
\label{Cycles}

Let $C_n$ be the $n$-vertex cycle. \citet{Currie-EJC02} proved that 
\begin{equation*}
\label{Cycle}
\pi(C_n)=
\begin{cases}
4	&\text{if }n\in\{5,7,9,10,14,17\},\\
3	&\text{otherwise.}
\end{cases}
\end{equation*}
\citet{BW08} considered walk-nonrepetitive colourings of cycles, and showed:
\begin{equation*}
\label{PathCycle}
\rho(C_n) \leq \sigma(C_n )\leq 5.
\end{equation*}

\begin{open}
	Is $\pich(C_n )\leq 4$ for infinitely many $n$? 
	Is $\rho(C_n )\leq 4$ for infinitely many $n$?
	It is possible that $\pich(C_n )\leq 3$ or $\rho(C_n )\leq 3$ for infinitely many $n$, but these questions are open even for paths (\cref{Path3Choose,PathRhoOpen}). 
	Is $\sigma(C_n )\leq 4$ for infinitely many $n$?
\end{open}

%

\subsection{Bounded Degree Graphs}
\label{BoundedDegree}

\citet{AGHR02} proved that graphs with maximum degree $\Delta$ are nonrepetitively edge-colourable with 
$O(\Delta^2)$ colours. The precise bound shown was $\pi'(G)\leq 2^{16}\Delta^2$. \citet{AGHR02} remarked that the proof also works for nonrepetitive vertex colourings; that is, $\pi(G)\leq 2^{16}\Delta^2$. Several authors subsequently improved this constant: 
to $36\Delta^2$ by \citet{Grytczuk07}, 
to $16\Delta^2$ by \citet{Gryczuk-IJMMS07}, 
to $(12.2+o(1))\Delta^2$ by \citet{HJ-DM11}, 
and to $10.4\Delta^2$ by \citet{KSX12}. 
All these proofs used the Lov\'asz Local Lemma~\citep{EL75}. 
\citet{DJKW16} improved the constant to 1, by showing that for every graph $G$ with maximum degree $\Delta$,
\begin{equation}
\label{DeltaSquared}
\pi(G) \leq \Delta^2 + O(\Delta^{5/3}).
\end{equation}
The proof of \citet{DJKW16} uses entropy compression; see \citep{GMP14,GMP20,EsperetParreau} for refinements and simplifications to the method. Equation~\cref{DeltaSquared} was subsequently proved using a variety of techniques: the local cut lemma of \citet{Bernshteyn17}, cluster-expansion \citep{BFPS11,Aprile14}, and a novel counting argument due to \citet{Rosenfeld20}. The best known asymptotic bound is $\pi(G) \leq \Delta^2 + O(\Delta^{5/6})$ due to \citet{HS17}. 

We present two proof of \cref{DeltaSquared}. The first, perhaps surprisingly, uses nothing more than the Lov\'asz Local Lemma. The second is the counting arguement due to \citet{Rosenfeld20}. Both proofs use the following well known observation:

\begin{lem}
\label{PathsContainingVertex}
For every graph $G$ with maximum degree $\Delta$, for every vertex $v$ of $G$, and for every $s\in\NN$, there are at most $s\Delta(\Delta-1)^{2s-2}$ paths on $2s$ vertices that contain $v$ (where we consider a path to be a subgraph of $G$, so that a path and its reverse are counted once.)  
\end{lem}

\begin{proof}
	Let $\mathcal{P}$ be the set of $2s$-vertex paths in $G$ that contain $v$. For each $P\in\mathcal{P}$, by choosing the start vertex of $P$ appropriately, we may consider $v$ to be the $i$-th vertex in $P$, for some $i\in\{1,\dots,s\}$. There are at most $\Delta(\Delta-1)^{2s-2}$ paths in $\mathcal{P}$ in which $v$ is the first vertex (since are at most $\Delta$ choices for the neighbour of $v$ in the path, and once this is fixed, for each of the $2s-2$ internal vertices there are at most $\Delta-1$ choices). For each $i\in\{2,\dots,s\}$, there are at most $\Delta(\Delta-1)^{2s-2}$ paths in $\mathcal{P}$ in which $v$ is the $i$-th vertex (since at $v$ there are 
	at most $\Delta(\Delta-1)$ choices for the neighbours of $v$ in the path, and once these are fixed, for each of the remaining $2s-3$ internal vertices there are at most $\Delta-1$ choices). 	In total, there are at most $s\Delta(\Delta-1)^{2s-2}$ paths on $2s$ vertices that contain $v$. 
\end{proof}


\subsection{Lov\'asz Local Lemma}

The Lov\'asz Local Lemma is a powerful tool for proving the existence of combinatorial objects, and has been applied in numerous and diverse settings. The following is a statement of the General Local Lemma, which is due to Lov\'asz and first published by \citet{Spencer77}.

\begin{lem}[\GLLL]
\label{glll}
	Let $\EE=\{A_1,\dots,A_n\}$ be a set of `bad' events, such that each $A_i$ is mutually independent of $\EE\setminus(\DD_i\cup \{A_i\})$ for some $\DD_i\subseteq\EE$. Let $x_1,\dots,x_n\in[0,1)$ such that for each $i\in\{1,\dots,n\}$,
	$$\prob(A_i) \leq \prod_{A_j\in\DD_i} (1-x_j).$$
	Then with positive probability, none of $A_1,\dots,A_n$ occur. 
\end{lem}

\cref{glll} can be difficult to apply, since choosing the right values of $x_1,\dots,x_n$ is somewhat mysterious. The following Weighted Local Lemma by \citet[p.221]{MR02} avoids this difficulty, since in practice the weights $t_1,\dots,t_n$ are self-evident. 

\begin{lem}[\WLLL]
\label{wlll}
	Let $\EE=\{A_1,\dots,A_n\}$ be a set of `bad' events, such that each $A_i$ is mutually independent of $\EE\setminus(\DD_i\cup \{A_i\})$ for some $\DD_i\subseteq\EE$. Assume $p\in[0,\frac14]$ and $t_1,\dots,t_n\geq 1$ are real numbers such that for each $i\in\{1,\dots,n\}$,
	\begin{itemize}
		\item $\prob(A_i) \leq p^{t_i}$, and
		\item $\sum_{A_j\in\DD_i} (2p)^{t_j} \leq \frac{t_i}{2}$.
	\end{itemize}
	Then with positive probability, none of $A_1,\dots,A_n$ occur. 
\end{lem}

\cref{wlll} leads to the following straightforward proof that $\pi(G)\leq O(\Delta(G)^2)$, which we include as a warm-up. 

\begin{prop}
\label{2DeltaSquared}
For every graph $G$ with maximum degree $\Delta\geq 1$, 
$$\pich(G)\leq \ceil{ 2\Delta^2 + 4\Delta \sqrt{\Delta+1} + 4\Delta }.$$
\end{prop}

\begin{proof}
Let $c:= 2\Delta^2 + 4\Delta \sqrt{\Delta+1} + 4\Delta$ and $r:= \frac{2\Delta^2}{c}<1$ and $p:=\frac{1}{c}\leq \frac{1}{4}$.  Let $L$ be a $\ceil{c}$-list assignment for $G$. Colour each vertex $v$ of $G$, independently at random, by an element of $L(v)$. Let $P_1,\dots,P_n$ be the non-empty paths in $G$ with an even number of vertices. For each $i\in\{1,\dots,n\}$, let $A_i$ be the event that $P_i$ is repetitively coloured, and let $t_i:=\frac12 |V(P_i)|$. Then $\prob(A_i) \leq p^{t_i}$, and the first condition in \cref{wlll} is satisfied. Let $\DD_i:= \{A_j: P_j\cap P_i\neq\emptyset, j\neq i\}$. Then $A_i$ is mutually independent of $\{A_1,\dots,A_n\}\setminus(\DD_i\cup \{A_i\})$. By \cref{PathsContainingVertex}, each $P_i$ intersects at most $2t_is\Delta^{2s-1}$ paths with $2s$ vertices. The second condition in \cref{wlll} is satisfied since 
$$\sum_{A_j\in\DD_i}\!\! (2p)^{t_i} \; 
\leq  \sum_{s\geq 1} (2t_is\Delta^{2s-1}) (2p)^s
= \frac{2t_i}{\Delta} \sum_{s\geq 1} s(2p\Delta^2)^s 
= \frac{2t_i}{\Delta} \sum_{s\geq 1} sr^s 
= \frac{2t_i}{\Delta} \frac{r}{(1-r)^2} 
= \frac{t_i}{2}.$$
The penultimate equality here uses a formula for the sum of an arithmetico-geometric sequence~\citep{ArithmeticoGeometricSequence}. The last equality is proved by solving the quadratic, $4r=\Delta(1-r)^2$, and substituting $r=\frac{2\Delta^2}{c}$. By \cref{wlll}, with positive probability, none of $A_1,\dots,A_n$  occur. Hence there exists an $L$-colouring such that none of $A_1,\dots,A_n$ occur, in which case there are no repetitively coloured paths. Therefore $\pich(G)\leq \ceil{c}$. 
\end{proof}

\citet{MR02} write, ``As the reader will see upon reading the proof of the Weighted Local Lemma, the constant terms in the statement can be adjusted somewhat if needed.'' The next lemma does this.

\begin{lem}[\ELLL]
\label{ewlll}
Fix $\epsilon,\delta,p>0$ such that $0< p \leq \delta< \frac{1}{1+\epsilon}$. Define $\gamma := (1+\epsilon)\delta$ and $\alpha := \tfrac{1}{\gamma}\log(\tfrac{1}{1-\gamma})$ and $\beta := \tfrac{1}{\alpha}\log(1+\epsilon)$. Let $\EE=\{A_1,\dots,A_n\}$ be a set of `bad' events, such that each $A_i$ is mutually independent of $\EE\setminus(\DD_i\cup \{A_i\})$ for some $\DD_i\subseteq\EE$. Let $t_1,\dots,t_n\geq 1$ be real numbers such that for each $i\in\{1,\dots,n\}$,
\begin{itemize}
	\item $\prob(A_i) \leq p^{t_i}$, and
	\item $\displaystyle \sum_{A_j\in\DD_i} ((1+\epsilon)p)^{t_j} \leq \beta t_i $.
\end{itemize}
Then with positive probability, none of $A_1,\dots,A_n$ occur. 
\end{lem}


Note that \cref{ewlll} with $\epsilon=1$ implies \cref{wlll} since $\beta \geq \frac{1}{2}$ for $p\in[0,\frac{1}{4}]$.  The proof of \cref{ewlll} is essentially the same as the proof of \cref{wlll} by \citet{MR02}, who use $\epsilon=1$ and $\alpha=2\log 2$. 
 
\begin{proof}[Proof of \cref{ewlll}.]For each $i\in\{1,\dots,n\}$, let $x_i := ((1+\epsilon)p)^{t_i}$. Then $x_i\in[0, (1+\epsilon) \delta]$. Note that $\alpha\geq 1$ and $1-x\geq \exp(-\alpha x)$ for all $x\in [0,(1+\epsilon)\delta]$. Thus
\begin{align*}
x_i \prod_{A_j\in\DD_i} (1-x_j) 
& \geq x_i \prod_{A_j\in\DD_i} \exp( -\alpha x_j )\\
& = x_i \exp( - \alpha \sum_{A_j\in\DD_i}  x_j )\\
& = ((1+\epsilon)p)^{t_i}  \exp( - \alpha \sum_{A_j\in\DD_i} ((1+\epsilon)p)^{t_j})\\
& \geq ((1+\epsilon)p)^{t_i} \exp( -\alpha \beta t_i )\\
& = ((1+\epsilon)p)^{t_i} \exp( - \log(1+\epsilon) t_i )\\
& = p^{t_i} \\
& \geq \mathbb{P}(A_i).
\end{align*}	
The result follows from \cref{glll}.
\end{proof}

We now give our first proof of \cref{DeltaSquared}.

\begin{thm}
	For every graph $G$ with maximum degree $\Delta\geq 2$, 
	$$\pich(G)\leq \Delta^2 + 2^{4/3} \Delta^{5/3}  + O( \Delta^{4/3}) .$$	
\end{thm}

\begin{proof}
Let $\epsilon := 2^{1/3}\Delta^{-1/3}$. (The reason for this definition will be apparent at the end of the proof.)\ 
Let $\delta:= (1+\epsilon)^{-1}\Delta^{-2}$. As in \cref{ewlll}, define $\gamma := (1+\epsilon)\delta=\Delta^{-2}$ and $\alpha := \tfrac{1}{\gamma}\log(\tfrac{1}{1-\gamma})$ and $\beta := \tfrac{1}{\alpha}\log(1+\epsilon)$. Note that $1\leq \alpha \leq 4 \log(\tfrac{4}{3})<1.15$ since $\gamma \leq \frac{1}{4}$. (It may help the reader's intuition to pretend that $\alpha=1$.)\ Let 
$$c:= (1+\epsilon)\Delta( \Delta + \tfrac{1}{\beta} (\sqrt{2\beta\Delta+1}+1) ).$$
We now prove that $\pich(G)\leq \ceil{c}$. First we write $c$ in a more convenient form. Since $\beta>0$, 
$$(2\beta\Delta+2-2\sqrt{2\beta\Delta+1})(2\beta\Delta+2+2\sqrt{2\beta\Delta+1}) = (2\beta\Delta+2)^2 - 4(2\beta\Delta+1) = 4\beta^2\Delta^2.$$
Thus
$$\frac{2\beta\Delta^3}{2\beta\Delta+2-2\sqrt{2\beta\Delta+1}} 
= \frac{\Delta}{2\beta} (2\beta\Delta+2+2\sqrt{2\beta\Delta+1}) 
= \Delta( \Delta + \tfrac{1}{\beta}( \sqrt{2\beta\Delta+1} +1 ) ). $$
Muliplying by $1+\epsilon$, 
$$c= \frac{(1+\epsilon)2\beta\Delta^3}{(2\beta\Delta+2)- \sqrt{8\beta\Delta+4}}.$$
Let $p:=\frac{1}{c}$. Then $p\leq \delta\leq \frac{1}{1+\epsilon}$, as required by \cref{ewlll}.   
Let $r:= (1+\epsilon)p \Delta^2<1$. Then 
$$r=\frac{(1+\epsilon)\Delta^2}{c} = \frac{(2\beta\Delta+2)- \sqrt{8\beta\Delta+4}}{2\beta\Delta}
 = \frac{(2\beta\Delta+2)- \sqrt{(2\beta\Delta+2)^2 - 4 \beta^2\Delta^2}}{2\beta\Delta}.$$
By the quadratic formula,  $\beta\Delta r^2 +(-2\beta\Delta-2)r+\beta\Delta = 0$. 
That is, 	
$$2r= \beta\Delta(1-2r+r^2)= \beta\Delta(1-r)^2.$$

Let $L$ be a $\ceil{c}$-list assignment for $G$. Colour each vertex $v$ of $G$, independently at random, by an element of $L(v)$. Let $P_1,\dots,P_n$ be the non-empty paths in $G$ with an even number of vertices. Here we consider a path to be a subgraph of $G$, so that a path and its reverse are the same path.  For each $i\in\{1,\dots,n\}$, let $A_i$ be the event that $P_i$ is repetitively coloured, and let $t_i:=\frac12 |V(P_i)|$. Then $\prob(A_i)\leq p^{t_i}$, and the first condition in \cref{ewlll} is satisfied. Let $\DD_i:= \{A_j: P_j\cap P_i\neq\emptyset, j\neq i\}$. Then $A_i$ is mutually independent of $\{A_1,\dots,A_n\}\setminus(\DD_i\cup \{A_i\})$. By \cref{PathsContainingVertex}, for each $s\in\NN$, each $P_i$ intersects at most $2t_is\Delta^{2s-1}$ paths with $2s$ vertices. (The lower order terms in this result can be improved by using \cref{PathsContainingVertex} more precisely.)\ The second condition in \cref{ewlll} is satisfied since 
\begin{align}
\label{KeyLine}
\sum_{A_j\in\DD_i}\!\! ((1+\epsilon)p)^{t_i} 
& \leq  \sum_{s\geq 1} (2t_is\Delta^{2s-1}) ((1+\epsilon)p)^s\nonumber \\
&	 \leq  \sum_{s\geq 1} (2t_is\Delta^{2s-1}) \left(\frac{r}{\Delta^2}\right)^s\nonumber \\
    &= \frac{2t_i}{\Delta}  \sum_{s\geq 1} sr^s \nonumber \\
	&= \frac{2t_i}{\Delta} \frac{r}{(1-r)^2} \nonumber \\
	&= \beta\, t_i.
\end{align}
The penultimate equality uses a formula for the sum of an arithmetico-geometric sequence~\citep{ArithmeticoGeometricSequence}. By \cref{wlll}, with positive probability, none of $A_1,\dots,A_n$  occur. Hence there exists an $L$-colouring such that none of $A_1,\dots,A_n$ occur, in which case there are no repetitively coloured paths. Therefore $\pich(G)\leq c$. 

It remains to prove the upper bound on $\ceil{c}$. Using Taylor series expansion as $\epsilon\to0$, 
\begin{align*}
c = & 
(1+\epsilon)\Delta^2 +
\frac{(1+\epsilon)\Delta(\sqrt{2\beta\Delta+1}+1)}{\beta} \\
\leq & 
(1+\epsilon)\Delta^2 + 
\frac{(1+\epsilon)\Delta(\sqrt{2\beta\Delta}+2)}{\beta} \\
= & 
(1+\epsilon)\Delta^2 + 
\frac{(1+\epsilon)\sqrt{2\Delta^3}}{\sqrt{\beta}} +
\frac{2(1+\epsilon)\Delta}{\beta} \\
= & 
(1+\epsilon)\Delta^2 + 
 \frac{(2\alpha)^{1/2}\Delta^{3/2}(1+\epsilon)}{\sqrt{\log(1+\epsilon)}} +
 \frac{2\alpha\Delta(1+\epsilon)}{\log(1+\epsilon)} \\
= & 
(1+\epsilon)\Delta^2 + 
(2\alpha)^{1/2}\Delta^{3/2} \left(
 	\epsilon^{-1/2} + \tfrac{5}{4}\epsilon^{1/2}  + \tfrac{17}{96} \epsilon^{3/2} - 
 	\tfrac{13}{384} \epsilon^{5/2}  + O(\epsilon^3)		\right)  \\
& \hspace*{20mm}  + 2\alpha\Delta \left( 
	\epsilon^{-1}  + \tfrac{3}{2} + \tfrac{5}{12} \epsilon - \tfrac{1}{24}\epsilon^2 + \tfrac{11}{720} \epsilon^3 + O(\epsilon^4) \right) \\
= & 
(1+\epsilon)\Delta^2 + 
(2\alpha)^{1/2}\Delta^{3/2} \left( \epsilon^{-1/2} + \tfrac{5}{4}\epsilon^{1/2}  + O( \epsilon^{3/2})		\right)  
  + 2\alpha\Delta \left( \epsilon^{-1}  + \tfrac{3}{2} + O( \epsilon ) \right) \\
= & 
(1+2^{1/3}\Delta^{-1/3})\Delta^2 + 
(2\alpha)^{1/2}\Delta^{3/2} \left( (2^{1/3}\Delta^{-1/3})^{-1/2} + \tfrac{5}{4}(2^{1/3}\Delta^{-1/3})^{1/2}  + O(\Delta^{-1/2})		\right)  +\\
& \hspace*{32mm}
+ 2\alpha\Delta \left( (2^{1/3}\Delta^{-1/3})^{-1}  + \frac{3}{2} + O( \Delta^{-1/3} ) \right) \\
= & 
\Delta^2 + 2^{1/3}\Delta^{5/3} + 
(2\alpha)^{1/2}\Delta^{3/2} 
\left( 2^{-1/6}\Delta^{1/6} + \tfrac{5}{4}\,2^{1/6}\Delta^{-1/6}  + O(\Delta^{-1/2})		\right)  +\\
& \hspace*{32mm}
+ 2\alpha\Delta \left( 2^{-1/3}\Delta^{1/3}  + \tfrac{3}{2} + O( \Delta^{-1/3} ) \right) \\
= & 
\Delta^2 + 2^{1/3}\Delta^{5/3} + 
\Big( 2^{1/3} \alpha^{1/2}\Delta^{5/3}  + 
\alpha^{1/2} 5\,2^{-4/3} \Delta^{4/3}  +  O(\Delta)  \Big) + \\
& \hspace*{32mm} \left( \alpha 2^{2/3}\Delta^{4/3}  + 3\alpha\Delta + O( \Delta^{2/3} ) \right) \\
= & 
\Delta^2 + 
2^{1/3}( 1+  \alpha^{1/2}) \Delta^{5/3}  + ( \alpha^{1/2} 5\cdot 2^{-4/3} +  \alpha 2^{2/3}) \Delta^{4/3}  + 
O(\Delta)   .
\end{align*}
Note that $\alpha^{1/2}=(\tfrac{1}{\gamma}\log(\tfrac{1}{1-\gamma}))^{1/2} \leq 1 + \gamma$ since $\gamma\leq\frac14$. Thus 
$\alpha^{1/2} \leq 1 + \Delta^{-2}$, implying $c \leq \Delta^2 + 2^{4/3} \Delta^{5/3}  + (5\cdot 2^{-4/3} +2^{2/3}) \Delta^{4/3} + O(\Delta)$. 
\end{proof}

I now reflect on how to use the \ELLL. First introduce a parameter $\epsilon=\epsilon(\Delta)$, which tends to 0 as $\Delta\to\infty$. Leave $\epsilon$ undefined at this stage; it can be determined optimally at the end of this process. Introduce a variable $c$ for the number of colours, and  leave $c$ undefined at first. Guess a lower bound $c'$ for $c$, as close to $c$ as possible. In the above proof, $c'$ is  $(1+\epsilon)\Delta^2$. 
Let $p:=\frac{1}{c}$ and $\delta:=\frac{1}{c'}$. So $\delta$ is a close upper bound for $p$. Define $\gamma$, $\alpha$ and $\beta$ (in terms of $\delta$ and $\epsilon$) as in \cref{ewlll}.  Define appropriate events, and compute their probabilities (in term of $p$), from which the weights $t_1,\dots,t_n$ should be self-evident. Then bound the dependencies of events. Equation~\cref{KeyLine} is the heart of the proof. Using the bounds on the dependencies, determine an upper bound $X_i(p,\epsilon)$ for $\sum_{A_j\in\DD_i}((1+\epsilon)p)^{t_i}$. Then solving the equation $X_i(p,\epsilon)=\beta t_i$  gives a value for $p$ and in turn a value for $c$ (in terms of $\epsilon$) so that \cref{ewlll} is applicable. Finally, choose $\epsilon$ to minimise $c$. Using this approach, the mysterious process of choosing the numbers $x_1,\dots,x_n$ in the General Local Lemma is partially automated. 

%
%
%
%
%
%
%
%

%
%
%

\subsection{Rosenfeld Counting}
\label{RosenfeldCounting}

The following proof of \cref{DeltaSquared} uses a clever counting argument due to \citet{Rosenfeld20}\footnote{The bound in \cref{DeltaSquaredRosenfeld} is slightly less than the bound of \citet{Rosenfeld20}, improving the coefficient in the $\Delta$ term.}, which is inspired by the so-called power series method for pattern avoidance~\citep{Rampersad11,Ochem16,BW13}. See \citep{WW20} for an abstract generalisation of \cref{DeltaSquaredRosenfeld} that has applications to several other (hyper)graph colouring problems. 

\begin{thm}
\label{DeltaSquaredRosenfeld}
For every graph $G$ with maximum degree $\Delta\geq 2$,
$$\pich(G) \leq  \Delta^2 + 3\cdot 2^{-2/3} \Delta^{5/3} + 2^{2/3} \Delta^{4/3} - \Delta  - 2^{4/3} \Delta ^{2/3} + 2.$$
\end{thm}

This theorem is implied by the following lemma with $r:= (1+2^{1/3} \Delta^{-1/3})^{-1}$. In fact, this lemma proves a stronger result that implies there are exponentially many colourings and is essential for the inductive argument. For a list assignment $L$ of a graph $G$, let $\Pi(G,L)$ be the number of nonrepetitive $L$-colourings of $G$.

\begin{lem}
	\label{RosenfeldRevised}
	Fix an integer $\Delta\geq 2$ and a real number $r \in(0,1)$. Let 
	$$\beta := \frac{(\Delta-1)^2}{r} \quad \text{and}\quad c:= \ceil*{\beta + \frac{\Delta}{(1-r)^2}}.$$  
	Then for every graph $G$ with maximum degree $\Delta$, for every $c$-list assignment $L$ of $G$, and for every vertex $v$ of $G$, 
	$$\Pi(G,L) \geq \beta\, \Pi(G-v,L).$$ 
\end{lem}

\begin{proof}
We proceed by induction on $|V(G)|$. The base case with $|V(G)|=1$ is trivial (assuming $\Pi(G,L)\neq\emptyset$ if $V(G)=\emptyset$). Let $n$ be an integer such that the lemma holds for all graphs with less than $n$ vertices.
Let $G$ be an $n$-vertex graph with maximum degree $\Delta$. Let $L$ be a $c$-list assignment of $G$. Let $v$ be any vertex of $G$. Let $F$ be the set of $L$-colourings of $G$ that are repetitive but are nonrepetitive on $G-v$. Then 
	\begin{align}
	\label{F}
	\Pi(G,L) = |L(v)| \, \Pi(G-v,L) \,-\,|F| \geq c \, \Pi(G-v,L) \,-\,|F|.
	\end{align}
We now upper-bound $|F|$. For $i\in\NN$, let $F_i$ be the set of colourings in $F$, for which there is a  repetitively path in $G$ on $2i$ vertices. Then $|F| \leq \sum_{i\in\NN} |F_i|$. For each colouring $\phi$ in $F_i$ there is a repetitively path $PQ$ on $2i$ vertices in $G$ such that $v\in V(P)$,  $G-V(P)$ is nonrepetitively coloured by $\phi$, and $\phi$ is completely determined by the restriction of $\phi$ to $G-V(P)$ colouring (since the colouring of $Q$ is identical to the colouring of $P$). Charge $\phi$ to $PQ$. The number of colourings in $F_i$ charged to $PQ$ is at most $\Pi(G-V(P),L)$. Since $P$ contains $v$ and $i-1$ other vertices, by induction
	\begin{align*}
	\Pi(G-v,L) \geq \beta^{i-1} \, \Pi(G-V(P),L). 
	\end{align*}
Thus the number of colourings in $F_i$ charged to $PQ$ is at most $\beta^{1-i}\,\Pi(G-v,L)$. By \cref{PathsContainingVertex}, there are at most $i\Delta(\Delta-1)^{2i-2}$ paths on $2i$ vertices including $v$.  Thus 
	\begin{align*}
	|F_i|  \leq  	i\,\Delta(\Delta-1)^{2i-2}\,\beta^{1-i}\,\Pi(G-v,L)
	& = 	i\,\Delta \left(\frac{ (\Delta-1)^2}{\beta} \right)^{i-1}\,\Pi(G-v,L)\\
	& = 	i\,\Delta r^{i-1}\,\Pi(G-v,L).
	\end{align*}
	Hence
	\begin{align*}
	|F|  \leq \sum_{i\in\NN} |F_i| 
	=
	\sum_{i\in\NN}  i\,\Delta r^{i-1}\,\Pi(G-v,L)
	& =
	\Delta\, \Pi(G-v,L) \sum_{i\in\NN}  i\,r^{i-1}\\
	& =
	\frac{\Delta}{(1-r)^2}\, \Pi(G-v,L) .
	\end{align*}
By \cref{F}, 
	\begin{align*}
	\Pi(G,L) 
	\geq c \, \Pi(G-v,L) \,-\,|F|
	& \geq c \, \Pi(G-v,L) \,-\,\frac{\Delta}{(1-r)^2}\, \Pi(G-v,L)	\\
	& \geq \beta\, \Pi(G-v,L)  ,
	\end{align*}
as desired. 
\end{proof}

%

\begin{open}
What is the maximum nonrepetitive chromatic number of graphs $G$ with maximum degree 3? \cref{RosenfeldRevised} with $r=0.389$ implies $\pich(G)\leq 19$. I expect this bound can be improved using entropy compression tailored to the $\Delta\leq 3$ case.  
\end{open}

\subsection{Lower Bound}

\citet{AGHR02}  proved the following lower bound on the nonrepetitive chromatic nuber of bounded degree graphs. 

\begin{thm}[\cite{AGHR02}]
\label{DegreeLowerBound}
There is an absolute constant $c>0$ such that for all $\Delta$ there exists a graph $G$ with maximum degree $\Delta$, and
$$\pi(G) \geq \frac{c \Delta^2}{\log\Delta}.$$
\end{thm}

\begin{proof}
	We make no attempt to optimise the constant $c$. We may assume that $\Delta$ is sufficiently large, since the result is trivial for small $\Delta$ if $c$ is small enough. Let $n$ be a (large) integer. Define $p:=4 \sqrt{\frac{\log n}{n}}$. Let $G=G(12n,p)$ be the graph with vertex set $\{1,\dots,12n\}$ obtained by choosing each pair of distinct vertices to be an edge, independently at random with probability $p$. 
	
	We claim that $G$ satisfies the following three properties  with probability tending to 1 as $n\to\infty$:
	\begin{enumerate}[(i)]
		\item The maximum degree of $G$ is at most $\Delta:= 96\sqrt{n\log n}$.
		\item There is at least one edge of $G$ between any two disjoint sets of vertices each of size at least $n$.
		\item For every collection of $3n$ pairwise disjoint subsets $\{u_1,v_1\},\{u_2,v_2\},\dots,\{u_{3n},v_{3n}\}$ of $V(G)$, there is a subset $S\subseteq \{1,\dots,3n\}$ with $|S|\geq n$ such that the graph with vertex-set $S$ and edge-set $\{st: u_su_t,v_sv_t\in E(G)\}$ is connected.
	\end{enumerate}
	
	To prove this claim, we show that for each of (i) -- (iii) the probability of failure tends to 0 as $n\to\infty$. 
	
	For (i), the expected degree of each vertex $v$ of $G$ equals $p(12n-1) < 48 \sqrt{n\log n}$. Chernoff's Inequality implies that  the probability that $\deg(v) > 96\sqrt{n \log n}$ is less than $\exp(-16\sqrt{n \log n})$. Thus the probability that some vertex of $G$ has degree greater than $96\sqrt{n \log n}$ is less than $\exp(-16\sqrt{n \log n})n \to 0$. 
	
	For (ii), consider disjoint sets $A,B\subseteq V(G)$ with $|A|,|B|\geq n$. The probability that there is no $AB$-edge in $G$ equals $(1-p)^{|A|\,|B|} \leq \exp( -p\, |A|\,|B|)\leq  \exp( - 4n^{3/2} \sqrt{ \log n}) $. The number of such pairs $A,B$ is less than $\binom{12n}{n}^2\leq  (12e)^{2n}$. Hence, the probability that (ii) fails is less than 
	$\exp( -4 n^{3/2}\sqrt{\log n}) (12e)^{2n} = \exp( -(4 + o(1))\sqrt{\log n}\,n^{3/2} ) \to 0$.
	
	For (iii), fix pairwise disjoints subsets $\{u_1,v_1\},\{u_2,v_2\},\dots,\{u_{3n},v_{3n}\}$ of $V(G)$. We now estimate the probability that there is no set $S$ as in (iii). Let $H$ be the graph with vertex-set $\{1,\dots,3n\}$ and edge-set $\{ij: u_iu_j,v_iv_j\in E(G)\}$. Then $H$ is a random graph with $3n$ vertices in which every pair of distinct vertices forms an edge, independently at random with probability $p^2=\frac{16\log n}{n}$. Our objective is to estimate the probability that there is no connected component of at least $n$ vertices in $H$. If this happens, then the set of vertices of $H$ can be partitioned into two disjoint sets, each of size at least $n$ with no edges between them. The probability of this event is less than 
	$2^{3n}(1-p^2)^{n^2} < 2^{3n} \exp( -p^2 n^2) \leq 2^{3n} \exp( -16 (\log n) n) = 2^{3n} n^{ -16 n}$. 
	The number of pairwise disjoint subsets $\{u_1,v_1\},\{u_2,v_2\},\dots,\{u_{3n},v_{3n}\}$ of $V(G)$ equals 
	$\binom{12n}{6n} (6n-1)!! = \binom{12n}{6n} \frac{ (6n)!}{2^{3n} 3n!}  < (12n)^{6n}$. 
	%
	Thus the probability that (iii) fails is less than $(12n)^{6n} 2^{3n} n^{-16n} \to 0$. This completes the proof of the claim.
	
	Returning to the proof of the theorem, suppose that $G$ satisfies (i)--(iii). Consider any $6n$-colouring of $G$. Omit one vertex from each color class containing an odd number of vertices, and partition the remaining vertices within each colour class into pairs. This produces at least $3n$ pairs $\{u_1, v_1\},\dots,\{u_{3n}, v_{3n}\}$, where $u_i$ and $v_i$ have the same color. Thus there is a subset $S \subseteq \{1,\dots,3n\}$ satisfying (iii). By (ii) applied to the sets $\{u_t:t\in S\}$ and $\{v_s : s \in S\}$ there is an edge $u_tv_s$ of $G$ with $s, t \in S$. Let $(s = s_1, s_2, \dots , s_r = t)$ be a path from $s$ to $t$ in the graph with vertex-set $S$ and edge-set $\{ij:u_iu_j,v_iv_j \in E(G)\}$. Such a path exists, by (iii). Then, the path
	$(u_s,u_{s_2},u_{s_3},\dots,u_t,v_s,v_{s_2}v_{s_3},\dots, v_t)$ in $G$ is repetitively coloured. Thus $\pi(G)> 6n\geq \frac{\Delta^2}{3072 \log\Delta}$.
\end{proof}

We finish this subsection with a number of open problems. 


\begin{open}
What is the maximum nonrepetitive chromatic number of graphs with maximum degree $\Delta$? 	The answer is between $\Omega(\Delta^2/\log\Delta)$ and $\Delta^2+O(\Delta^{5/6})$. Given the plethora of proofs of the $(1+o(1))\Delta^2$ upper bound, it would be very interesting to obtain a $(1-\epsilon)\Delta^2$ upper bound for some fixed $\epsilon$. 
\end{open}

Little is known about stroll-nonrepetitive and walk-nonrepetitive colourings of graphs with bounded degree. 

\begin{open}
\label{DeltaRho}
Is there a function $f$ such that $\rho(G) \leq f( \Delta(G))$ for every graph $G$?
\end{open}

\begin{open}[\citep{BV08,BW08}] 
\label{DeltaSigma}
Is there a function $f$ such that $\sigma(G) \leq f( \Delta(G))$ for every graph $G$?
\end{open}

By \cref{SigmaCharacterisation}, questions \cref{DeltaRho,DeltaSigma} have the same answer. 

\cref{DeltaSigma} was extensively studied by \citet{BW08}. Indeed, \citet{BW08} formulated a conjecture that they claimed would imply a positive answer to \cref{DeltaSigma}. However, \citet{Hendrey20} disproved the conjecture. 
Also note that \citet{Aprile14} claimed to prove an affirmative answer to \cref{DeltaSigma}, but the proof has an error [personal communication, Manuel Aprile 2017].

The analogous lower bound questions are also of interest. 

\begin{open}
Is there a quadratic or super-quadratic lower bound on $\rho$ or $\sigma$ for some graph of maximum degree $\Delta$?
\end{open}






\subsection{Edge Colourings and Line Graphs}

An edge-colouring of a graph $G$ is nonrepetititive if for every path $P$ in $G$, the sequence of colours on the edges of $P$ is not a repetition. Nonrepetitive edge colourings have been studied in several papers~\citep{KT17,BV08, AGHR02,BK-AC04}, as has walk-nonrepetitive edge colourings~\citep{BV08}. Since any two edges incident with a common vertex form a repetition, 
\begin{equation}
\label{EdgePiLowerBound}
\pi'(G) \geq \chi'(G) \geq \Delta(G).
\end{equation}
Conversely, \citet{AGHR02} showed that $\pi'(G) \leq O(\Delta^2)$ for every graph $G$ with maximum degree $\Delta$. Let $L(G)$ be the line graph of a graph $G$. That is, $V(L(G)):=E(G)$ where two vertices in $L(G)$ are adjacent whenever the corresponding edges in $G$ share a common endpoint. Since every path in $G$ corresponds to a path in $L(G)$, 
\begin{equation}
\label{LineGraph}
\pi'(G) \leq \pi(L(G)) ).
\end{equation} 
The best upper bound, due to \citet{Rosenfeld20}, is 
\begin{equation}
\label{EdgePiBounds}
\Delta(G) \leq \pi'(G)\leq \Delta(G)^2 + O(\Delta^{5/3}).
\end{equation}
Resolving the gap in the bounds in \cref{EdgePiBounds} is an important open problem.

\begin{open}[\citep{AGHR02}]
\label{EdgeDegree}
Is there a constant $c$ such that for every graph $G$, 
$$\pi'(G) \leq c\,\Delta(G) \,?$$
\end{open}

Note that equality does not necessarily hold in \cref{LineGraph}. For example, $\pi'(C_4)=2$ but $\pi(L(C_4))=\pi(C_4)=3$. In general, a path in $L(G)$ might correspond to a cycle in $G$, so a nonrepetitive vertex-colouring of $L(G)$ does not necessarily correspond to a nonrepetitive edge-colouring of $G$. Nevertheless we have the following bounds. 
$$\pi(L(G)) 
\leq (1+o(1)) \Delta(L(G))^2 
\leq 
(4+o(1)) \Delta(G)^2
\leq (4+o(1)) \pi'(G)^2.
$$

\citet{AGHR02} determined the nonrepetitive chromatic index of complete graphs as follows (thus answering \cref{EdgeDegree} in the affirmative in this case). 

\begin{thm}[\citep{AGHR02}]
	\label{EdgeCompleteGraph}
For every $k\in\NN$, 
$$\pi'(K_{2^k}) = 2^k-1.$$
\end{thm}

\begin{proof}
\cref{EdgePiLowerBound} implies the lower bound, $\pi'(K_{2^k}) \geq 2^k-1$. For the upper bound, we may assume that $V(K_{2^k})$ is the set of elements of the additive group $\mathbb{Z}_2^k$. Colour each edge $vw$ by $v+w$, where addition is in $\mathbb{Z}_2^k$. Since $v+w=0$ if and only if $v=w$, each edge is coloured by a non-zero element of $\mathbb{Z}_2^k$. Suppose for the sake of contradiction that $P=(v_1,v_2,\dots,v_{2t+1})$ is a path whose edges are repetitively coloured. Thus $v_i+v_{i+1}=v_{t+i}+v_{t+i+1}$ for each $i\in\{1,\dots,t\}$. Hence 
$$v_1+v_{t+1}=\sum_{i=1}^t(v_i+v_{i+1}) = \sum_{i=1}^t(v_{t+i}+v_{t+i+1}) = v_{t+1}+v_{2t+1}.$$
Therefore $v_1=v_{2t+1}$ and $P$ is a cycle. This contradiction shows that $K_{2^k}$ is nonrepetitively coloured, and $\pi'(K_{2^k}) \leq 2^k-1$. 
\end{proof}

\begin{cor}[\citep{AGHR02}]
For every $n\in\NN$, 
$$n-1 \leq \pi'(K_n) \leq 2n-3.$$
\end{cor}

The proof method in \cref{EdgeCompleteGraph} generalises to show that for complete bipartite graphs, 
$\pi'(K_{2^k,2^k}) = 2^k$. Thus $n \leq \pi'(K_n) \leq 2n-1$. 
However, determining 
$\pi(L(K_n))$ and $\pi(L(K_{n,n}))$ are open. 

\begin{open}
What is $\pi(L(K_n))$? We know $n-1 \leq \pi(L(K_n)) \leq (4+o(1))n^2$. 
\end{open}

\begin{open}
\label{LineGraphCompleteBipartite}
What is $\pi(L(K_{n,n}))$? We know $n \leq \pi(L(K_{n,n})) \leq (4+o(1))n^2$. 
\end{open}



Total Thue coloring was introduced by \citet{SS15a}. A colouring of the edges and the vertices of a graph is a \emph{weak total Thue coloring} if the sequence of consecutive vertex-colors and edge-colors of
every path is nonrepetitive. If, in addition, the sequence of vertex-colors and the sequence of edge-colors of any path are both nonrepetitive then this is a \emph{(strong) total Thue coloring}. 

For every path $P$ in a graph $G$, the sequence of vertices and edges in $P$ corresponds to a path in the 1-subdivision of $G$. Thus every nonrepetitive colouring of $G^{(1)}$ defines a weak total nonrepetitive colouring of $G$. Hence the weak total Thue chromatic number of $G$ is at most $\pi(G^{(1)})$. Similarly, if $G'$ is the square of $G^{(1)}$, then every nonrepetitive colouring of $G'$ defines a strong total nonrepetitive colouring of $G$.  Hence the strong total nonrepetitive chromatic number of $G$ is at most $\pi(G')$. 

\citet{Rosenfeld20} proved that every graph with maximum degree $\Delta$ has weak total Thue chromatic number at most $6\Delta$ and at most $\ceil{\frac{17}{4}\Delta}$ if $\Delta\geq 300$. \cref{1SubdivisionDelta} gives an upper bound on $\pi(G^{(1)})$ which implies that every graph with maximum degree $\Delta$ has weak total Thue chromatic number at most $\ceil{5.22\, \Delta}$. 


\section{Trees and Treewidth}
\label{TreesTreewidth}

\subsection{Trees}
\label{Trees}

\citet{BGKNP07} proved that every tree is path-nonrepetitively 4-colourable. \citet{DEJWW20} extended this result for strolls. 

\begin{thm}[\citep{BGKNP07,DEJWW20}]
For every tree $T$, $$\pi(T) \leq \rho(T) \leq 4.$$
\end{thm}

\begin{proof}
Let $r$ be any vertex of $T$. Let $V_i:=\{ v \in V(T): \dist_T(r,v)=i\}$. Thus  $(V_0,V_1,\dots,V_n)$ is a shadow-complete layering of $T$. Each $G[V_i]$ is an independent set and is thus stroll-nonrepetitively 1-colourable. The result then follows from \cref{ShadowCompleteRho}.
\end{proof}

\citet{BGKNP07} show that $\pi(T)=4$ for some tree $T$.

\citet{BW08} characterised walk-nonrepetitive colourings of trees as follows. This strengthens the analogous characterisation for general graphs in \cref{SigmaCharacterisation}. 

\begin{lem}[\citep{BW08}]
\label{TreesWalk}
A colouring $c$ of a tree $T$ is walk-nonrepetitive if and only if $c$ is path-nonrepetitive and distance-$2$.
\end{lem}

\begin{proof}
Every walk-nonrepetitive colouring is path-nonrepetitive and distance-$2$ by \cref{SigmaCharacterisation}. We now prove the converse. Assume $c$ is a nonrepetitive distance-$2$ colouring of $T$. Suppose on the contrary that $T$ has a repetitively coloured non-boring walk. Let $W=(v_1,v_2,\dots,v_{2t})$ be a repetitively coloured non-boring  walk in $T$ of minimum length. Some vertex is repeated in $W$, as otherwise $W$ would be a repetitively coloured path. By considering the reverse of $W$, without loss of generality, $v_i=v_j$ for some $i\in\{1,\dots,t-1\}$ and $j\in\{i+2,\dots,2t\}$. Choose $i$ and $j$ to minimise $j-i$. Thus $v_i$ is not in the sub-walk $(v_{i+1},v_{i+2},\dots,v_{j-1})$. Since $T$ is a tree, $v_{i+1}=v_{j-1}$. Thus $i+1=j-1$, as otherwise $j-i$ is not minimised. That is, $v_i=v_{i+2}$. Assuming $i\neq t-1$, since $W$ is repetitively coloured, $c(v_{t+i})=c(v_{t+i+2})$, which implies that $v_{t+i}=v_{t+i+2}$ because $c$ is a distance-$2$ colouring. Thus, even if $i=t-1$, deleting the vertices $v_{i},v_{i+1},v_{t+i},v_{t+i+1}$ from $W$, gives a walk $(v_1,v_2,\dots,v_{i-1},v_{i+2},\dots,v_{t+i-1},v_{t+i+2},\dots,v_{2t})$ that is also repetitively coloured. This contradicts the minimality of the length of $W$.
\end{proof}

\citet{BW08} showed the following bound on $\sigma(T)$.  

\begin{lem}[\citep{BW08}]
\label{TreeSigma}
For every tree $T$ with maximum degree $\Delta$, 
$$\Delta(T)+1 \leq \sigma(T) \leq 4\Delta.$$
\end{lem}

\begin{proof}
The lower bound follows from \cref{SigmaBounds}. For the upper bound, root $T$ at some leaf vertex $r$. Let $V_i:=\{ v \in V(T): \dist_T(r,v)=i\}$. Thus  $(V_0,V_1,\dots,V_n)$ is a shadow-complete layering of $T$. Each $G[V_i]$ is an independent set and is thus stroll-nonrepetitive in any colouring. Each vertex $v\in V_i$ has at most $\Delta-1$ children, all of which are in $V_{i+1}$. Let $G$ be the graph obtained from $T$ by adding an edge between every pair of vertices with a common parent. A greedy algorithm shows that $G$ is $\Delta$-colourable. This colouring of $G$ satisfies the property in \cref{ShadowCompleteWalk}. Thus $\sigma(T) \leq 4\chi(G) \leq 4\Delta(T)$.
\end{proof}

\begin{open}
Is there a constant $c$ such that 
$\sigma(T) \leq \Delta(T)+c$ for every tree $T$?
\end{open}

\cref{TreeSigma,ProductRho} imply:

\begin{cor}
For every graph $G$ and every tree $T$, 
$$\rho(G \boxtimes T)\le 4\Delta(T) \, \rho(G) .$$ 
\end{cor}

\citet{FOOZ11} proved the following surprising result about the nonrepetitive list chromatic number of trees. 

\begin{thm}[\citep{FOOZ11}]
\label{TreesUnboundedChoose}
Trees have unbounded nonrepetitive list chromatic number.
\end{thm}


\cref{TreesUnboundedChoose} leads to the natural question: which classes of trees have bounded nonrepetitive list chromatic number $\pich$? \citet{GPZ11} proved an affirmative answer for paths (\cref{PathChoose}). More generally, \citet{DJKW16} proved that caterpillars have bounded $\pich$ (see Appendix~B in the arXiv version of \citep{DJKW16}). Since caterpillars are the graphs with pathwidth 1, \citet{DJKW16} asked whether trees or graphs with bounded pathwidth have bounded $\pich$. These questions were completely answered by \citet{GJKM16}, who showed that trees with bounded pathwidth have bounded $\pich$, but this result does not extend to general graphs, by constructing a family of pathwidth-2 graphs with unbounded $\pich$.

While $O(\Delta^2)$ is an almost tight upper bound on the nonrepetitive list chromatic number of graphs with maximum degree $\Delta$, it is interesting to ask for which graph classes is there a $O(\Delta^{2-\epsilon})$ bound for some fixed $\epsilon>0$. \citet{KM13} proved such a result for trees. 

\begin{thm}[\citep{KM13}] 
\label{TreeChooseDelta}
For any fixed $\epsilon>0$, and for every tree $T$ with maximum degree $\Delta$,
$$\pich(T) \leq O(\Delta^{1+\epsilon}).$$
\end{thm}


\begin{open}
What is the maxmum nonrepetitive nonrepetitive list chromatic number of  trees with maximum degree $\Delta$. The best lower bound, due to \citet{FOOZ11}, is $\text{polylog}(\Delta)$. The best upper bound is	$O(\Delta^{1+\epsilon})$ due to \citet{KM13}.
\end{open}

\subsection{Treewidth}
\label{Treewidth}

Treewidth measures how similar a given graph is to a tree, and is particularly important in structural and algorithmic graph theory; see the surveys \citep{Bodlaender-TCS98, Reed-AlgoTreeWidth03,HW17}. It is defined as follows. A \emph{tree-decomposition} of a graph $G$ consists of a collection $\{B_x\subseteq V(G) : x\in V(T)\}$ of subsets of $V(G)$, called \emph{bags}, indexed by the vertices of a tree $T$, and with the following properties:
\begin{itemize}
\item for every vertex $v$ of $G$, the set $\{x\in V(T) : v\in B_x\}$ induces a non-empty (connected) subtree of $T$, and
\item for every edge $vw$ of $G$, there is a vertex $x\in V(T)$ for which $v,w\in B_x$.
\end{itemize}
The \emph{width} of such a tree-decomposition is $\max\{|B_x|:x\in V(T)\}-1$. 
The \emph{treewidth} of a graph $G$ is the minimum width of a tree-decomposition of $G$. Tree-decompositions and tree-width were introduced by \citet{RS-II}, although several equivalent notions were previously studied in the literature. 

Tree-decompositions and treewidth are closely related to chordal graphs. A graph is \emph{chordal} if there is no chordless cycle of length greater than 3. It is well-known that a graph is chordal if and only if it has a tree-decomposition in which each bag is  a clique \citep{Diestel5}. It follow that a graph has treewidth $k$ if and only if $G$ is a subgraph of a chordal graph with no clique on $k+2$ vertices \citep{Diestel5}. We will need the following result.

\begin{lem}[\citep{KP-DM08,DMW05}]
\label{ShadowCompleteChordal}
Every BFS-layering $(V_0,\dots,V_n)$ of a connected chordal graph $G$ is shadow-complete.
\end{lem}

\begin{proof}
Say $V_0=\{r\}$. Let $H$ be a connected component of $G[V_i]$ for some $i\in\{1,\dots,n\}$. Let $X$ be the set of vertices in $V_{i-1}$ adjacent to some vertex in $H$. Suppose for the sake of contradiction that distinct vertices $v,w\in X$ are not adjacent. There is a walk from $v$ to $w$ through $r$ in $G[V_0\cup\dots\cup V_i]$. Thus there is a shortest path $P$ from $v$ to $w$ in $G[V_0\cup\dots\cup V_i]$. Since $vw\not\in E(G)$, $P$ has at least one internal vertex. By definition $v$ has a neighbour $x$ in $H$, and $w$ has a neighbour $y$ in $H$. Since $H$ is connected, there is a path $Q$ from $y$ to $x$ in $H$. Choose $x$, $y$ and $Q$ to minimise the length of $Q$. It follows that $(P,w,y,Q,x,v)$ is a chordless cycle on at least four vertices in $G$. This contradiction shows that $X$ is a clique, and $(V_0,\dots,V_n)$ is shadow-complete. 
\end{proof}

\citet{BGKNP07} proved that every tree is nonrepetitively $4$-colourable. \citet{BV07} and \citet{KP-DM08} independently proved that graphs of bounded treewidh have bounded nonrepetitive chromatic number. The best bound is due to \citet{KP-DM08}, who showed that every graph with treewidth $k$ is nonrepetitively $4^k$-colourable. \citet{DEJWW20} showed that the proof of \citet{KP-DM08} actually gives the following  stronger result: 

\begin{thm}[\citep{DEJWW20}]
\label{TreewidthRho}
For every graph $G$ of treewidth $k$, 
$$\pi(G) \leq \rho(G)\le 4^k.$$
\end{thm}

\begin{proof}
The proof proceeds by induction on $k$. If $k=0$, then $G$ has
no edges, so assigning the same colour to all the vertices gives a stroll-nonrepetitive  colouring.
Now assume that $k\ge 1$. 
We may assume that $G$ is connected. 
Consider a tree-decomposition of $G$ of width at most $k$. 
By adding edges if necessary, we may assume that every bag of the tree-decomposition is a clique. 
Thus, $G$ is chordal with clique-number at most $k+1$.
Let $(V_0,V_1,\ldots)$ be a BFS-layering of $G$. By \cref{ShadowCompleteChordal}, $(V_0,\dots,V_n)$ is shadow-complete.
Moreover, the subgraph $G[V_i]$ of $G$ induced by each layer $V_i$ has treewidth at most $k-1$. This is clear for $i=0$ (since $k\geq 1$), and for $i\geq 1$ this follows from the fact that the graph $G[V_i]$ plus a universal vertex is a minor of $G$ (contract $V_0 \cup \cdots \cup V_{i-1}$ into a single vertex and remove $V_{i+1}, V_{i+2}, \dots$), and thus has treewidth at most $k$. Since removing a universal vertex decreases the treewidth by exactly one, it follows that $G[V_i]$ has treewidth at most $k-1$. 
By induction, $\rho( G[V_i]) \leq 4^{k-1}$ for each $i\in\{0,1,\dots,n\}$. 
By \cref{ShadowCompleteRho}, $\rho(G)\leq 4\cdot 4^{k-1}=4^k$. 
\end{proof}

\begin{open}[\citep{KP-DM08}]
\label{TreewidthPolynomial}
Is there an upper bound on $\pi(G)$ or $\rho(G)$ that is polynomial in $\tw(G)$? There is a quadratic lower bound,
since $\pi(G)\geq\st(G)$ and \citet{ACKKR04} proved that for each $k\in\NN$ there is a graph $G$ with $\tw(G)=k$ and  $\st(G)=\binom{k+2}{2}$. 
\end{open}

The following corollary of \cref{ProductRho,CompleteProductRho,TreewidthRho,} will be useful later. 
 
\begin{lem}[\citep{DEJWW20}]
\label{ProductGraphPathTreewidth}
For every graph $G$, path $P$ and integer $\ell$, 
$$\rho(G\boxtimes P \boxtimes K_\ell) \le \ell\, 4^{\tw(H)+1}.$$
\end{lem}

\subsection{Pathwidth}
\label{Pathwidth}

\citet{DJKW16} answered \cref{TreewidthPolynomial} in the affirmative for $\pi$ and pathwidth. The proof works for $\rho$. 

\begin{thm}[\citep{DJKW16}]
\label{PathwidthPi}
For every graph $G$ with pathwidth $k$,
$$\pi(G)\leq \rho(G) \leq 2k^2+6k+1.$$
\end{thm}




The proof of \cref{PathwidthPi} depends on the following helpful way to think about graphs of bounded pathwidth\footnote{\citet{DJKW16} presented \cref{Attack} in terms of the lexicographical product $P_m\cdot K_{k+1}$, which equals $P_m\boxtimes K_{k+1}$.}.

\begin{lem}[\citep{DJKW16}]
  \label{Attack}
  Every graph $G$ with pathwidth $k$ contains pairwise disjoint sets
  $B_1,\dots,B_m$ of vertices, such that: 
  \begin{itemize}
  \item no two vertices in distinct $B_i$ are adjacent,
  \item $G[B_i]$ has pathwidth at most $k-1$ for each $i\in\{1,\dots,m\}$, and
  \item if $H$ is the graph obtained from $G$ by deleting $B_i$ and
    adding a clique on $N_G(B_i)$ for each $i\in\{1,\dots,m\}$, then $H$ is isomorphic to a subgraph of $P_m\boxtimes K_{k+1}$.
  \end{itemize}
\end{lem}

\begin{proof}
  Consider a path decomposition $\mathcal{D}$ of $G$ with width $k$.
  Let $X_1,\dots,X_m$ be the set of bags in $\mathcal{D}$, such that
  $X_1$ is the first bag in $\mathcal{D}$, and for each $i\geq2$, the
  bag $X_i$ is the first bag in $\mathcal{D}$ that is disjoint from
  $X_{i-1}$. Thus $X_1,\dots,X_m$ are pairwise disjoint.  For
  $i\in[1,m]$, let $B_i$ be the set of vertices that only appear in
  bags strictly between $X_i$ and $X_{i+1}$ (or strictly after $X_m$
  if $i=m$). By construction, each such bag intersects $X_i$. Hence
  $G[B_i]$ has pathwidth at most $k-1$. Since each $X_i$ separates
  $B_{i-1}$ and $B_{i+1}$ (for $i\neq m$), no two vertices in distinct
  $B_i$ are adjacent. Moreover, the neighbourhood of $B_i$ is
  contained in $X_i\cup X_{i+1}$ (or $X_i$ if $i=m$).  Hence the graph
  $H$ (defined above) has vertex set $X_1\cup\cdots\cup X_{m}$ where
  $X_i\cup X_{i+1}$ is a clique for each $i\in[1,m-1]$.  Since
  $|X_i|\leq k+1$, the graph $H$ is isomorphic to a subgraph of $P_{m}\boxtimes K_{k+1}$.
\end{proof}

\begin{proof}[Proof of \cref{PathwidthPi}] 
  We proceed by induction on $k\geq 0$.  Every graph with pathwidth 0
  is edgeless, and is thus nonrepetitively 1-colourable, as desired.
  Now assume that $G$ is a graph with pathwidth $k\geq 1$.  Let
  $B_1,\dots,B_m$ be the sets that satisfy \cref{Attack}. Let $X:= B_1\cup\dots\cup B_m$. Since no two vertices in distinct $B_i$ are adjacent, $\pw(G[X])\leq k-1$. By induction, $\rho(G[X]) \leq 2(k-1)^2+6(k-1)+1$. Let $G'$ be the graph obtained from $G$ by adding a clique on $N_G(B_i)$ for each $i\in\{1,\dots,m\}$. By \cref{Attack}, $G'[ V(G)\setminus X ]$ is isomorphic to $P_{m} \boxtimes K_{k+1}$, which is stroll-nonrepetitively $4(k+1)$-colourable by \cref{ProductRho}. By construction, $(V(G)\setminus X,X)$ is a shadow-complete layering of $G'$. By the proof of \cref{ShadowCompleteRho} (using distinct sets of colours for $X$ and $V(G)\setminus X$), we have $\pi(G) \leq \rho(G) \leq \rho(G') \leq 4(k+1)+ 2(k-1)^2+6(k-1)-4=  2k^2+6k-4$. 
\end{proof}

\begin{open}[\citep{DJKW16}] 
What is the maximum nonrepetitive chromatic number of graphs with pathwidth $k$? The best knwon bounds are $\Omega(k)$ and $O(k^2)$. 	
\end{open}

Since graphs of pathwidth $k$ are $k$-degenerate, \cref{SigmaPiDegen,PathwidthPi} imply the following bounds on the walk-nonrepetitive chromatic number of graphs with given pathwidth.

\begin{cor}
	For every graph $G$ with pathwidth $k$, 
	$$\Delta(G)+1\leq \sigma(G) \leq (2k^2+6k+1)(k\,\Delta(G)+1).$$
\end{cor}

\subsection{Treewidth and Degree}
\label{TreewidthDegree}

\citet{BW08} proved the following polynomial bound on $\pi$ for graphs of bounded treewidth and maximum degree, thus solving \cref{TreewidthPolynomial} for bounded degree graphs. The same proof works for $\rho$. 

\begin{lem}[\citep{BW08}]
\label{TreewidthDegreeRhoPi}
For every graph $G$ with treewidth $k$ and maximum degree $\Delta$, 
\begin{align*}
\pi(G) \leq \rho(G) \leq 10(k+1)(\tfrac{7}{2}\Delta-1).
\end{align*}
\end{lem}

\begin{proof}
Let $\ell:= \floor{\frac{5}{2}(k+1)(\frac{7}{2}\Delta-1)}$. \citet{Wood09} proved\footnote{The proof is a minor improvement to a similar result by an anonymous referee of the paper by \citet{DO95}. The result in \citep{Wood09,DO95} is presented in terms of tree-partitions, which are easily seen to be equivalent to strong products of a tree and complete graph.} that $G$ is a subgraph of $T\boxtimes K_\ell$ for some tree $T$ with maximum degree at most $\ell\Delta$. By \cref{TreewidthRho}, $\rho(T)\leq 4$. Of course, $\sigma(K_\ell)=\ell$. 
By \cref{ProductRho}, 
\begin{equation*}
\rho(G) \leq \rho(T\boxtimes K_\ell) \leq \rho(T) \, \sigma(K_\ell) \leq 4\ell \leq 
10(k+1)(\tfrac{7}{2}\Delta-1).\qedhere
\end{equation*}
\end{proof}

Now consider walk-nonrepetitive colourings of graphs with given treewidth and given maximum degree. Every graph with treewidth $k$ is $k$-degenerate. Thus \cref{SigmaPiDegen,TreewidthRho} imply that graphs with bounded treewidth have $\Theta(\Delta)$ walk-nonrepetitive chromatic number. 

\begin{cor}
	For every graph $G$ with treewidth $k$ and maximum degree $\Delta$, 
	$$\Delta+1 \leq \sigma(G) \leq 4^k ( k\Delta+1 ).$$
\end{cor}

\cref{SigmaPiDegen,SigmaBounds,TreewidthDegreeRhoPi} imply the following polynomial bounds on $\sigma$:

\begin{lem}[\citep{BW08}]
	\label{TreewidthDegreeSigma}
	For every graph $G$ with treewidth $k$ and maximum degree $\Delta$, 
	\begin{align*}
	\Delta+1 \leq \sigma(G) \leq \rho(G)\, \chi(G^2) \leq  10(k+1)(\tfrac72 \Delta-1) \,\min\{ k\Delta+1,\Delta^2+1\}.
	\end{align*}
\end{lem}

These results lead to the following results for strong products. 
\cref{ProductRho,TreewidthDegreeSigma} imply:

\begin{cor}
\label{GraphTreewidthRho}
For every graph $G$ and for every graph $H$ with treewidth $k$ and maximum degree $\Delta$,  
$$\rho(G \boxtimes H)\leq \rho(G) \; 10(k+1)(\tfrac72 \Delta-1) \,\min\{ k\Delta+1,\Delta^2+1\}.$$
\end{cor}

\cref{TreewidthRho,GraphTreewidthRho} imply:

\begin{cor}
For every graph $G$ with treewidth $\ell$ and for every graph $H$ with treewidth $k$ and maximum degree $\Delta$,  
$$\rho(G \boxtimes H)\leq 4^\ell \; 10(k+1)(\tfrac72 \Delta-1) \,\min\{ k\Delta+1,\Delta^2+1\}.$$
\end{cor}

\cref{DeltaSquared,DegreeLowerBound} give tight bounds (up to a logarithmic factor) on the nonrepetitive list chromatic number of graphs with given maximum degree. However, the nonrepetitive list chromatic number of graphs with given maximum degree and bounded treewidth is wide open. 

\begin{open}
Are graphs of bounded treewidth and maximum degree $\Delta$
nonrepetitively $O(\Delta^{2-\epsilon})$-choosable, for some fixed $\epsilon>0$?
By \cref{TreeChooseDelta}, the answer is `yes' for trees. The treewidth 2 case is open. 
\end{open}

\subsection{Outerplanar Graphs}
\label{Outerplanar}

A graph is \emph{outerplanar} if it has a drawing in the plane with no crossing and with all the vertices on the boundary of a single face. Here we consider nonrepetitive colourings of outerplanar graphs. First, note the  following folklore result:

\begin{lem}
\label{OuterplanarLayering}
Every edge-maximal outerplanar graph has a shadow-complete layering $(V_0,V_1,\dots,V_n)$ such that for each $i\in\{0,1,\dots,n\}$  each connected component of $G[V_i]$ is a path. 
\end{lem}

\begin{proof}
Since $G$ is edge-maximal, $G$ is connected and chordal. 
Let $(V_0,V_1,\dots)$ be a BFS-layering of $G$.
By \cref{ShadowCompleteChordal}, $(V_0,V_1,\dots)$  is shadow-complete. 
For $i\geq 1$, let $G_i$ be the graph obtained from $G[V_0\cup\dots\cup V_i]$ by 
contracting $V_0\cup\dots\cup V_{i-1}$ into a single vertex $w$.  Since  outerplanarity is a minor-closed property and $G[V_0\cup\dots\cup V_{i-1}]$ is connected, $G_i$ is outerplanar. 
Since every vertex in $V_i$ has a neighbour in $V_{i-1}$, $w$ dominates $G_i$. 
If $G[V_i]$ contains a cycle $C$, then $G_i$ contains a $K_4$-minor, 
which is a contradiction since $K_4$ is not outerplanar. 
If $G[V_i]$ contains $K_{1,3}$, then $G_i$ contains $K_{2,3}$, 
which is a contradiction since $K_{2,3}$ is not outerplanar. 
Thus $G[V_i]$ is a forest with maximum degree at most 2. 
Hence, each connected component of $G[V_i]$ is a path.
\end{proof}

\cref{ShadowCompletePi,OuterplanarLayering} imply the following result independently due to \citet{BV08} and \citet{KP-DM08}:

\begin{thm}[\citep{BV08,KP-DM08}] 
\label{OuterplanarPi}
For every outerplanar graph $G$, 
$$\pi(G) \leq 12.$$
\end{thm}

\citet{BV07} also proved the following lower bound:

\begin{prop}[\citep{BV07}]
\label{OuterplanarLowerBound}
There exists an outerplanar graph $G$ with
$$\pi(G)\geq 7.$$
\end{prop}

\begin{proof}
Let $n$ be a sufficiently large integer. 
Let $T$ be the complete $n$-ary tree of height $n$. 
Let $G$ be obtained from $T$ by adding, for each non-leaf vertex $v$ of $T$, a path $P_v$ on the children of $v$. This can be done so that $G$ is outerplanar. 
Suppose for the sake of contradiction that there exists a nonrepetitive colouring $\phi$ of $G$ with colour-set $\{1,\dots,6\}$. 

\begin{claim}
For every vertex $v$ of $T$ that is neither a leaf nor the parent of a leaf, there is a set $C_v$ of four colours, each of which appears at least twice on $P_v$. 
\end{claim}

\begin{proof}
Suppose for the sake of contradiction that $\phi(v)=1$ but only three colours, say $2,3,4$, appear on $P_v$. 
All three colours appear on any four consecutive vertices of $P_v$. 
Thus, for sufficiently large $n$, each of $2,3,4$ appears at least twice on $P_v$. 

If this 3-colouring of $P_v$ is distance-2, then the sequence of colours on $P_v$, without loss of generality, starts $234234$, which is a repetitively coloured 6-vertex path. 
Thus the colouring of $P_v$ is not distance~2. 
Hence $P_v$ contains a subpath $(a,b,c,d,e,f,g)$, where without loss of generality, $\phi(a)=\phi(c)=3$ and $\phi(b)=2$ (since $n$ is sufficiently large). 

Let $x$ be any vertex of $P_c$. If $\phi(x)=1$ then $xcva$ is coloured $1313$. 
If $\phi(x)=2$ then $xcba$ is coloured $2323$. 
If $\phi(x)=3$ then $xc$ is coloured $33$. 
Thus $P_c$ is coloured by $4,5,6$ and each colour appears since $n$ is sufficiently large.

Now, $\phi(d)\in\{2,3,4\}$ since $d\in V(P_v)$.
If $\phi(d)=2$ then $abcd$ is coloured $3232$.
If $\phi(d)=3$ then $cd$ is coloured $33$.
Thus $\phi(d)=4$.
Let $x$ be any vertex of $P_d$. 
If $\phi(x)=1$ then $xdvy$ is coloured $1414$ for some $y\in V(P_v)\setminus \{d\}$ coloured 4. 
If $\phi(x)=3$ then $xdcy$ is coloured $3434$ for some $y\in V(P_c)$ coloured 4. 
If $\phi(x)=4$ then $xd$ is coloured $44$. 
Thus $P_d$ is coloured by $2,5,6$ and each colour appears since $n$ is sufficiently large.

Now, $\phi(e)\in\{2,3,4\}$ since $e\in V(P_v)$.
If $\phi(e)=3$ then $edcy$ is coloured $3434$ for some $y\in V(P_c)$ coloured 4. 
If $\phi(e)=4$ then $de$ is coloured $44$.
Thus $\phi(e)=2$.
Let $x$ be any vertex of $P_e$. 
If $\phi(x)=1$ then $xevb$ is coloured $1212$. 
If $\phi(x)=2$ then $xe$ is coloured $22$. 
If $\phi(x)=4$ then $xedy$ is coloured $4242$ for some vertex $y\in V(P_d)$ coloured 2. 
Thus $P_e$ is coloured by $3,5,6$  and each colour appears since $n$ is sufficiently large.

Now, $\phi(f)\in\{2,3,4\}$ since $f\in V(P_v)$.
If $\phi(f)=2$ then $ef$ is coloured $22$.
If $\phi(f)=4$ then $fedy$ is coloured $4242$ for some $y\in V(P_d)$ coloured 2. 
Thus $\phi(f)=3$.
Let $x$ be any vertex of $P_f$. 
If $\phi(x)=1$ then $xfvc$ is coloured $1313$. 
If $\phi(x)=2$ then $xfey$ is coloured $2323$  for some $y\in V(P_e)$ coloured 3. 
If $\phi(x)=3$ then $xf$ is coloured $33$. 
Thus $P_f$ is coloured by $4,5,6$ and each colour appears since $n$ is sufficiently large.

Now, $\phi(g)\in\{2,3,4\}$ since $g\in V(P_v)$.
If $\phi(g)=3$ then $fg$ is coloured $33$.
If $\phi(g)=2$ then $gfey$ is coloured $2323$ for some $y\in V( P_e)$ coloured 3. 
Thus $\phi(g)=4$.

Therefore the subpath $(b,c,d,e,f,g)$ is coloured $234234$. This contradiction shows that $P_v$ is assigned at least four colours.

At this point we have assumed that $n\geq n_0$ for some fixed number $n_0$. Taking $n\geq 6n_0$, we can partition $P_v$ into six disjoint subpaths each with $n_0$ vertices, and by the above argument, at least four distinct colours appear in each subpath. Since at most five colours appear on $P_v$, on at least two of these subpaths the same set of four colours appears. This completes the proof. 
\end{proof}

Let $u$ be a vertex of $T$ that is neither a leaf nor the parent of a leaf. Let $v$ be a vertex in $P_u$ with $\phi(v)\in C_u$. Let $x$ be any vertex in $P_v$. If $\phi(x)=\phi(v)$ then $xv$ is repetitively coloured. If $\phi(x)=\phi(u)$ then $xvuy$ is repetitively coloured for some vertex $y\in V(P_u)\setminus\{v\}$ coloured $\phi(v)$. Such a vertex $y$ exists by the claim and since $\phi(v)\in C_u$. Hence $C_v=\{1,\dots,6\}\setminus\{\phi(v),\phi(u)\}$. 

Let $r$ be the root of $T$. 
By the claim, without loss of generality, $\phi(r)=1$ and $C_r=\{2,3,4,5\}$. 
Let $a$ be a vertex in $P_r$ coloured 2; thus $C_a=\{3,4,5,6\}$. 
Let $b$ be a vertex in $P_a$ coloured 3; thus $C_b=\{1,4,5,6\}$. 
Let $c$ be a vertex in $P_b$ coloured 1. 
Let $d$ be a vertex in $P_r$ coloured 3; thus $C_d=\{2,4,5,6\}$. 
Let $e$ be a vertex in $P_d$ coloured 2. 
These vertices exist and are neither leaves nor parent of leaves, since $T$ has sufficiently large height. 
Now $(c,b,a,r,d,e)$ is a path in $T$ coloured $132132$. 
This contradiction completes the proof.
\end{proof}

Note that the proof of \cref{OuterplanarLowerBound} actually shows that there is an outerplanar graph $G$ such that every 6-colouring has a repetitively coloured path on 2, 4 or 6 vertices.

\begin{open}
\label{OuterplanarPiOpen}
What is the maximum nonrepetitive chromatic number of an outerplanar graph? The answer is in $\{7,8,\dots,12\}$. This question may have a bearing on \cref{TreewidthPolynomial}.
\end{open}

For stroll-nonrepetitive colourings, \cref{ShadowCompleteRho,OuterplanarLayering} imply:

\begin{thm}
\label{OuterplanarRho}
For every outerplanar graph $G$, 
$$\rho(G) \leq 16.$$
\end{thm}

\begin{open}
\label{OuterplanarRhoOpen}
What is the maximum stroll-nonrepetitive chromatic number of an outerplanar graph? The answer is in $\{7,8,\dots,16\}$. Any improvement to the upper bound would have a bearing on \cref{PlanarOpen}. If $\rho(P)\leq 3$ for each path $P$ (\cref{PathRhoOpen}) then $\rho(G) \leq 12$ for every outerplanar graph $G$. 
\end{open}

\citet{LW06} proved that $\chi(G^2)\leq\Delta(G)+2$ for every outerplanar graph $G$. 
\cref{SigmaBounds,OuterplanarRho} thus imply that the walk-nonrepetitive chromatic number of outerplanar graphs is $\Theta(\Delta)$. 

\begin{cor}
For every outerplanar graph $G$ with maximum degree $\Delta$, 
$$\Delta+1 \leq \sigma(G)\leq 16\Delta+32.$$
\end{cor}

\section{Planar Graphs and Beyond}
\label{PlanarAndBeyond}

\subsection{Planar Graphs}
\label{Planar}

\citet{AGHR02} first asked whether planar graphs have bounded nonrepetitive chromatic number. For several years, this problem was widely recognised as the most important open problem in the field of nonrepetitive graph colouring. The first non-trivial upper bound was due to \citet*{DFJW13}, who proved that $\pi(G)\leq O(\log |V(G)|)$ for all planar graphs $G$. The above question was solved by \citet{DEJWW20}. Much of the above machinery involving strong products was developed as a tool to answer the question for planar graphs. 

\begin{thm}[\citep{DEJWW20}]
\label{PlanarRho}
For every planar graph $G$, $$\pi(G) \leq \rho(G)\leq 768.$$
\end{thm}

\begin{proof}
\citet{DJMMUW19,DJMMUW20} proved that every planar graph $G$ is a subgraph of $H\boxtimes P\boxtimes K_3$ for some graph $H$ with treewidth 3 and path $P$. 
By \cref{CompleteProductRho}, $\rho(G) \leq \rho(H\boxtimes P\boxtimes K_3) 
\leq 3\, \rho(H\boxtimes P)$, which is at most $12\, \rho(H)$ by \cref{PathProductRho}. Since $H$ has treewidth at most 3, $\rho(G) \leq  12\cdot 4^3= 768$ by \cref{TreewidthRho}.
\end{proof}

We now present the best known lower bound on the nonrepetitive chromatic number of planar graphs. 

\begin{prop}[Pascal Ochem; see \citep{DFJW13}]
There is a planar graph $G$ such that $\pi(G)\geq 11$. 
\end{prop}

\begin{proof}
By \cref{OuterplanarLowerBound}, there is an outerplanar graph $H$ with $\pi(H)\geq 7$. Let $G$ be the following planar graph. Start with a path $P=(v_1,\dots,v_{22})$. Add two adjacent vertices $x$ and $y$ that both dominate $P$. Let each vertex $v_i$ in $P$ be adjacent to every vertex in a copy $H_i$ of $H$. Suppose on the contrary that $G$ is nonrepetitively $10$-colourable. Without loss of generality, $x$ and $y$ are respectively coloured $1$ and $2$. A vertex in $P$ is \emph{redundant} if its colour is used on some other vertex in $P$. If no two adjacent
vertices in $P$ are redundant then at least $11$ colours appear exactly once on $P$, which is a contradiction. Thus some pair of consecutive vertices $v_i$ and $v_{i+1}$ in $P$ are redundant. Without loss of generality, $v_i$ and $v_{i+1}$ are respectively coloured $3$ and $4$. If some vertex in $H_i\cup H_{i+1}$ is coloured $1$ or $2$,
then since $v_i$ and $v_{i+1}$ are redundant, with $x$ or $y$ we have a repetitively coloured path on 4 vertices. Now assume that no vertex
in $H_i\cup H_{i+1}$ is coloured $1$ or $2$. If some vertex in $H_i$
is coloured $4$ and some vertex in $H_{i+1}$ is coloured $3$, then
with $v_i$ and $v_{i+1}$, we have a repetitively coloured path on 4
vertices. Thus no vertex in $H_i$ is coloured $4$ or no vertex in
$H_{i+1}$ is coloured $3$. Without loss of generality, no vertex in
$H_i$ is coloured $4$. Since $v_i$ dominates $H_i$, no vertex in $H_i$
is coloured $3$. We have proved that no vertex in $H_i$ is coloured
$1,2,3$ or $4$, which is a contradiction, since $\pi(H_i)\geq
7$. Therefore $\pi(G)\geq 11$.
\end{proof}

\begin{open}
\label{PlanarOpen}
What is the maximum nonrepetitive chromatic number of a planar graph? The answer is in $\{11,\dots,768\}$. Note that in the proof of \cref{PlanarRho}, since $H$ is planar with treewidth 3, to improve the bound of 768 to 576 it would suffice to show that $\rho(P)\leq 3$ for each path $P$, or $\rho(Q)\leq 12$ for each outerplanar graph $Q$.
\end{open}

We briefly mention that several papers studied colourings of plane graphs in which only facial paths are required to be nonrepetitively coloured \citep{HJSS11,Przybyo14,Przybyo13,BC13,JS12,JS09,JS09,Gutowski18,BDMR17,CJS17}. \citet{BC13} proved that every plane graph is facially-nonrepetitively $24$-colourable, and \citet{Gutowski18} proved that every plane graph is facially-nonrepetitively list $O(1)$-colourable. This latter result is in sharp contrast to \cref{TreesUnboundedChoose}, which says that trees have unbounded nonrepetitive list chromatic number.

\cref{PlanarRho} leads to a $\Theta(\Delta)$ bound on the walk-nonrepetitive chromatic number of planar graphs. \citet{HM03} proved that $\chi(G^2)\leq 2\Delta(G)+25$ for every planar graph $G$. Thus \cref{PlanarRho,SigmaBounds}  implies:

\begin{cor}
For every planar graph $G$, 
$$\Delta(G)+1 \leq \sigma(G) \leq 1536\; \Delta(G) + 19200.$$
\end{cor}



The above result for planar graphs can be combined with other results in various ways. For example, \cref{TreewidthDegreeSigma,PlanarRho} imply:

\begin{cor}
\label{PlanarTreewidthRho}
For every planar graph $G$ and for every graph $H$ with treewidth $k$ and maximum degree $\Delta$,  
$$\rho(G \boxtimes H)\leq \rho(G) \,\sigma(H) \leq 
7680\, (k+1)(\tfrac72 \Delta-1) \,\min\{ k\Delta+1,\Delta^2+1\}.$$
\end{cor}

\subsection{Graphs on Surfaces}
\label{Surfaces}

\citet{DEJWW20} proved the following generalisation of \cref{PlanarRho} for graphs of bounded Euler genus.

\begin{thm}[\citep{DEJWW20}]
\label{GenusRho}
For every graph $G$ with Euler genus $g$,
$$\pi(G) \leq \rho(G) \leq 256 \max\{2g,3\}.$$
\end{thm}

\begin{proof}
\citet{DJMMUW19,DJMMUW20} proved that every graph $G$ of Euler genus $g$ is a subgraph of $H \boxtimes P \boxtimes K_{\max\{2g,3\}}$ for some graph $H$ with treewidth at most $3$ and some path $P$. Thus \cref{ProductRho,TreewidthRho} imply 
\begin{align*}
\pi(G) \leq \rho(G) 
\leq \rho( H \boxtimes P \boxtimes K_{\max\{2g,3\}} ) 
\leq \max\{2g,3\}  \cdot \rho( H\boxtimes P ) 
& \leq \max\{2g,3\}  \cdot 4 \cdot \rho( H ) \\
& \leq  \max\{2g,3\}  \cdot 4^4\\
& =  256 \max\{2g,3\}.\qedhere
\end{align*}
\end{proof}

\citet{AEH13} proved that for all $\epsilon>0$ and $g\geq 0$, for sufficiently large $\Delta$, every graph $G$ with Euler genus $g$ and maximum degree at most $\Delta$ satisfies $\chi(G^2)\leq (\frac32 + \epsilon) \Delta$. Thus \cref{GenusRho,SigmaBounds}  implies:

\begin{cor}
For every graph $G$ of Euler genus $g$, 
$$\Delta(G)+1 \leq \sigma(G) \leq O( g \Delta (G) ) .$$
\end{cor}

\begin{open}
What is the maximum nonrepetitive chromatic number of a graph with Euler genus $g$? \cref{GenusRho} proves a $O(g)$ upper bound. The best known lower bound follows from \cref{DegreeLowerBound} [Louis Esperet, personal communication, 2020]. In particular, \cref{DegreeLowerBound} implies there is a graph $G$ with $m\leq O(n^{3/2} \log^{1/2} n)$ edges and $\pi(G)\geq \Omega(n)$. Say $G$ has Euler genus $g$. Then $g \leq m$ and $\pi(G) \geq \Omega(n) \geq  \Omega( m^{2/3} / \log^{1/3} m)\geq \Omega( g^{2/3} / \log^{1/3} g)$.
\end{open}

%
%
%
%

%

\subsection{Minor-Closed Classes}
\label{MinorClosedClass}

This section proves the result of \citet{DEJWW20} that graphs excluding a fixed minor or fixed topological minor have bounded nonrepetitive chromatic number. The same method shows the analogous result for stroll-nonrepetitive chromatic number. The proof employs the following  graph minor structure theorem of \citet{RS-XVI}. A \emph{torso} of a tree-decomposition is a graph induced by a bag augmented with a clique on each intersection of that bag with another bag of the tree-decomposition. 

\begin{thm}[\citep{RS-XVI}]
\label{GMST}
For every graph $X$, there is an integer $k\geq 1$ such that every $X$-minor-free graph has a tree-decomposition in which each torso is $k$-almost-embeddable.
\end{thm}

We omit the definition of $k$-almost embeddable from this paper, since we do not need it. All we need to know is the following theorem of \citet{DJMMUW19,DJMMUW20}, where $A+B$ is the complete join of graphs $A$ and $B$. 

\begin{thm}[\citep{DJMMUW19,DJMMUW20}] 
\label{AlmostEmbeddableStructure}
Every $k$-almost embeddable graph is a subgraph of 
$$K_k + ( H\boxtimes P \boxtimes K_{\max\{6k,1\}} )$$ 
for some graph $H$ with treewidth at most $11k+10$ and for some path $P$. 
\end{thm}

\begin{lem}[\citep{DEJWW20}]
\label{AlmostEmbeddablePi}
For every $k$-almost embeddable graph $G$, 
$$\pi(G) \leq \rho(G)  \leq k+ 6k \cdot 4^{11(k+1)}.$$
\end{lem}
\begin{proof}
Observe that $\rho(G+K_k)=\rho(G)+k$ for every graph $G$ and integer $k\ge 0$. Thus \cref{AlmostEmbeddableStructure,ProductRho,TreewidthRho} imply that for every $k$-almost embeddable graph $G$, 
\begin{align*}
\pi(G) \leq \rho(G) 
& \leq \rho( K_k + ( H\boxtimes P \boxtimes K_{\max\{6k,1\}}) )  \\
& \leq k+ 6k \, \rho(H\boxtimes P)  \\
& \leq k+ 6k \cdot 4  \rho(H)  \\
& \leq k+ 6k \cdot 4^{11(k+1)}.\qedhere
\end{align*}
\end{proof}

A tree-decomposition $(B_x :x\in V(T))$ of a graph $G$ has adhesion $r$ if $|B_x\cap B_y|\leq r$ for each edge $xy\in E(T)$. \citet{DMW17} proved the following useful result in the case of $\pi$. The same proof works for $\rho$. We delay the proof of \cref{RichColour} until \cref{RichColourSubsection}.

\begin{lem}[\citep{DMW17}] 
\label{RichColour}
Let $G$ be a graph that has a tree-decomposition with adhesion $r$. 
Then 
\begin{align*}
(a) \quad  \pi(G)  &\leq  4^r \,\max_H \pi(H)\\
(b) \quad \rho(G)  &\leq  4^r \,\max_H \rho(H).
\end{align*}
where both maximums are taken over the torsos $H$ of the tree-decomposition. 
\end{lem}

The following theorem of \citet{DEJWW20} confirms a conjecture of \citet{Grytczuk-DM08,Gryczuk-IJMMS07}.

\begin{thm}[\citep{DEJWW20}]
\label{MinorPiRho}
For every graph $X$, there is an integer $c$ such that for every $X$-minor-free graph $G$,
$$\pi(G) \leq \rho(G) \leq c.$$
\end{thm}

\begin{proof}
\cref{AlmostEmbeddablePi,GMST} imply that every $X$-minor-free graph $G$ has a tree-decomposition $(B_x:x\in V(T))$ such that each torso is nonrepetitively $c$-colourable (where $c:= k+ 6k \cdot 4^{11(k+1)}$). For each edge $xy\in E(T)$, since $B_x\cap B_y$ induces a clique in each torso, the adhesion of $(B_x:x\in V(T))$ is at most $c$. By \cref{RichColour}, $\pi(G)  \leq \rho(G) \leq c\,4^c$, as desired. 
\end{proof}

\begin{open}
What is the maximum nonrepetitive chromatic number of $K_t$-minor-free graphs? Since the above proof depends on the Graph Minor Structure Theorem (\cref{GMST}) the constant in \cref{MinorPiRho} is huge. It would be very interesting to prove \cref{MinorPiRho} without using the Graph Minor Structure Theorem. 
\end{open}

$K_t$-minor-free graphs are $O(t\sqrt{\log t})$-degenerate \citep{Kostochka82,Kostochka84,Thomason84,Thomason01}. Thus \cref{SigmaPiDegen,MinorPiRho} implies the following bounds on the walk-nonrepetitive chromatic number of graphs excluding a fixed minor. 

\begin{thm}
\label{MinorSigma}
For every graph $X$, there is an integer $c$ such that 
for every $X$-minor-free graph $G$ with maximum degree $\Delta$,
$$\Delta+1 \leq \sigma(G) \leq c \Delta.$$
\end{thm}

To obtain  results for graphs excluding a topological minor we use the following version of the structure theorem of \citet{GM15}.

\begin{thm}[\citep{GM15}] 
\label{GroheMarx}
For every graph $X$, there is a constant $k$ such that every graph excluding $X$ as a topological minor has a tree-decomposition such that each torso is $k$-almost-embeddable or has at most $k$ vertices with degree greater than $k$.
\end{thm}

\begin{thm}
\label{TopoMinorPi}
For every graph $X$, there is an integer $c$ such that 
for every $X$-topological-minor-free graph $G$,
$$\pi(G) \leq c.$$
\end{thm}

\begin{proof}
\cref{AlmostEmbeddablePi} says that every $k$-almost embeddable graph is nonrepetitively $c_1$-colourable, where $c_1:= k+ 6k \cdot 4^{11(k+1)}$. Equation~\cref{DeltaSquared} implies that if a graph has at most $k$ vertices with degree greater than $k$, then it is nonrepetitively $c_2$-colourable, where $c_2 := k^2+O(k^{5/3}) + k$. Let $c := \max\{c_1,c_2\}$. \cref{GroheMarx} implies that every topological-minor-free graph $G$ has a tree-decomposition $(B_x:x\in V(T))$ such that each torso is nonrepetitively $c$-colourable. For each edge $xy\in E(T)$, since $B_x\cap B_y$ induces a clique in each torso, the adhesion of $(B_x:x\in V(T))$ is at most $c$. By \cref{RichColour}, $\pi(G)  \leq  c\,4^c$, as desired. 
\end{proof}

%
%

\subsection{Proof of \cref{RichColour}}
\label{RichColourSubsection}

A tree decomposition $(B_x\subseteq V(G):x\in V(T))$ of a graph $G$ is \emph{$k$-rich} if $B_x\cap B_y$ is a clique in $G$ on at most $k$ vertices, for each edge $xy\in E(T)$. Rich tree decomposition are implicit in the graph minor structure theorem, as demonstrated by the following lemma.

\begin{lem}[\citep{DMW17}]
\label{ProduceRichDecomp}
For every fixed graph $H$ there are constants $k\geq 1$ and $\ell\geq 1$ depending only on $H$, such that every $H$-minor-free graph $G_0$ is a spanning subgraph of a graph $G$ that has a $k$-rich tree decomposition such that each bag induces an $\ell$-almost-embeddable subgraph of $G$. 
\end{lem}

\begin{proof}
By \cref{GMST}, there is a constant $\ell=\ell(H)$ such that $G_0$ has a tree decomposition $\mathcal{T}:=(B_x\subseteq V(G):x\in V(T))$ in which each torso is $\ell$-almost-embeddable. Let $G$ be the graph obtained from $G$ by adding a clique on  $B_x\cap B_y$ for each edge $xy\in E(T)$. Let $\mathcal{T'}$ be the tree decomposition of $G$  obtained from $\mathcal{T}$. Each bag of $\mathcal{T'}$ is the torso of the corresponding bag of $\mathcal{T}$, and thus induces an $\ell$-almost-embeddable subgraph of $G$. \citet{DMW17} observed that there is a constant $k$ depending only on $\ell$ such that every clique in an $\ell$-almost embeddable graph has size at most $k$. Thus $\mathcal{T'}$ is a $k$-rich tree decomposition of $G$. 
\end{proof}


The following lemma by \citet{DMW17} generalises \cref{ShadowCompleteChordal}. For a subgraph $H$ of a graph $G$, a tree decomposition  $(C_y\subseteq V(H):y\in V(F))$  of  $H$ is  \emph{contained in} a tree decomposition  $(B_x\subseteq V(G):x\in V(T))$ of $G$ if for each bag $C_y$ there is a bag $B_x$ such that $C_y\subseteq B_x$. 

\begin{lem}[\citep{DMW17}]
\label{RichShadow}
Let $G$ be a graph with a  $k$-rich tree decomposition $\mathcal{T}$ for some $k\geq 1$. Then $G$ has a shadow-complete layering $(V_0,V_1,\dots,V_t)$ such that every shadow has size at most $k$, and for each $i\in\{0,\dots,t\}$, the subgraph $G[V_i]$ has a  $(k-1)$-rich tree decomposition contained in $\mathcal{T}$.
\end{lem}

\begin{proof} 
\wish{Can we simplify this proof using \cref{ShadowCompleteChordal}?}
We may assume that $G$ is connected with at least one edge. Say  $\mathcal{T}=(B_x\subseteq V(G):x\in V(T))$  is a $k$-rich tree decomposition of $G$. If $B_x\subseteq B_y$ for some edge $xy\in E(T)$, then contracting $xy$ into $y$ (and keeping bag $B_y$) gives a new $k$-rich tree decomposition of $G$. Moreover, if a tree decomposition of a subgraph of $G$ is contained in the new tree decomposition of $G$, then it is contained in the original. Thus we may assume that  $B_x\not\subseteq B_y$ and $B_y\not\subseteq B_x$ for each edge $xy\in V(T)$.

Let $G'$ be the graph obtained from $G$ by adding an edge between every pair of vertices in a common bag (if the edge does not already exist). Let $r$ be a vertex of $G$. Let $\alpha$ be a node of $T$ such that $r\in B_\alpha$. Root $T$ at $\alpha$. Now every non-root node of $T$ has a parent node. Since $G$ is connected, $G'$ is connected. For $i\geq 0$, let $V_i$ be the  set of vertices of $G$ at distance $i$ from $r$ in $G'$. Thus, for some $t$,  $(V_0,V_1,\dots,V_t)$ is a layering of $G'$ and also of $G$ (since $G\subseteq G'$). 

\wish{In $G'$, each bag is a clique, so $G'$ is chordal. By \cref{ShadowCompleteChordal},  $(V_0,V_1,\dots,V_t)$ is a shadow-complete layering of $G'$.}

Since each bag $B_x$ is a clique in $G'$,  $V_1$ is the set of vertices of $G$ in bags that contain $r$ (not including $r$ itself). More generally, $V_i$ is the set of vertices $v$ of $G$ in bags that intersect $V_{i-1}$ such that $v$ is not in $V_0\cup\dots\cup V_{i-1}$.

Define $B'_\alpha:=B_\alpha\setminus\{r\}$ and $B''_\alpha:=\{r\}$. For a non-root node $x\in V(T)$ with parent node $y$, define $B'_x:=B_x\setminus B_y$ and $B''_x:=B_x\cap B_y$. Since $B_x\not\subseteq B_y$, it follows that $B'_x\neq\emptyset$.  One should think that $B'_x$ is the set of vertices that first appear in $B_x$ when traversing down the tree decomposition from the root, while $B''_x$ is the set of vertices in $B_x$ that appear above $x$ in the tree decomposition. 

Consider a node $x$ of $T$. Since $B_x$ is a clique in $G'$, $B_x$ is contained in at most two consecutive layers. Consider (not necessarily distinct) vertices $u,v$ in the set $B'_x$, which is not empty. Then the distance between $u$ and $r$ in $G'$ equals the distance between $v$ and $r$ in $G'$. Thus $B'_x$ is contained in one layer, say $V_{\ell(x)}$. Let $w$ be the neighbour of $v$ in some shortest path between $v$ and $r$ in $G'$. Then $w$ is in $B''_x\cap V_{\ell(x)-1}$. In conclusion, each bag $B_x$ is contained in precisely two consecutive layers, $V_{\ell(x)-1}\cup V_{\ell(x)}$, such that $\emptyset\neq B'_x\subseteq V_{\ell(x)}$ and $B_x\cap V_{\ell(x)-1}\subseteq B''_x\neq\emptyset$. Also, observe that if $y$ is an ancestor of $x$ in $T$, then $\ell(y)\leq\ell(x)$. Call this property $(\star)$. 

We now prove that $G[V_i]$ has the desired $(k-1)$-rich tree decomposition. Since $G[V_0]$ has one vertex and no edges, this is trivial for $i=0$. Now assume that $i\in\{1,\dots,t\}$.

Let $T_i$ be the subgraph of $T$ induced by the nodes $x$ such that $\ell(x)\leq i$. By property $(\star)$, $T_i$ is a (connected) subtree of $T$. We claim that $\mathcal{T}_i:=(B_x\cap V_i:x\in V(T_i))$ is a $T_i$-decomposition of $G[V_i]$. First we prove that each vertex $v\in V_i$ is in some bag of $\mathcal{T}_i$. Let $x$ be the node of $T$ closest to $\alpha$ such that $v\in B_x$. Then $v\in B'_x$ and $\ell(x)=i$. Hence $v$ is in the bag $B_x\cap V_i$ of $\mathcal{T}_i$, as desired. 

Now  we prove that for each edge $vw\in E(G[V_i])$, both $v$ and $w$ are in a common bag of $\mathcal{T}_i$. Let $x$ be the node of $T$ closest to $\alpha$ such that $v\in B_x$. Let $y$ be the node of $T$ closest to $\alpha$ such that $w\in B_y$. Thus $v\in B'_x$ and $x\in V(T_i)$, and $w\in B'_y$ and $y\in V(T_i)$. Since $vw\in E(G)$, there is a bag $B_z$ containing both $v$ and $w$, and $z$ is a descendant of both $x$ and $y$ in $T$ (by the definition of $x$ and $y$). Without loss of generality, $x$ is on the $y\alpha$-path in $T$. Moreover, $v$ is also in $B_y$ (since $v$ and $w$ are in a common bag of $\mathcal{T}$). Thus $v$ and $w$ are in the bag  $B_y\cap V_i$ of $\mathcal{T}_i$, as desired. 

Finally, we prove that for each vertex $v\in V_i$, the set of bags in $\mathcal{T}_i$ that contain $v$ correspond to a (connected) subtree of $T_i$. By assumption, this property holds in $T$. Let $X$ be the subtree of $T$ whose corresponding bags in $\mathcal{T}$ contain $v$. Let $x$ be the root of $X$. Then $v\in B'_x$ and $\ell(x)=i$. By property $(\star)$, $\ell(z)\geq i$ for each node $z$ in $X$. Moreover, again by property $(\star)$, deleting from $X$ the nodes $z$ such that $\ell(z)\geq i+1$ gives a connected subtree of $X$, which is precisely the subtree of $T_i$ whose bags in $\mathcal{T}_i$ contain $v$. 

Hence $\mathcal{T}_i$ is a $T_i$-decomposition of $G[V_i]$. By definition, $\mathcal{T}_i$ is contained in $\mathcal{T}$. 

We now prove that $\mathcal{T}_i$ is $(k-1)$-rich. Consider an edge $xy\in E(T_i)$. Without loss of generality, $y$ is the parent of $x$ in $T_i$. Our goal is to prove that $B_x \cap B_y\cap V_i=B''_x\cap V_i$ is a clique on at most $k-1$ vertices. Certainly, it is a clique on at most $k$ vertices, since $\mathcal{T}$ is $k$-rich. Now, $\ell(x)\leq i$ (since $x\in V(T_i)$). If $\ell(x)<i$ then $B_x\cap V_i=\emptyset$, and we are done. Now assume that $\ell(x)=i$. Thus $B'_x\subseteq V_i$ and $B'_x\neq\emptyset$. Let $v$ be a vertex in $B'_x$. Let $w$ be the neighbour of $v$ on a shortest path in $G'$ between $v$ and $r$. Thus $w$ is in $B''_x\cap V_{i-1}$. Thus  $|B''_x\cap V_i|\leq k-1$, as desired. Hence $\mathcal{T}_i$ is $(k-1)$-rich. 

We now prove that $(V_0,V_1,\dots,V_t)$ is shadow-complete. Let $H$ be a connected component of $G[V_i\cup V_{i+1}\cup\dots\cup V_t]$ for some $i\in\{1,\dots,t\}$. Let $X$ be the subgraph of $T$ whose corresponding bags in $\mathcal{T}$ intersect $V(H)$. Since $H$ is connected, $X$ is indeed a connected subtree of $T$. Let $x$ be the root of $X$. 
Consider a vertex $w$ in the shadow of $H$. That is, $w\in V_{i-1}$ and $w$ is adjacent to some vertex $v$ in $V(H)\cap V_i$. 
Let $y$ be the node closest to $x$ in $X$ such that $v\in B_y$. Then $v\in B'_y$ and $w\in B''_y$. Thus $\ell(y)=i$. 
Note that $B_x\subseteq V_{\ell(x)-1} \cup V_{\ell(x)}$ and some vertex in $B_x$ is in $V(H)$ and is thus in $V_i\cup V_{i+1}\cup\dots \cup V_t$. Thus $\ell(x)\geq i$. Since $x$ is an ancestor of $y$ in $T$, 
$\ell(x)\leq \ell(y)=i$ by property $(\star)$, implying $\ell(x)=i$. Thus $w\in B''_x$. Since $B''_x$ is a clique, the shadow of $H$ is a clique. Hence $(V_0,V_1,\dots,V_t)$ is shadow-complete. 
Moreover, since $|B''_x|\leq k$, the shadow of $H$ has size at most $k$. 
\end{proof}


Iterating \cref{ShadowCompletePi} gives the next lemma.

\begin{lem}[\citep{DMW17}] 
\label{ShadowPi}
For some number $c$, let $\mathcal{G}_0$ be a class of graphs $G$ with $\pi(G)\leq c$. For $k\geq1$, let $\mathcal{G}_k$ be a class of graphs that have a shadow-complete layering such that each layer induces a graph in $\mathcal{G}_{k-1}$. Then $\pi(G) \leq c\,4^k$ for every graph $G\in \mathcal{G}_k$.
\end{lem}

\cref{RichShadow,ShadowPi} lead to the following result:

\begin{lem}
\label{RichColourPi}
Let $G$ be a graph that has a $k$-rich tree decomposition $\mathcal{T}$ such that the subgraph induced by each bag is nonrepetitively $c$-colourable. Then 
$$\pi(G) \leq c\,4^k.$$
\end{lem}

\begin{proof}
For $j\in\{0,\dots,k\}$, let $\mathcal{G}_j$ be the set of induced subgraphs of $G$ that have a $j$-rich tree decomposition contained in $\mathcal{T}$. Note that $G$ itself is in $\mathcal{G}_k$. Consider a graph $G'\in \mathcal{G}_0$. Then $G'$ is the union of disjoint subgraphs of $G$, each of which is contained in a bag of $\mathcal{T}$ and is thus nonrepetitively $c$-colourable. Thus  $G'$ is nonrepetitively $c$-colourable. Now consider some $G'\in \mathcal{G}_j$ for some $j\in\{1,\dots,k\}$. Thus $G'$ is an induced subgraph of $G$ with a $j$-rich tree decomposition contained in $\mathcal{T}$. By \cref{RichShadow}, $G'$ has a shadow-complete layering $(V_0,\dots,V_t)$ such that for each layer $V_i$, the induced subgraph $G'[V_i]$  has a $(j-1)$-rich tree decomposition $\mathcal{T}_i$ contained in $\mathcal{T}$. Thus $G'[V_i]$ is in $\mathcal{G}_{j-1}$. By \cref{ShadowPi}, the graph $G$ is nonrepetitively $4^kc$-colourable.
\end{proof}

An identical proof using \cref{ShadowCompleteRho} instead of \cref{ShadowCompletePi} gives the following analogous result for stroll-nonrepetitive colourings. 


\begin{lem}
\label{RichColourRho}
Let $G$ be a graph that has a $k$-rich tree decomposition $\mathcal{T}$ such that the subgraph induced by each bag is stroll-nonrepetitively $c$-colourable. Then $$\rho(G) \leq c\,4^k.$$
\end{lem}

If a graph $G$ has a tree-decomposition with adhesion $r$ such that each torso is nonrepetitively $c$-colourable, then $G$ is a subgraph of a graph that has an $r$-rich tree-decomposition such that each bag is nonrepetitively $c$-colourable. Thus \cref{RichColourPi,RichColourRho} imply \cref{RichColour}, whose proof was the goal of this subsection. 


\subsection{Non-Minor-Closed Classes}

This section explores generalisations of the results in \cref{Planar,Surfaces} for various non-minor-closed classes. All the results are due to \citet{DMW} and are based on the previous work on treewidth and strong products. \citet{DMW} only presented their results for $\pi$, but the proofs immediately generalise for $\rho$. 

A graph is \emph{$k$-planar} if it has a drawing in the plane in which each edge is involved in at most $k$ crossings. Such graphs provide a natural  generalisation of planar graphs, and are important in graph drawing research; see the recent bibliography on 1-planar graphs and the 140 references therein \citep{KLM17}. \citet{DMW} extended the above-mentioned result of \citet{DJMMUW19,DJMMUW20} to show that every 1-planar graph is a subgraph of $H\boxtimes P\boxtimes K_{30}$ for some planar graph $H$ with treewidth at most 3 and for some path $P$. \cref{ProductGraphPathTreewidth} then implies:

\begin{thm}[\citep{DMW}] 
For every $1$-planar graph $G$, $$\pi(G)\le \rho(G) \leq 30\times 4^4=7680.$$
\end{thm}

Similarly, \citet{DMW} proved that every $k$-planar graph is a subgraph of $H\boxtimes P\boxtimes K_{18k^2+48k+30}$, for some graph $H$ of treewidth $\binom{k+4}{3}-1$ and for some path $P$.  \cref{ProductGraphPathTreewidth} then implies:

\begin{thm}[\citep{DMW}] 
For every $k$-planar graph $G$, 
$$\pi(G)\le \rho(G) \leq (18k^2+48k+30) 4^{\binom{k+4}{3}}.$$
\end{thm}

More generally, a graph $G$ is \emph{$(g,k)$-planar} if it has a drawing in a surface with Euler genus at most $g$ in which each edge is involved in at most $k$ crossings. \citet{DMW} proved that every $(g,k)$-planar graph is a subgraph of $H\boxtimes P \boxtimes K_\ell$ for some graph $H$ with $\tw(H) \leq \binom{k+5}{4}-1$, where $\ell:=\max\{2g,3\}\cdot(6k^2+16k+10)$. \cref{ProductGraphPathTreewidth} then implies:

\begin{thm}[\citep{DMW}] 
For every $(g,k)$-planar graph $G$, 
$$  \pi(G)\le \rho(G) \leq \max\{2g,3\}\cdot(6k^2+16k+10) 4^{\binom{k+5}{4}}.$$
\end{thm}

Map graphs provide another natural generalisation of graphs embedded in surfaces. Start with a graph $G_0$ embedded in a surface of Euler genus $g$, with each face labelled a `nation' or a `lake', where each vertex of $G_0$ is incident with at most $d$ nations. Let $G$ be the graph whose vertices are the nations of $G_0$, where two vertices are adjacent in $G$ if the corresponding faces in $G_0$ share a vertex. Then $G$ is called a \emph{$(g,d)$-map graph}.  A $(0,d)$-map graph is called a (plane) \emph{$d$-map graph}; see \citep{FLS-SODA12,CGP02} for example. The $(g,3)$-map graphs are precisely the graphs of Euler genus at most $g$; see \citep{DEW17}. So $(g,d)$-map graphs generalise graphs embedded in a surface; now assume that $d\geq 4$.

\citet{DMW} proved that every $d$-map graph  is a subgraph of $H \boxtimes P \boxtimes K_{21d(d-3)}$ for some path $P$ and for some graph $H$ with $\tw(H)\leq 9$. \cref{ProductGraphPathTreewidth} then implies:

\begin{thm}[\citep{DMW}]
\label{PlaneMapPartition}
For every $d$-map graph $G$,
$$ \pi(G) \leq \rho(G) \leq 21 \cdot 4^{10} d(d-3).$$
\end{thm}

\citet{DMW} proved that for integers $g\geq 0$ and $d\geq 4$, if $\ell:=  7d(d-3)\, \max\{2g,3\}$ then every $(g,d)$-map graph $G$ is a subgraph of $H \boxtimes P \boxtimes K_{\ell}$ for some path $P$ and for some graph $H$ with $\tw(H)\leq 14$. \cref{ProductGraphPathTreewidth} then implies:

\begin{thm}[\citep{DMW}]
For integers $g\geq 0$ and $d\geq 4$, and for every $(g,d)$-map graph $G$,
$$ \pi(G) \leq  7\cdot 4^{15} \,d(d-3)\, \max\{2g,3\} .$$
\end{thm}

\section{Subdivisions}
\label{Subdivisions}

This section studies nonrepetitive colourings of graph subdivisions. These results are of independent interest, and will be important when considering expansion in the following section. 

\subsection{Upper Bounds: Small Subdivisions}

\citet{NOW11} observed the following simple upper bounds on the nonrepetitive chromatic number of any subdivision. Note that much better bounds follow. 

\begin{lem}[\citep{NOW11}]
	\label{NonRepSub}
	\textup{(a)} For every $(\leq1)$-subdivision $H$ of a graph $G$, $$\pi(H)\leq\pi(G)+1.$$
	\textup{(b)} For every $(\leq2)$-subdivision $H$ of a graph $G$, $$\pi(H)\leq\pi(G)+2.$$
	\textup{(c)} For every subdivision $H$ of a graph $G$, $$\pi(H)\leq\pi(G)+3.$$
\end{lem}

\begin{proof}
	First we prove (a). Given a nonrepetitive $k$-colouring of $G$, introduce a new colour for each division vertex of $H$. Since this colour does not appear elsewhere, a repetitively coloured path in $H$ defines a repetitively coloured path in $G$. Thus $H$ contains no repetitively coloured path. Part (b) follows by applying (a) twice. 
	
	Now we prove (c). Let $n$ be the maximum number of division vertices on some edge of $G$. By \cref{PathPi}, $P_n$ has a nonrepetitive $3$-colouring $(c_1,c_2,\dots,c_n)$. Arbitrarily orient the edges of $G$. Given a nonrepetitive $k$-colouring of $G$, choose each $c_i$ to be one of three new colours
	for each arc $vw$ of $G$ that is subdivided $d$ times, 
	colour the division vertices from $v$ to $w$ by $(c_1,c_2,\dots,c_d)$. 
	Suppose $H$ has a repetitively coloured path $P$. Since $H-V(G)$ is a collection of disjoint paths, each of which is nonrepetitively coloured, $P$ includes some original vertices of $G$. Let $P'$ be the path in $G$ obtained from $P$ as follows. If $P$ includes the entire subdivision of some edge $vw$ of $G$ then replace that subpath by $vw$ in $P'$. If $P$ includes a subpath of the subdivision of some edge $vw$ of $G$, then without loss of generality, it includes $v$, in which case replace that subpath by $v$ in $P'$. Since the colours assigned to division vertices are distinct from the colours assigned to original vertices, a $t$-vertex path of division vertices in the first half of $P$ corresponds to a $t$-vertex path of division vertices in the second half of $P$. Hence $P'$ is a repetitively coloured path in $G$. This contradiction proves that $H$ is nonrepetitively coloured. Hence $\pi(H)\leq k+3$.
\end{proof}

Note that \cref{NonRepSub}(a) is best possible in the weak sense that $\pi(C_5)=4$ and $\pi(C_4)=3$; see \citep{Currie-TCS05}. But better asymptotic results can be obtained. 

First we prove that for graphs with maximum degree $\Delta$, the $O(\Delta)$ upper bound for general graphs can be improved to $O(\Delta)$ for subdivided graphs. The proof uses Rosenfeld counting. 

\begin{thm}
\label{1SubdivisionDelta}
For every $(\geq 1)$-subdivision of a graph with maximum degree $\Delta$, 
$$\pich(G) \leq \ceil{5.22\, \Delta}.$$
\end{thm}

\cref{1SubdivisionDelta} follow from the next lemma with $r= 0.36$. Recall that $\Pi(G,L)$ is the number of nonrepetitive $L$-colourings of a graph $G$.

\begin{lem}
	\label{SubdividedRosenfeld}
	Fix an integer $\Delta\geq 2$ and a real number $r \in(0,1)$. Let 
	$$\beta := \frac{\Delta-1}{r} \quad \text{and}\quad c:= \ceil*{\beta + \frac{\Delta}{(1-r)^2}}.$$  
	Then for every $(\geq 1)$-subdivision $G$ of a graph with maximum degree $\Delta$, 
	for every $c$-list assignment $L$ of $G$, 
	and for every vertex $v$ of $G$, 
	$$\Pi(G,L) \geq \beta\, \Pi(G-v,L).$$ 
\end{lem}

\begin{proof}
	We proceed by induction on $|V(G)|$. The base case with $|V(G)|=1$ is trivial (assuming $\Pi(G,L)=1$ if $V(G)=\emptyset$). Let $n$ be an integer such that the lemma holds for all graphs with less than $n$ vertices.
	Let $G$ be an $n$-vertex $(\geq 1)$-subdivision of a graph with maximum degree $\Delta$. 
	Let $L$ be a $c$-list assignment of $G$. Let $v$ be any vertex of $G$. Let $F$ be the set of $L$-colourings of $G$ that are repetitive but are nonrepetitive on $G-v$. Then 
	\begin{align}
	\label{Fsubdiv}
	\Pi(G,L) = |L(v)| \, \Pi(G-v,L) \,-\,|F| \geq c \, \Pi(G-v,L) \,-\,|F|.
	\end{align}
	We now upper-bound $|F|$. For $i\in\NN$, let $F_i$ be the set of colourings in $F$, for which there is a  repetitively path in $G$ on $2i$ vertices. Then $|F| \leq \sum_{i\in\NN} |F_i|$. For each colouring $\phi$ in $F_i$ there is a repetitively path $PQ$ on $2i$ vertices in $G$ such that $v\in V(P)$, 
	$G-V(P)$ is nonrepetitively coloured by $\phi$, and $\phi$ is completely determined by the restriction of $\phi$ to $G-V(P)$ colouring (since the colouring of $Q$ is identical to the colouring of $P$). Charge $\phi$ to $PQ$. The number of colourings in $F_i$ charged to $PQ$ is at most $\Pi(G-V(P),L)$. Since $P$ contains $v$ and $i-1$ other vertices, by induction
	\begin{align*}
	\Pi(G-v,L) \geq \beta^{i-1} \, \Pi(G-V(P),L). 
	\end{align*}
	Thus the number of colourings in $F_i$ charged to $PQ$ is at most $\beta^{1-i}\,\Pi(G-v,L)$. A simple adaptation of \cref{PathsContainingVertex} shows that there are at most $i\Delta(\Delta-1)^{i-1}$ paths on $2i$ vertices including $v$. 
	Thus 
	\begin{align*}
	|F_i|  \leq  	i\,\Delta(\Delta-1)^{i-1}\,\beta^{1-i}\,\Pi(G-v,L)
	= 	i\,\Delta r^{i-1}\,\Pi(G-v,L).
	\end{align*}
	Hence
	\begin{align*}
	|F|  \leq \sum_{i\in\NN} |F_i| 
	=
	\sum_{i\in\NN}  i\,\Delta r^{i-1}\,\Pi(G-v,L)
	=
	\Delta\, \Pi(G-v,L) \sum_{i\in\NN}  i\,r^{i-1}
	=
	\frac{\Delta}{(1-r)^2}\, \Pi(G-v,L) .
	\end{align*}
	By \cref{Fsubdiv}, 
	\begin{align*}
	\Pi(G,L) 
	\geq c \, \Pi(G-v,L) \,-\,|F|
	\geq c \, \Pi(G-v,L) \,-\,\frac{\Delta}{(1-r)^2}\, \Pi(G-v,L)	
	\geq \beta\, \Pi(G-v,L)  ,
	\end{align*}
	as desired. 
\end{proof}

Upper bounds on $\pi(G^{(1)})$ that do not depend on $\Delta(G)$ are difficult. There is a lower bound,  
$$\pi(G^{(1)}) \geq \st(G^{(1)}) \geq \sqrt{\chi(G)},$$
where the second inequality was proved by \citet{Wood-DMTCS05}. We have the following upper bound that also involves $\chi(G)$. 

\begin{lem}
\label{Subdiv1}
For every graph $G$, $$\pi(G^{(1)}) \leq 2 \ceil{ ( \chi(G)\,\pi(G) )^{1/3} }^2 + \chi(G)$$
\end{lem}

\begin{proof}
	Let $k:= \ceil{ ( \chi(G)\,\pi(G) )^{1/3} }$. 
	Let $c$ be a proper colouring of $G$ with colour-set $\{1,\dots,\chi(G)\}$. 
	Let $A:= \{1,\dots,\ceil{ \frac{k^2}{\chi(G)} } \}$ and $B:= \{1,\dots,k \}$. 
	Let $\phi$ be a nonrepetitive colouring of $G$ with colour-set $A\times B$ 
	(which exists since $|A|\,|B| = \ceil{ \frac{k^2}{\chi(G)} } k 	\geq   \frac{k^3}{\chi(G)} \geq	 \pi(G)$).
	For each vertex $v$ of $G$, if $c(v)=i$ and $\phi(v)=(a,b)$ then colour $v$ by $(i,a)$. 
	For each edge $vw$ of $G$, if $c(v)<c(w)$ and $\phi(v)=(a,b)$ and $\phi(w)=(a',b')$, then colour the division vertex of $vw$ by $(b,b')$. 
	The number of colours is at most $\chi(G) |A|  + |B|^2 = \chi(G) 
	\ceil{ \frac{k^2}{\chi(G)} } + k^2 \leq
2k^2 + 	\chi(G)
= 2 \ceil{ ( \chi(G)\,\pi(G) )^{1/3} }^2 + \chi(G)$. 
	
	Suppose for the sake of contradiction that $G^{(1)}$ contains a repetitively coloured path $P=(p_1,\dots,p_{2t})$. Since $P$ alternates between original and division vertices, exactly one of $p_1$ and $p_{2t}$ is original. Without loss of generality, $p_1$ is an original vertex (otherwise consider the reverse path). Since original and division vertices are assigned distinct colours, $p_i$ is original if and only if $p_{t+i}$ is original. Thus $Q:=(p_1,p_3,\dots,p_{t-1},p_{t+1},p_{t+3},\dots,p_{2t-1})$ is path in $G$. In particular, $t$ is even and at least 2. Consider each $j\in\{1,3,\dots,t-1\}$. Then $p_j$ and $p_{t+j}$ are coloured $(i,a)$, and $p_{j+1}$ and $p_{t+j+1}$ are coloured $(b,b')$ for some distinct $i,i'\in\{1,\dots,\chi(G)\}$ and $a\in A$ and $b,b'\in B$. If $i<i'$ then $\phi(p_j)=\phi(p_{t+j})=(a,b)$, otherwise $\phi(p_j)=\phi(p_{t+j})=(a,b')$. Hence $Q$ is $\phi$-repetitive. 
	
	This contradiction shows that $G^{(1)}$ is nonrepetitively coloured. Hence 	$\pi(G^{(1)}) \leq 
	2 \ceil{ ( \chi(G)\,\pi(G) )^{1/3} }^2 + \chi(G)$.
\end{proof}

For 2-subdivisions and 3-subdivisions we have the following improved upper bounds.

\begin{lem}
\label{Subdiv2}
For every graph $G$, $$\pi(G^{(2)}) \leq 3 \ceil{ \pi(G)^{1/2}} .$$
\end{lem}

\begin{proof}
Let $k:=  \ceil{ \pi(G)^{1/2} }$. Let $\phi$ be a nonrepetitive colouring of $G$ with colour-set $\{1,\dots,k\}\times\{1,\dots,k\}$. Arbitrarily orient the edges of $G$. Let $A(G)$ be the resulting set of arcs of $G$. For each original vertex $v\in V(G)$, if $\phi(v)=(a,i)$ then colour $v$ by $B_a$. For each arc $vw\in A(G)$, if $(v,x,y,w)$ is the path in $G^{(2)}$ corresponding to $vw$, and $\phi(v)=(a,i)$ and $\phi(w)=(b,j)$, then colour $x$ by $C_i$ and colour $y$ by $D_j$. There are at most $3k$ colours.

Suppose for the sake of contradiction that $(p_1,\dots,p_n,q_1,\dots,q_n)$ is a repetitively coloured path in $G^{(2)}$. Since only original vertices are assigned a type-$A$ colour, 
$p_i$ is original if and only if $q_i$ is original. Say $p_{\ell_1},p_{\ell_2},\dots,p_{\ell_t},q_{\ell_1},q_{\ell_2},\dots,q_{\ell_t}$ are the original vertices in $Q$. 

Suppose that $t=0$. Then $Q$ is the subpath formed by the two division vertices of some edge of $G$. These two vertices are assigned distinct colours, implying $Q$ is nonrepetitively coloured. 

Now assume that $t\geq 1$. Then $R:=(p_{\ell_1},p_{\ell_2},\dots,p_{\ell_t},q_{\ell_1},q_{\ell_2},\dots,q_{\ell_t})$ is a path in $G$. Without loss of generality, $p_{\ell_t+1}$ is in the first half of $Q$ (otherwise consider $Q$ in the reverse order). For each $i\in\{1,\dots,t\}$, if $p_{\ell_i}$ and $q_{\ell_i}$ are coloured $B_a$, and $p_{\ell_i+1}$ and $q_{\ell_i+1}$ are coloured $C_j$ or $D_j$, then $\phi(p_{\ell_i})=\phi(q_{\ell_i})=(a,j)$. Hence $R$ is $\phi$-repetitive. 

This contradiction shows that $G^{(2)}$ is nonrepetitively coloured. Hence 
$\pi(G^{(2)}) \leq 3k$. 
\end{proof}

\begin{lem}
	\label{Subdiv3}
	For every graph $G$, $$\pi(G^{(3)}) \leq (4+o(1))\, \pi(G)^{2/5} .$$
\end{lem}

\begin{proof}
Let $k:=  \ceil{ \pi(G)^{1/5}}$. Let $\phi$ be a nonrepetitive colouring of $G$ with colour-set $\{1,\dots,k^2\}\times\{1,\dots,k^2\}\times\{1,\dots,k\}$. Arbitrarily orient the edges of $G$. Let $A(G)$ be the resulting set of arcs of $G$. For each original vertex $v\in V(G)$, if $\phi(v)=(a,b,c)$ then colour $v$ by $A_a$. For each arc $vw\in A(G)$, if $(v,x,m,y,w)$ is the path in $G^{(2)}$ corresponding to $vw$, and $\phi(v)=(a,b,c)$ and $\phi(w)=(a',b',c')$, then colour $x$ by $B_b$, colour $m$ by $C_{c,c'}$, and colour $y$ by $B'_{b'}$. There are at most $4k^2$ colours.

Suppose for the sake of contradiction that $(p_1,\dots,p_n,q_1,\dots,q_n)$ is a repetitively coloured path in $G^{(2)}$. Since only original vertices are assigned a type-$A$ colour, 
$p_i$ is original if and only if $q_i$ is original. Say $p_{\ell_1},p_{\ell_2},\dots,p_{\ell_t},q_{\ell_1},q_{\ell_2},\dots,q_{\ell_t}$ are the original vertices in $Q$. 

Suppose that $t=0$. Then $Q$ is contained in the subpath formed by the three division vertices of some edge of $G$. These three vertices are assigned distinct colours, implying $Q$ is nonrepetitively coloured. 

Now assume that $t\geq 1$. Then $R:=(p_{\ell_1},p_{\ell_2},\dots,p_{\ell_t},q_{\ell_1},q_{\ell_2},\dots,q_{\ell_t})$ is a path in $G$. Without loss of generality, $p_{\ell_t+1}$ is in the first half of $Q$ (otherwise consider $Q$ in the reverse order). Consider each $i\in\{1,\dots,t\}$. Then $p_{\ell_i}$ and $q_{\ell_i}$ are coloured $A_a$ for some $a\in \{1,\dots,k^2\}$. And $p_{\ell_i+1}$ and $q_{\ell_i+1}$ are coloured $B_b$ or $B'_b$, for some $b\in\{1,\dots,k^2\}$. And $p_{\ell_i+2}$ and $q_{\ell_i+2}$ are coloured $C_{c,c'}$, for some $c,c'\in\{1,\dots,k\}$. If $p_{\ell_i+1}$ and $q_{\ell_i+1}$ are coloured $B_b$, then 
$\phi(p_{\ell_i})=\phi(q_{\ell_i})=(a,b,c)$. Otherwise, $p_{\ell_i+1}$ and $q_{\ell_i+1}$ are coloured $B'_b$, implying $\phi(p_{\ell_i})=\phi(q_{\ell_i})=(a,b,c')$. In both cases, $\phi(p_{\ell_i})=\phi(q_{\ell_i})$. Hence $R$ is $\phi$-repetitive. This contradiction shows that $G^{(2)}$ is nonrepetitively coloured. Hence 
$\pi(G^{(3)}) \leq 4k^2$. 
\end{proof}

\subsection{Upper Bounds: Large Subdivisions}

Loosely speaking, \cref{NonRepSub} says that nonrepetitive colourings of subdivisions are not much ``harder" than nonrepetitive colourings of the original graph. This intuition is made more precise if we subdivide each edge many times. Then nonrepetitive colourings of subdivisions are much ``easier" than nonrepetitive colourings of the original graph. In particular, \citet{Gryczuk-IJMMS07} proved that every graph has a nonrepetitively 5-colourable subdivision. This bound was improved to 4 by \citet{BW08} and by \citet{MS09}, and to 3 by \citet{PZ09} (affirming a conjecture of \citet{Gryczuk-IJMMS07}). This deep generalisation of Thue's Theorem implies that the class of nonrepetitively $3$-colourable graphs is not contained in a proper topologically-closed class. 

For each of these results, the number of division vertices per edge is $O(|V(G)|)$ or $O(|E(G)|)$. Improving these bounds, \citet{NOW11} proved that every graph has a nonrepetitively 17-colourable subdivision with $O(\log |V(G)|)$ division vertices per edge, and that $\Omega(\log n)$ division vertices are needed on some edge of any nonrepetitively $O(1)$-colourable subdivision of $K_n$ (see \cref{CompleteGraphSubdivLowerBound} below). No attempt was made to optimise the constant 17. \citet{GPZ11} asked whether every graph has a nonrepetitively $c$-choosable subdivision, for some constant $c$? This problem was solved by \citet{DJKW16}, who proved that every graph has a nonrepetitively $5$-choosable subdivision. Each edge $vw$ is subdivided $O(\log\deg(v)+\log\deg(w))$ times, which is $O(\log \Delta)$ for graphs of maximum degree $\Delta$, which is at most the $O(\log|V(G)|)$ bound of \citet{NOW11}. 

\begin{table}[H]
	\caption{Bounds on the number of colours and number of division vertices in nonrepetitively colourable subdivisions.}
	\begin{center}
		\begin{tabular}{ccl}
			\hline
			\# colours & \# division vertices per edge $vw$ & reference\\
			\hline
			$\pi\leq 5$ & $O(|E(G)|)$ & \citet{Gryczuk-IJMMS07}\\
			$\pi\leq 4$ & $O(|V(G)|)$ & \citet{BW08} \\
			$\pi\leq 4$ & $O(|E(G)|)$ & \citet{MS09}\\
			$\pi\leq17$ & $O(\log |V(G)|)$ & \citet{NOW11}\\ 
			$\pich\leq5$ & $O(\log \deg(v) + \log\deg(w) ))$ & \citet{DJKW16} \\ 
			$\pi\leq 5$ & $O(\log \pi(G))$ & \cref{5Subdiv}\\ 
			$\pi\leq 3$ & $O(|E(G)|)$ & \citet{PZ09}\\
			\hline
		\end{tabular}
	\end{center}
\end{table}

\cref{5Subdiv} below presents a new construction with 5 colours and $O(\log\pi(G))$ division vertices per edge. Like the result of \citet{DJKW16} the number of division vertices per edge depends on the original graph, and in the worst case is $O(\log|V(G)|)$, thus matching the upper bound of \citet{NOW11}. For graphs of maximum degree $\Delta$, since $\pi(G)\leq O(\Delta^2)$, the new upper bound matches the bound $O(\log\Delta)$ by \citet{DJKW16}. Note that \citet{BGKNP07} proved that every tree has a nonrepetitively 3-colourable subdivision.

\begin{open}
Does every graph $G$ have a $c$-choosable subdivision with $O(\log \pi(G)))$ division vertices per edge, for some constant $c$?
\end{open}

Here we include the proof of \citet{BW08}\footnote{The original proof of \citet{BW08} had an error, which was reported and corrected by Antonides, Spychalla and Yamzon; see the corrigendum in \citep{BW08}.}. 

\begin{thm}[\citep{BW08}]
\label{4ColourableSubdivision}
Every graph $G$ has a nonrepetitively 4-colourable subdivision. 
\end{thm}

\begin{proof}
Without loss of generality, $G$ is connected. Say $V(G)=\{v_0,v_1,\dots,v_{n-1}\}$ ordered by non-decreasing distance from $v_0$. As illustrated in \cref{Subdiv}, let $G'$ be the subdivision of $G$ obtained by subdividing each edge $v_iv_j\in E(G)$ (with $i<j$) $2(j-i)-1$ times. The \emph{depth} of each vertex $x$ of $G'$ is the distance from $v_0$ to $x$  in $G'$. In the original graph $G$, for each $j\in\{1,\dots,n\}$, vertex $v_j$ has a neighbour $v_i$ with $i<j$; it follows that $v_j$ is at depth $2j$. For $i\geq 0$, let $V_i$ be the set of vertices in $G'$ at depth $i$. So $(V_0,V_1,\dots,V_{2n})$ is a layering of $G'$. Note that the endpoints of each edge are in consecutive layers. (Think of $v_0,v_1,\dots,v_{n-1}$ on a horizontal line in this order, with a vertical line through each $v_i$, and an additional vertical line between $v_i$ and $v_{i+1}$. Each edge of $G$ is subdivided at each point it crosses a vertical line.)

\begin{figure}[!ht]
\begin{center}
\includegraphics{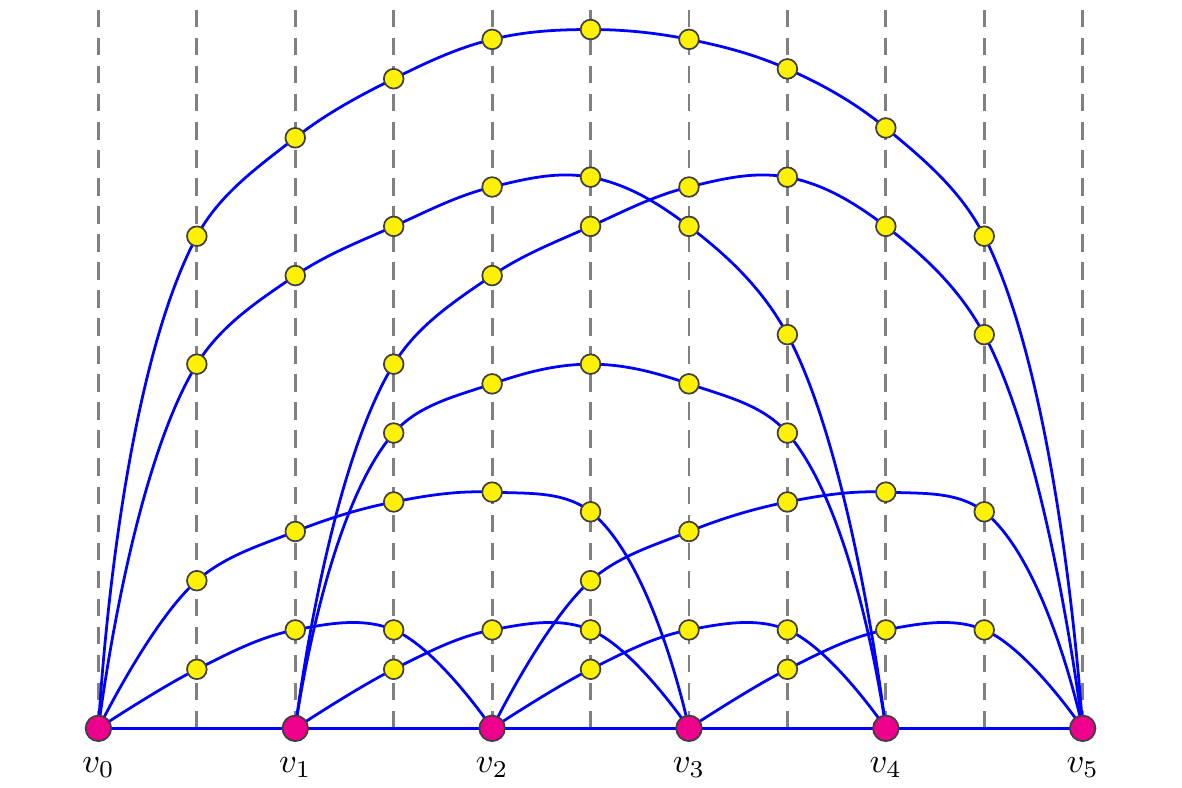}
\caption{\label{Subdiv} Proof of \cref{4ColourableSubdivision}: The subdivision $H$ with $G=K_6$.}
\end{center}
\end{figure}

By \cref{PathSigma}, there is a walk-nonrepetitive $4$-colouring $\phi$ of the path $P=(p_0,p_1,\dots,p_{2n})$. Colour each vertex of $G'$ at depth $j$ by $\phi(p_j)$. 

Suppose that $G'$ contains a repetitively coloured path $Q=(x_1,x_2,\dots,x_t,y_1,y_2,\dots,y_t)$. Let $R$ be the projection of $Q$ into $P$. Since no two adjacent vertices in $G'$ are at the same depth, $R$ is a walk in $P$. Since $\phi$ is walk-nonrepetitive, $R$ is boring. Thus $x_j$ and $y_j$ are at the same depth for each $j\in\{1,\dots,t\}$ Since no two adjacent vertices in $G'$ are at the same depth, $t\geq 2$. 

First suppose that $t=2$. Since $x_1$ and $y_1$ are at the same depth, and $x_2$ is adjacent to both $x_1$ and $y_1$, it must be that $x_2$ is an original vertex (since division vertices only have two neighbours, and they are at distinct depths). Similarly, $y_1$ is an original vertex. This is a contradiction, since no two original vertices of $G$ are adjacent in $G'$. Now assume that $t\geq 3$. 

First suppose that $x_{j-1}$ and $x_{j+1}$ are at the same depth for some $j\in\{2,\dots,t-1\}$. Thus $x_j$ is an original vertex of $G$. Say $x_j$ is at depth $i$. Without loss of generality, $x_{j-1}$ and $x_{j+1}$ are at depth $i-1$. There is at most one original vertex in each layer. Thus $y_j$, which is also at depth $i$, is a division vertex. Now $y_j$ has two neighbours in $H$, which are at depths $i-1$ and $i+1$. Thus $y_{j-1}$ and $y_{j+1}$ are at depths $i-1$ and $i+1$, which contradicts the fact that $x_{j-1}$ and $x_{j+1}$ are both at depth $i-1$.

Now assume that for all $j\in\{2,\dots,t-1\}$, the vertices $x_{j-1}$ and $x_{j+1}$ are at distinct depths.

Say $x_1$ is at depth $i$. Without loss of generality, $x_2$ is at depth $i+1$ (since no edge of $G'$ has both endpoints at the same depth). It follows that $x_j$ is at depth $i+j-1$ for each $j\in\{1,\dots,t\}$. In particular, $x_t$ is at depth $i+t-1$. Now, $y_1$ is at depth $i$ (the same depth as $x_1$). Since $x_ty_1$ is an edge, and every edge goes between consecutive levels, $|(i+t-1)-i|=1$, implying $t=2$, which is a contradiction. Hence we have a path-nonrepetitive $4$-colouring of $G'$. 
\end{proof}

The next lemma is a key to our $O(\log\pi(G))$ bounds on the number of division vertices per edges in a nonrepetitively $c$-colourable subdivision.

\begin{lem}
\label{SubdivLemma}
Assume that there exist at least $k$ distinct nonrepetitive $r$-colourings of the $t$-vertex path, for some $k,r,t\in\NN$. Then for every graph $G$ with $\pi(G)\leq k$,
$$\pi(G^{(2t+1)}) \leq r+2.$$
\end{lem}

\begin{proof}
Fix a nonrepetitive colouring $\phi$ of $G$ with colour-set $\{1,\dots,k\}$. Let $c_1,\dots,c_{k}$ be distinct nonrepetitive colourings of the path $(1,\dots,t)$, each with colour-set $\{1,\dots,r\}$. That is, for all distinct $i,j\in\{1,\dots,k\}$ there exists $\ell\in\{1,\dots,t\}$ such that $c_i(\ell)\neq c_j(\ell)$. 
	
For each edge $vw$ of $G$, if $P_{vw}=(v,x_1,\dots,x_t,z,y_t,\dots,y_1,w)$ is the path in $G^{(2t+1)}$ corresponding to $vw$ and $\phi(v)=i$, then colour $x_\ell$ by $c_{i}(\ell)$ for each $\ell\in\{1,\dots,t\}$. The same rule colours $y_\ell$ by $c_j(\ell)$ where $j=\phi(w)$. Call $z$ a \emph{middle} vertex. Colour each original vertex $r+1$ and colour each middle vertex $r+2$. 
	
Suppose for the sake of contradiction that $G^{(2t+1)}$ contains a repetitively coloured path $Q=(p_1,\dots,p_n,q_1,\dots,q_n)$. Since only original vertices are coloured $r+1$, 
$p_i$ is original if and only if $q_i$ is original. 
Say $p_{i_1},p_{i_2},\dots,p_{i_a},q_{i_1},q_{i_2},\dots,q_{i_a}$ are the original vertices in $Q$.

Suppose that $a=0$. Then for some edge $vw$ of $G$, $Q$ is a subpath of $(x_1,\dots,x_t,z,y_t,\dots,y_1)$ using the above notation for $P_{vw}$. Since only middle vertices are coloured $r+2$, without loss of generality, 
$Q$ is a subpath of $(x_1,\dots,x_t)$, which is a contradiction since $(x_1,\dots,x_t)$ is nonrepetitively coloured. 

Now assume that $a\geq 1$. Then $R:=(p_{i_1},p_{i_2},\dots,p_{i_a},q_{i_1},q_{i_2},\dots,q_{i_a})$ is a path in $G$.
Without loss of generality, the middle vertex of the edge $p_{i_a}q_{i_1}$ is in the first half of $Q$ 
(otherwise consider $Q$ in the reverse order). 
Consider $j\in\{1,2,\dots,a\}$ and $\ell\in\{1,2,\dots,t\}$. 
By construction,  
$p_{i_j+\ell}$ is coloured $c_\alpha(\ell)$ where $\alpha:=\phi(p_{i_j})$, and
$q_{i_j+\ell}$ is coloured $c_\beta(\ell)$ where $\beta:=\phi(q_{i_j})$.
Since $Q$ is repetitively coloured, 
$p_{i_j+\ell}$ and $q_{i_j+\ell}$ are assigned the same colour. 
That is, $c_\alpha(\ell) = c_\beta(\ell)$ for each $\ell\in\{1,2,\dots,t\}$.
Since $c_1,\dots,c_k$ are distinct colourings, $\alpha=\beta$. 
Thus $\phi(p_{i_j})=\phi(q_{i_j})$. 
Hence $R$ is a $\phi$-repetitively coloured path in $G$.

This contradiction shows that $G^{(2t+1)}$ is nonrepetitively ($r+2$)-coloured, and $\pi(G^{(2t+1)})\leq r+2$. 
\end{proof}

Numerous authors have shown that paths have exponentially many nonrepetitive 3-colourings; see \citep{Berstel05} for a survey of such results. For example, \citet{EZ98} proved that for every $t\in\NN$ there are at least $2^{t/17}$ distinct nonrepetitive 3-colourings of the $t$-vertex path. This result and  \cref{SubdivLemma} imply:

\begin{thm}
\label{5Subdiv}
For every graph $G$, if $d:=2\ceil{17  \log_2 \pi(G)}+1$ then $$\pi(G^{(d)})\leq 5.$$
\end{thm}


For 6 or more colours, \cref{MultiColourPath} with $t:= \ceil{\log_{r-2} \pi(G)}$ and  \cref{SubdivLemma} imply:

\begin{thm}
\label{GenSubdiv}
For every graph $G$ and integer $r\geq 4$, 
if $d := 2\ceil{\log_{r-2} \pi(G)}+1$ then $$\pi(G^{(d)})\leq r+2.$$
\wish{make this $d\geq ....$}
\end{thm}

We can restate this result as the following upper bound on $\pi(G^{(d)})$.

\begin{thm}
\label{dSubdiv}
For every graph $G$ and  odd integer $d\geq 3$, 
$$\pi(G^{(d)}) \leq  \pi(G)^{2/(d-1)} +4.$$
\wish{drop the odd $d$ assumption, and make this for every $(\geq d)$-subdivision.}	
\end{thm}

\begin{proof}
The result is trivial if $E(G)=\emptyset$. Now assume that $E(G)\neq\emptyset$, implying $\pi(G)\geq 2$. Let $t:=(d-1)/2$, which is  in $\NN $. Let $r:= \ceil{ \pi(G)^{1/t}} +2$. So $r\geq 4$ is an integer. By \cref{MultiColourPath} there exist at least $(r-2)^t\geq\pi(G)$ distinct nonrepetitive $r$-colourings of the $t$-vertex path. By \cref{SubdivLemma}, 
\begin{equation*}
\pi(G^{(d)})=\pi(G^{(2t+1)}) \leq r+2 =\ceil{ \pi(G)^{1/t} } +4 =\ceil{ \pi(G)^{2/(d-1)} } +4.\qedhere
\end{equation*}
\end{proof}

\subsection{Lower Bounds}
\label{SubdivLowerBounds}

We now set out to prove a converse of \cref{NonRepSub}; that is, if $H$ is a subdivision of $G$ with a bounded number of division vertices per edge, then $\pi(G)$ is bounded by a function of $\pi(H)$ (see \cref{GenSubdiv2Graph} below). The following tool by \citet{NR-JCTB00} will be useful.


\begin{lem}[\citep{NR-JCTB00}]\label{NesRas}
For every $k$-colouring of the arcs of an oriented forest $T$, there is a $(2k+1)$-colouring of the vertices of $T$, such that between each pair of (vertex) colour classes, all arcs go in the same direction and have the same colour. 
\end{lem}

A \emph{rooting} of a forest $F$ is obtained by nominating one vertex in each component tree of $F$ to be a \emph{root} vertex.

\begin{lem}[\citep{NOW11}]
\label{NonRepForest}
Let $T'$ be the $1$-subdivision of a forest $T$. Then for every nonrepetitive $k$-colouring $c$ of $T'$, and for every rooting of $T$, there is a nonrepetitive $k(k+1)(2k+1)$-colouring $q$ of $T$, such that:
\begin{enumerate}[(a)] 
\item For all edges $vw$ and $xy$ of $T$ with $q(v)=q(x)$ and $q(w)=q(y)$, the division vertices corresponding to $vw$ and $xy$ have the same colour in $c$.
\item  For all non-root vertices $v$ and $x$ with $q(v)=q(x)$, the division vertices corresponding to the parent edges of $v$ and $x$ have the same colour in $c$.
\item For every root vertex $r$ and every non-root vertex $v$, we have $q(r)\neq q(v)$.
\item For all vertices $v$ and $w$ of $T$, if $q(v)=q(w)$ then $c(v)=c(w)$.
\end{enumerate}
\end{lem}

\begin{proof}
Let $c$ be a nonrepetitive $k$-colouring of $T'$ with colour-set $\{1,\dots,k\}$.
Colour each edge of $T$ by the colour assigned by $c$ to the corresponding division vertex. 
Orient each edge of $T$ towards the root vertex in its component.
By \cref{NesRas}, there is a $(2k+1)$-colouring $f$ of the vertices of $T$, such that between each pair of (vertex) colour classes in $f$, all arcs go in the same direction and have the same colour in $c$. 
Consider a vertex $v$ of $T$. 
If $v$ is a root, let $g(r):=0$; otherwise let $g(v):=c(vw)$ where $w$ is the parent of $v$.
Let $q(v):=(c(v),f(v),g(v))$ for each vertex $v$ of $T$. 
The number of colours in $q$ is at most $k(k+1)(2k+1)$.
Observe that claims (c) and (d)  hold by definition.

We claim that $q$ is nonrepetitive.
Suppose on the contrary that there is a path $P=(v_1,\dots,v_{2s})$ in $T$ that is repetitively coloured by $q$. 
That is, $q(v_i)=q(v_{i+s})$  for each $i\in\{1,\dots,k\}$. Thus $c(v_i)=c(v_{i+s})$ and $f(v_i)=f(v_{i+s})$ and $g(v_i)=g(v_{i+s})$. Since no two root vertices are in a common path, (c) implies that every vertex in $P$ is a non-root vertex.

Consider the edge $v_iv_{i+1}$ of $P$ for some $i\in\{1,\dots,s-1\}$.
We have  $f(v_i)=f(v_{i+s})$ and  $f(v_{i+1})=f(v_{i+s+1})$. 
Between these two colour classes in $f$, all arcs go in the same direction and have the same colour. 
Thus the edge $v_iv_{i+1}$ is oriented from $v_i$ to $v_{i+1}$ if and only if  the edge $v_{i+s}v_{i+s+1}$ is oriented from $v_{i+s}$ to $v_{i+s+1}$. 
And $c(v_iv_{i+1})=c(v_{i+s}v_{i+s+1})$.

If at least two vertices $v_i$ and $v_j$ in $P$ have indegree $2$ in $P$, then some vertex between $v_i$ and $v_j$ in $P$ has outdegree $2$ in $P$, which is a contradiction. Thus at most one vertex has indegree $2$ in $P$.  
Suppose that $v_i$ has indegree $2$ in $P$. Then each edge $v_jv_{j+1}$ in $P$ is oriented from $v_j$ to $v_{j+1}$ if $j\leq i-1$, and from $v_{j+1}$ to $v_{j}$ if $j\geq i$ (otherwise two vertices have indegree $2$ in $P$). 
In particular, $v_1v_2$ is oriented from $v_1$ to $v_2$ and $v_{s+1}v_{s+2}$ is oriented from $v_{s+2}$ to $v_{s+1}$. 
This is a contradiction since the edge $v_1v_2$ is oriented from $v_1$ to $v_2$ if and only if  the edge $v_{s+1}v_{s+2}$ is oriented from $v_{s+1}$ to $v_{s+2}$.
Hence no vertex in $P$ has indegree $2$. Thus $P$ is a directed path.

Without loss of generality, $P$ is oriented from $v_1$ to $v_{2s}$. 
Let $x$ be the parent of $v_{2s}$. 
Now  $g(v_{2s})=c(v_{2s}x)$ and $g(v_s)=c(v_sv_{s+1})$ and $g(v_s)=g(v_{2s})$.
Thus $c(v_{s}v_{s+1})=c(v_{2s}x)$. 

Summarising, the path 
\begin{align*}
\big(\underbrace{ v_1,v_1v_2,v_2,\dots,v_s, v_sv_{s+1}},
\underbrace{ v_{s+1},v_{s+1}v_{s+2}, v_{s+2}, \dots, v_{2s}, v_{2s}x}\big)
\end{align*}
in $T'$ is repetitively coloured by $c$. (Here division vertices in $T'$ are described by the corresponding edge.)\ Since $c$ is nonrepetitive in $T'$, we have the desired contradiction. Hence $q$ is a nonrepetitive colouring of $T$.

It remains to prove claims (a) and (b). Consider two edges $vw$ and $xy$ of $T$, such that $q(v)=q(x)$ and $q(w)=q(y)$. Thus $f(v)=f(x)$ and $f(w)=f(y)$. Thus $vw$ and $xy$ have the same colour in $c$. Thus the division vertices corresponding to $vw$ and $xy$ have the same colour in $c$. This proves claim (a). Finally consider  non-root vertices $v$ and $x$ with $q(v)=q(x)$. Thus $g(v)=g(x)$. Say $w$ and $y$ are the respective parents of $v$ and $x$. By construction, $c(vw)=c(xy)$. Thus the division vertices of $vw$ and $xy$  have the same colour in $c$. This proves claim (b).
\end{proof}


A colouring of a graph is \emph{acylic} if adjacent vertcies are assigned distinct colours and every cycle is assigned at least three distinct colours; that is, the subgraph induced by any two colour classes is a forest. 
The \emph{acyclic chromatic number} of a graph $G$, denoted by $\acy(G)$, is the minimum number of colours in an acyclic colouring of $G$. Acyclic colourings are well studied \citep{AB78,AMR91,AMS96,Borodin-DM79}. For example, every planar graph is acyclically 5-colourable \citep{Borodin-DM79}.  We now extend \cref{NonRepForest} to apply to graphs with bounded acyclic chromatic number; see \citep{AM-JAC98,NR-JCTB00} for similar methods. 

\begin{lem}[\citep{NOW11}]\label{NonRepGraph}
Let $G'$ be the $1$-subdivision of a graph $G$, such that $\pi(G')\leq k$ and $\acy(G)\leq\ell$. Then $$\pi(G)\leq \ell\big(k(k+1)(2k+1)\big)^{\ell-1}.$$
\end{lem}

\begin{proof} 
Let $p$ be an acyclic $\ell$-colouring of $G$ with colour-set $\{1,\dots,\ell\}$.
Let $c$ be a nonrepetitive $k$-colouring of $G'$.
For distinct $i,j\in\{1,\dots,\ell\}$, let $G_{\{i,j\}}$ be the subgraph of $G$ induced by the vertices coloured $i$ or $j$ by $p$. 
Thus each $G_{\{i,j\}}$ is a forest, and $c$ restricted to $G'_{\{i,j\}}$ is nonrepetitive.

For distinct  $i,j\in\{1,\dots,\ell\}$, by \cref{NonRepForest} applied to $G_{\{i,j\}}$, there is a nonrepetitive $k(k+1)(2k+1)$-colouring $q_{\{i,j\}}$ of $G_{\{i,j\}}$ satisfying  \cref{NonRepForest}(a)--(d). 

Consider a vertex $v$ of $G$. For each colour $j\in\{1,\dots,\ell\}$ with $j\neq p(v)$, let $q_j(v):=q_{\{p(v),j\}}(v)$. 
Define $$q(v):=\Big(p(v),\big\{(j,q_{j}(v)):j\in\{1,\dots,\ell\},j\neq p(v)\big\}\Big).$$ 
Note that the number of colours in $q$ is at most $\ell\big( k(k+1)(2k+1)\big)^{\ell-1}.$ 
We claim that $q$ is a nonrepetitive colouring of $G$. 

Suppose on the contrary that some path $P=(v_1,\dots,v_{2s})$ in $G$ is repetitively coloured by $q$. That is, $q(v_a)=q(v_{a+s})$ for each $a\in\{1,\dots,s\}$. 
Thus $p(v_a)=p(v_{a+s})$  for each $a\in\{1,\dots,s\}$.
Let $i:=p(v_a)$. Consider any $j\in\{1,\dots,\ell\}$ with $j\neq i$. 
Thus $(j,q_j(v_a))=(j,q_j(v_{a+s}))$ and $q_j(v_a)=q_j(v_{a+s})$. 
Hence $c(v_a)=c(v_{a+s})$  by \cref{NonRepForest}(d).

Consider an edge $v_av_{a+1}$ for some $i\in\{1,\dots,s-1\}$. 
Let $i:=p(v_a)$ and $j:=p(v_{a+1})$. 
Now $q(v_a)=q(v_{a+s})$ and $q(v_{a+1})=q(v_{a+s+1})$. 
Thus $p(v_{a+s})=i$ and $p(v_{a+s+1})=j$. 
Moreover, $(j,q_j(v_a))=(j,q_j(v_{a+s}))$ and $(i,q_i(v_{a+1}))=(i,q_i(v_{a+s+1}))$. 
That is, $q_{\{i,j\}}(v_a)=q_{\{i,j\}}(v_{a+s})$ and $q_{\{i,j\}}(v_{a+1})=q_{\{i,j\}}(v_{a+s+1})$. 
Thus $c(v_av_{a+1})=c(v_{a+s}v_{a+s+1})$ by \cref{NonRepForest}(a).

Consider the edge $v_sv_{s+1}$. Let $i:=p(v_s)$ and $j:=p(v_{s+1})$. 
Without loss of generality, $v_{s+1}$ is the parent of $v_s$ in the forest $G_{\{i,j\}}$. 
In particular, $v_s$ is not a root of $G_{\{i,j\}}$. 
Since $q_{\{i,j\}}(v_s)=q_{\{i,j\}}(v_{2s})$ and by \cref{NonRepForest}(c),
$v_{2s}$ also is not a root of $G_{\{i,j\}}$.
Let $y$ be the parent of $v_{2s}$ in $G_{\{i,j\}}$. 
By \cref{NonRepForest}(b) applied to $v_s$ and $v_{2s}$,  we have $c(v_sv_{s+1})=c(v_{2s}y)$.

Summarising, the path
\begin{align*}
\big(\underbrace{ v_1,v_1v_2,v_2,\dots,v_s, v_sv_{s+1}},
\underbrace{ v_{s+1},v_{s+1}v_{s+2}, v_{s+2}, \dots, v_{2s}, v_{2s}y}\big)
\end{align*}
is repetitively coloured in $G'$. This contradiction proves that $G$ is repetitively coloured by $q$.
\end{proof} 

\cref{NonRepGraph} generalises for $(\leq 1)$-subdivisions as follows:

\begin{lem}[\citep{NOW11}]\label{NonRepGraphs}
Let $H$ be a $(\leq 1)$-subdivision of a graph $G$, such that $\pi(H)\leq k$ and $\acy(G)\leq\ell$. Then 
$$\pi(G)\leq \ell\big((k+1)(k+2)(2k+3)\big)^{\ell-1}.$$
\end{lem}

\begin{proof}
Since $G'$ is a $(\leq 1)$-subdivision of $H$, \cref{NonRepSub}(a)
implies that $\pi(G')\leq k+1$. \cref{NonRepGraph} implies the
result. 
\end{proof}

\begin{lem}[\citep{NOW11}]
\label{ConstructAcyclic}
Let $c$ be a nonrepetitive $k$-colouring of the 1-subdivision $G'$ of a graph $G$.
Then 
$$\acy(G) \leq k\cdot 2^{2k^2}.$$
\end{lem}

\begin{proof} 
Orient the edges of $G$ arbitrarily. Let $A(G)$ be the set of oriented arcs of $G$.
So $c$ induces a $k$-colouring of $V(G)$ and of $A(G)$.
For each vertex $v$ of $G$, let
  $$q(v):=\big\{c(v)\big\} \cup \big\{ (+, c(vw), c(w) ) : vw \in A(G) \big\} \cup 
  		\big\{ (-, c(wv), c(w) ) : wv \in A(G) \big\}.$$
The number of possible values for $q(v)$ is at most $k\cdot 2^{2k^2}$.
We claim that $q$ is an acyclic colouring of $G$.

Suppose on the contrary that $q(v) = q(w)$ for some arc $vw$ of $G$. 
Thus $c(v)=c(w)$ and $(+,c(vw), c(w) ) \in q(v)$, implying $(+,c(vw) , c(w) ) \in q(w)$. 
That is, for some arc $wx$, we have $c(wx)=c(vw)$ and $c(x)=c(w)$. 
Thus the path $(v,vw,w,wx)$ in $G'$ is repetitively coloured. 
This contradiction shows that $q$ properly colours $G$.
It remains to prove that $G$ contains no bichromatic cycle (with respect to  $q$). 

First consider a bichromatic path $P=(u,v,w)$ in $G$ with $q(u)=q(w)$. Thus $c(u)=c(w)$.

Suppose on the contrary that $P$ is oriented $(u,v,w)$, as illustrated in \cref{ConstructAcyclicFigure}(a). 
By construction, $(+,c(uv),c(v)) \in q(u)$, implying  $(+,c(uv),c(v)) \in q(w)$. 
That is, $c(uv)=c(wx)$ and $c(v)=c(x)$ for some arc $wx$ (and thus $x \neq v$). 
Similarly, $(-,c(vw),c(v)) \in q(w)$, implying  $(-,c(vw),c(v)) \in q(u)$. 
Thus $c(vw)=c(tu)$ and $c(v)=c(t)$ for some arc $tu$ (and thus $t \neq v$). 
Hence the 8-vertex path $(tu,u,uv,v,vw,w,wx,x)$ in $G'$ is repetitively coloured by $c$, as illustrated in \cref{ConstructAcyclicFigure}(b). 
This contradiction shows that both edges in $P$ are oriented toward $v$ or both are oriented away from $v$.

\begin{figure}[!h]
	\centering
	\includegraphics[width=\textwidth]{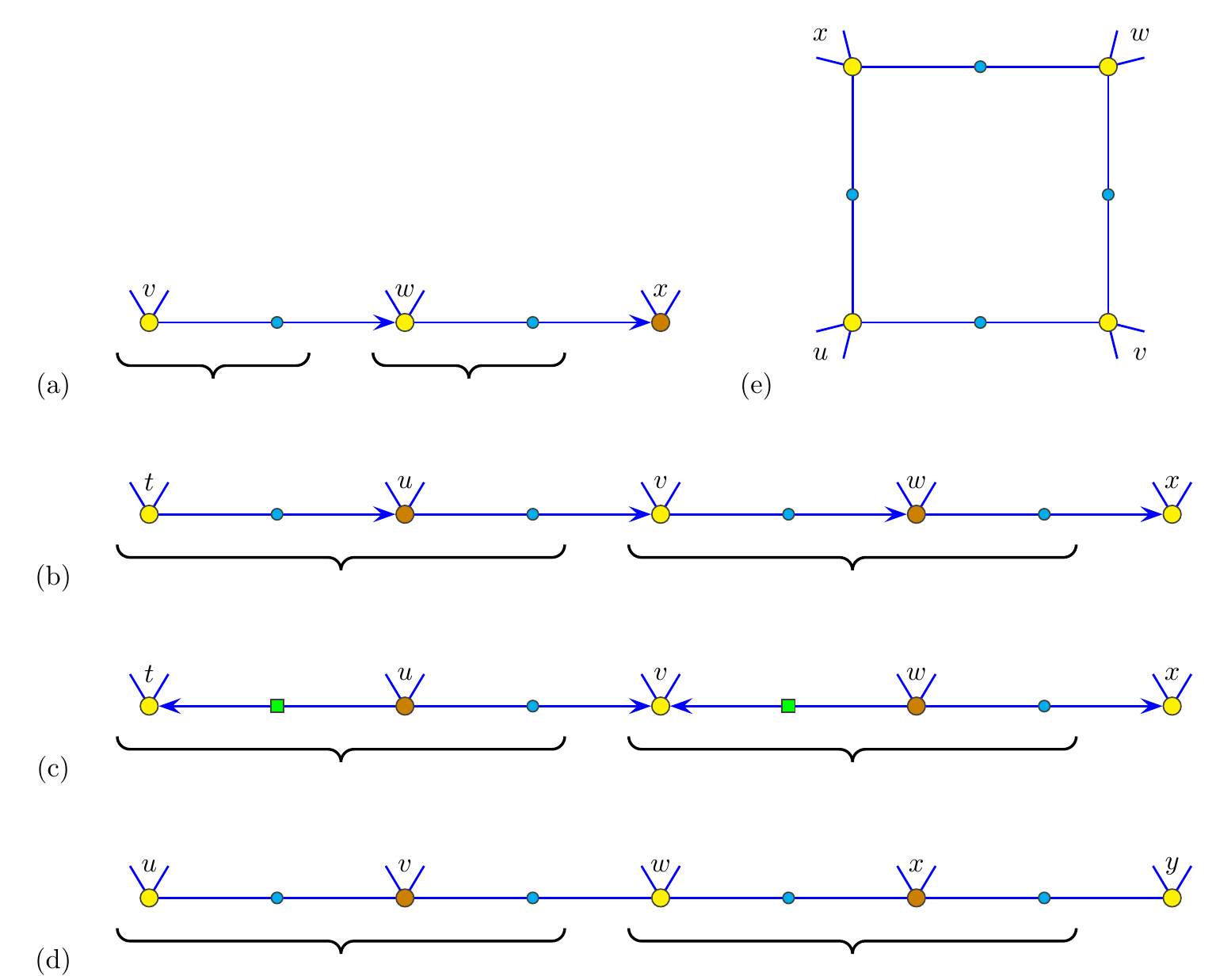}
	\caption{Illustration for \cref{ConstructAcyclic}.\label{ConstructAcyclicFigure}}
\end{figure}

Consider the case in which both edges in $P$ are oriented toward $v$.
Suppose on the contrary that $c(uv)\neq c(wv)$.
By construction, $(+,c(uv),c(v))\in q(u)$, implying $(+,c(uv),c(v))\in q(w)$.
That is, $c(uv)=c(wx)$ and $c(v)=c(x)$ for some arc $wx$ (implying $x \neq v$ since $c(uv)\neq c(wv)$). 
Similarly, $(+,c(wv),c(v))\in q(w)$, implying $(+,c(wv),c(v))\in q(u)$.
That is, $c(wv)=c(ut)$ and $c(t)=c(v)$ for some arc $ut$ (implying $t \neq v$ since $c(ut)=c(wv)\neq c(uv)$). 
Hence the path $(ut,u,uv,v,wv,w,wx,x)$ in $G' $ is repetitively coloured in $c$, as illustrated in \cref{ConstructAcyclicFigure}(c). 
This contradiction shows that $c(uv)=c(wv)$.
By symmetry, $c(uv)=c(wv)$ when both edges in $P$ are oriented away from $v$.

Hence in each component of $G'$, all the division vertices have the same colour in $c$.
Every bichromatic cycle contains a 4-cycle or a 5-path. 
If  $G$ contains a bichromatic 5-path $(u,v,w,x,y)$, then all the division vertices in $(u,v,w,x,y)$ have the same colour in $c$, and  $(u,uv,v,vw,w,wx,x,xy)$ is a repetitively coloured path in $G'$, as illustrated in \cref{ConstructAcyclicFigure}(d). 
Similarly, if $G$ contains a bichromatic 4-cycle $(u,v,w,x)$, then  all the division vertices in $(u,v,w,x)$  have the same colour in $c$, and $(u,uv,v,vw,w,wx,x,xu)$ is a repetitively coloured path in $G'$, as illustrated in \cref{ConstructAcyclicFigure}(e). 
 
Thus $G$ contains no bichromatic cycle, and $q$ is an acyclic colouring of $G$.
\end{proof}

Note that the above proof establishes the following stronger statement: If the 1-subdivision of a graph $G$ has a $k$-colouring that is nonrepetitive on paths with at most $8$ vertices, then $G$ has an acyclic $k\cdot 2^{2k^2}$-colouring in which each component of each 2-coloured subgraph is a star or a 4-path.

\cref{ConstructAcyclic,NonRepSub}(a) imply:

\begin{lem}[\citep{NOW11}]
\label{ConstructAcyclicPlus}
If some $(\leq1)$-subdivision of a graph $G$ has a nonrepetitive $k$-colouring, then 
$$\acy(G) \leq(k+1)\cdot 2^{2(k+1)^2}.$$
\end{lem}

\begin{lem}[\citep{NOW11}]
\label{Subdiv2Graph}
If $\pi(H)\leq k$ for some $(\leq 1)$-subdivision of a graph $G$, then 
$$\pi(G)\leq (k+1)\cdot 2^{2(k+1)^2}\big((k+1)(k+2)(2k+3)\big)^{(k+1)\cdot 2^{2(k+1)^2}-1}.$$
\end{lem}

\begin{proof}
$\acy(G)\leq(k+1)\cdot2^{2(k+1)^2}$ by \cref{ConstructAcyclicPlus}. 
The result follows from \cref{NonRepGraphs} with $\ell=(k+1)\cdot 2^{2(k+1)^2}$.
\end{proof}

Iterated application of \cref{Subdiv2Graph} proves the following theorem. 

\begin{thm}[\citep{NOW11}]
\label{GenSubdiv2Graph}
There is a function $f$ such that for every graph $G$ and every $(\leq d)$-subdivision $H$ of $G$,
$$\pi(G)\leq f(\pi(H),d).$$
\end{thm}

\cref{GenSubdiv2Graph} can be interpreted as follows. For each fixed integer $c\geq 3$, let $f_c(G)$ be the minimum integer $k$ such that $\pi(G')\leq c$ for some $(\leq k)$-subdivision $G'$ of $G$. The results discussed at the start of this section show that $f_c$ is well-defined. For $c\geq 5$, \cref{5Subdiv} shows that $f_c(G)\leq O(\log\pi(G))$ for every graph $G$. \cref{GenSubdiv2Graph} shows that a converse also holds. That is, for each integer $c\geq 5$, the parameters $f_c$ and $\pi$ are tied. 

\begin{open}
Does every graph $G$ have a nonrepetitively 3-colourable subdivision with $O(\log|V(G)|)$ or even $O(\log\pi(G))$ division vertices per edge? Note that \citet{BGKNP07} proved that $\pi(T^{(12)})\leq 3$ for every tree $T$. 
\end{open}

\begin{open}
Does every nonrepetitively $k$-colourable subdivision of a graph $G$ have an edge subdivided at least $c\log_k\pi(G))$ times, for some absolute constant $c>0$?
\end{open}

The results in the next section give an affirmative answer to this question for complete graphs, and indeed for any graph $G$ with $\pi(G)\geq c|V(G)|$.


\subsection{Subdivisions of Dense Graphs}

\citet{NOW11} proved the following generalisation of \cref{ExtremalPi} in the case of complete graphs. The same proof works for all graphs. 

\begin{lem}
\label{ExtremalSubdivision}
For every $(\leq d)$-subdivision $G'$ of a graph $G$,
$$|E(G)| \leq 2\pi(G')^{d+1} ( |V(G)| - \tfrac{c+1}{2} ) .$$
Moreover, 
$$|E(G)| \leq \pi(G^{(d)})^{d+1} ( |V(G)| - \tfrac{c+1}{2} ) .$$
\end{lem}

\cref{ExtremalSubdivision} is implied by the following stronger result. This strengthening will be useful in \cref{BoundedExpansion}. 

\begin{lem}
\label{ExtremalSubdivisionExtra}
Let $G'$ be a $(\leq d)$-subdivision a graph $G$. Assume that $G'$ is $c$-colourable with no repetitively coloured paths on at most $4d+4$ vertices. Then 
$$|E(G)| \leq 2 c^{d+1} ( |V(G)| - \tfrac{c+1}{2} ) .$$
Moreover, if $G'=G^{(d)}$ then 
$$|E(G)| \leq c^{d+1} ( |V(G)| - \tfrac{c+1}{2} ) .$$
\end{lem}

\begin{proof}
Arbitrarily orient each edge of $G$. Let $A(G)$ be the resulting set of arcs of $G$. Let $\phi$ be a $c$-colouring of $G'$ with colour-set $\{1,\dots,c\}$ and with no repetitively coloured paths on at most $4d+4$ vertices. Let $V_i$ be the set of vertices in $G$ coloured $i$. For each arc $e=vw\in A(G)$, if $e_1,\dots,e_d$ is the sequence of division vertices on $e$ from $v$ to $w$, then let $f(e):=( \phi(e_1), \dots,\phi(e_d) )$. Let $Z:= \{f(e) : e\in A(G)\}$. Note that 
$$|Z| \leq \sum_{i=0}^d c^i = \frac{c^{d+1}-1}{c-1} < 2c^d.$$
Moreover, if $G'=G^{(d)}$ then $|Z|\leq c^d$. 

For $i,j\in\{1,\dots,c\}$ and $z\in\{1,\dots,c\}^d$, let $G_{i,j,z}$ be the subgraph of $G$ with $V(G_{i,j,z}):=V_i\cup V_j$ and $E(G_{i,j,k}):=\{vw\in E(G): v\in V_i, w\in V_j, f(vw)=z\}$. Note that possibly $i=j$. We now bound the number of edges in each $G_{i,j,z}$. 
	
First suppose that $G_{i,j,z}$ contains a directed path $(u,v,w)$. Let $\alpha:=uv$ and $\beta:=vw$. Thus $f(\alpha)=f(\beta)=z$, and $\alpha$ and $\beta$ are both subdivided $d'$ times, for some $d'\in\{0,1,\dots,d\}$. Moreover, $\phi(u)=\phi(v)$ and $\phi(\alpha_i)=\phi(\beta_i)$ for each $i\in\{1,\dots,d'\}$. Thus $(u,\alpha_1,\dots,\alpha_{d'},v,\beta_1,\dots,\beta_{d'})$ is a repetitively coloured path in $G'$ on $2d'+2\leq 4d+4$ vertices, as illustrated in \cref{PathCyclePath}(a). This contradiction shows that $G_{i,j,z}$ contains no 2-edge directed path. 
	
Let $\overline{G_{i,j,z}}$ be the undirected graph underlying $G_{i,j,z}$. Suppose that $\overline{G_{i,j,z}}$ contains a cycle $C=(v_1,v_2,\dots,v_k)$ for some $k\geq 3$. Each edge of $C$ is subdivided $d'$ times, for some $d'\in\{0,1,\dots,d\}$. Since $G_{i,j,z}$ contains no 2-edge directed path, without loss of generality, the edges of $C$ are oriented $v_1v_2, v_3v_4, v_5v_6,\dots$ and $v_3v_2, v_5v_4, v_7v_6,\dots$. Thus $k\geq 4$. By construction, $i=\phi(v_1)=\phi(v_3)=\cdots$ and $j=\phi(v_2)=\phi(v_4)=\cdots$. Let $\alpha:=v_1v_2$ and $\beta:=v_3v_2$ and $\gamma:=v_3v_4$ and $\delta:= v_5v_4$. By construction $\phi(\alpha_i)=\phi(\beta_{d'+1-i})=\phi(\gamma_i)=\phi(\delta_{d'+1-i})$ for each $i\in\{1,\dots,k\}$. Thus $$(v_1,\alpha_1,\alpha_2,\dots,\alpha_{d'},v_2,\beta_{d'},\beta_{d'-1},\dots,\beta_1,v_3,\gamma_1,\gamma_2,\dots,\gamma_{d'},v_4,\delta_{d'},\delta_{d'-1},\dots,\delta_1)$$ is a repetitively coloured path in $G'$ on $4d'+4\leq 4d+4$ vertices, as illustrated in \cref{PathCyclePath}(b). 
	
This contradiction shows that $\overline{G_{i,j,z}}$ is acyclic. Thus $|E(\overline{G_{i,j,z}})| \leq |V_i|+|V_j|-1 $. Hence 
\begin{align*}
|E(G)| \leq \!\!\! \sum_{i,j\in\{1,\dots,c\}} \!\!\!\!\! |Z| (|V_i|+|V_j|-1)
	= & |Z| \left( -\frac{c(c+1)}{2} + \!\!\! \sum_{i\in\{1,\dots,c\}} \!\!\!\! c|V_i| \right)\\
	= &c\,|Z| \left( |V(G)| -\frac{c+1}{2} \right).
\end{align*}
The result follows by the above bounds on $|Z|$. 
\end{proof}

\begin{figure}
\centering
\includegraphics[width=\textwidth]{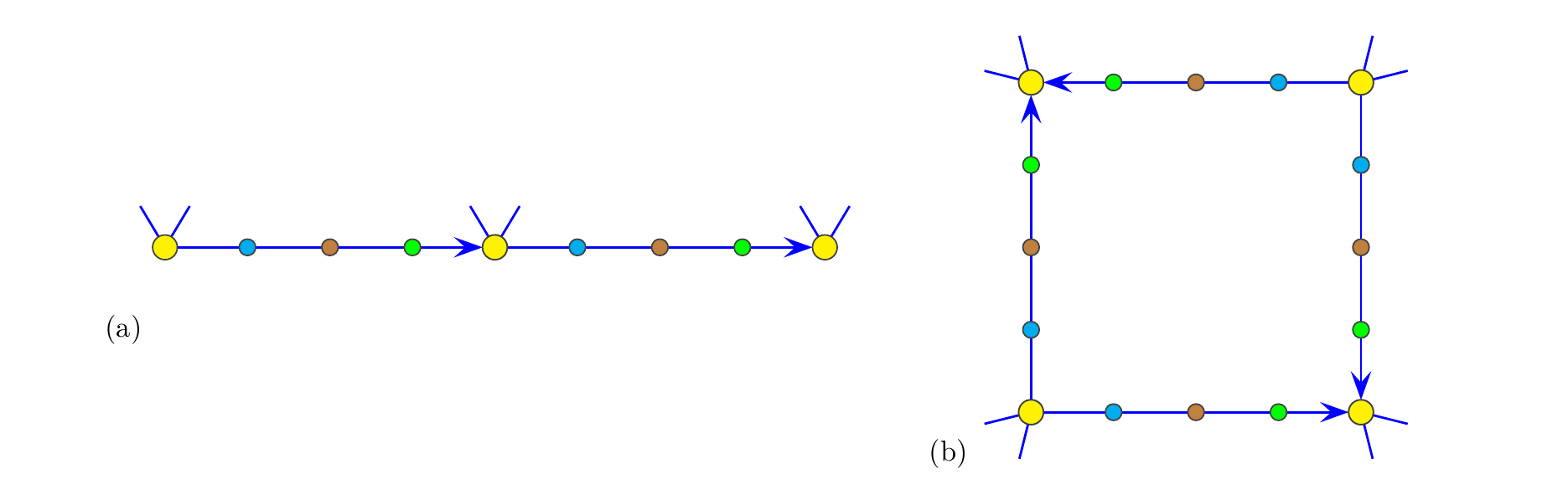}
\caption{Illustration for \cref{ExtremalSubdivision}.\label{PathCyclePath}}
\end{figure}

\cref{NonRepSub}(c) and \cref{ExtremalSubdivision} imply:

\begin{lem}
\label{GeneralCompleteSubdivision}
For every graph $G$ and integer $d\geq 0$, if $G'$ is any $(\leq d)$-subdivision of $G$, then
$$\pi(G') \geq \left(\frac{|E(G)|}{|V(G)|-1}\right)^{1/(d+1)} -3 .$$
\end{lem}

\subsection{Complete Graphs}

\cref{GenSubdiv2Graph} with $G=K_n$ implies that there is a function $f$ such that for every $(\leq d)$-subdivision $H$ of $K_n$, 
$$\pi(H)\geq f(n,d),$$ 
and  $\lim_{n\rightarrow\infty}f(n,d)=\infty$ for all fixed $d$. \citet{NOW11} obtained reasonable bounds on $f$, and indeed, for fixed $d\geq 2$, determined $\pi(K_n^{(d)})$ up to a constant factor.

\begin{lem}[\citep{NOW11}]
\label{CompleteGraphSubdivUpperBound}
Let $A\geq1$ and $B\geq2$ and $d\geq2$ be integers. If $n\leq A\cdot B^d$ then 
$$\pi(K_n^{(d)})\leq A+8B.$$
\end{lem}

\begin{proof}
Let $(c_1,\dots,c_d)$ be a nonrepetitive sequence such that $c_1=0$ and $\{c_2,c_3,\dots,c_d\}\subseteq\{1,2,3\}$. Let $\preceq$ be a total ordering of the original vertices of $K_n^{(d)}$. Since $n\leq A\cdot B^d$, the original vertices of $K_n^{(d)}$ can be labelled $$\{v=\blah{v_0,v_1,\dots,v_d}:1\leq v_0\leq A,1\leq v_i\leq B,1\leq i\leq d\}.$$ 

Colour each original vertex $v$ by $\col(v):=v_0$. 
Consider a pair of original vertices $v$ and $w$ with $v\prec w$.
If $(v,r_1,r_2,\dots,r_d,w)$ is the transition from $v$ to $w$, then for $i\in[1,d]$, colour the division vertex $r_i$ by $$\col(r_i):=(\delta(v_i,w_i),c_i,v_i),$$ where
$\delta(a,b)$ is the indicator function of $a=b$. 
We say this transition is \emph{rooted} at $v$.
Observe that the number of colours is at most $A+2\cdot4\cdot B=A+8B$.

Every transition is coloured $$\big(x_0,(\delta_1,c_1,x_1),(\delta_2,c_2,x_2),\dots,(\delta_d,c_d,x_d),x_{d+1}\big)$$
for some $x_0\in[1,A]$ and $x_1,\dots,x_{d+1}\in[1,B]$ and $\delta_1,\dots,\delta_d\in\{\text{true},\text{false}\}$.
Every such transition is rooted at the original vertex $\blah{x_0,x_1,\dots,x_d}$. That is, the colours assigned to a transition determine its root.

Suppose on the contrary that $P=(a_1,\dots,a_{2s})$ is a repetitively coloured path in $K_n^{(d)}$. Since every original vertex receives a distinct colour from every division vertex, for all $i\in[s]$, $a_i$ is an original vertex if and only if $a_{i+s}$ is an original vertex, and $a_i$ is a division vertex if and only if $a_{i+s}$ is a division vertex. 

By construction, every transition is coloured nonrepetitively. 
Thus $P$ contains at least one original vertex, implying $\{a_1,\dots,a_s\}$ contains at least one original vertex. 
If $\{a_1,\dots,a_s\}$ contains at least two original vertices, then $\{a_1,\dots,a_s\}$ contains a transition $(a_i,\dots,a_{i+d+1})$, implying $(a_{s+i},\dots,a_{s+i+d+1})$ is another transition receiving the same tuple of colours. Thus $(a_i,\dots,a_{i+d+1})$ and $(a_{s+i},\dots,a_{s+i+d+1})$ are rooted at the same original vertex, implying $P$ is not a path. 

Now assume there is exactly one original vertex $a_i$ in $\{a_1,\dots,a_s\}$. Thus $a_{s+i}$ is the only original vertex in $\{a_{s+1},\dots,a_{2s}\}$. Hence $(a_i,\dots,a_{s+i})$ is a transition, implying $s=d+1$. Without loss of generality, $a_i\prec a_{s+i}$ and this transition is rooted at $a_i$. 

Let $v:=a_i$ and $w:=a_{s+i}$. For $j\in[1,d]$, the vertex $a_{i+j}$ is the $j$-th vertex in the transition from $v$ to $w$, and is thus coloured $(\delta(v_j,w_j),c_j,v_j)$.

Suppose that $i\leq s-1$. Let $x$ be the original vertex such that the transition between $w$ and $x$ contains $\{a_{s+i+1},\dots,a_{2s}\}$. 
Now $$\col(a_{s+i+1})=\col(a_{i+1})=(\delta(v_1,w_1),c_1,v_1).$$
Since $c_1\neq c_d$, we have $w\prec x$.
For $j\in[1,s-i]$, the vertex $a_{s+i+j}$ is the $j$-th vertex in the transition from $w$ to $x$, and thus $$(\delta(w_j,x_j),c_j,w_j)=\col(a_{s+i+j})=
\col(a_{i+j})=(\delta(v_j,w_j),c_j,v_j).$$ 
In particular, $v_j=w_j$ for all $j\in[1,s-i]$. Note that if $i=s$ then this conclusion is vacuously true.

Now suppose that $i\geq 2$. Let $u$ be the original vertex such that the transition between $u$ and $v$ contains $\{a_1,\dots,a_{i-1}\}$. 
Now $$\col(a_{i-1})=\col(a_{s+i-1})=(\delta(v_d,w_d),c_d,v_d).$$
Since $c_d\neq c_1$, we have $u\prec v$. 
For $j\in[s-i+1,d]$, the vertex $a_{i+j-s}$ is the $j$-th vertex in the transition from $u$ to $v$, and thus $$(\delta(u_j,v_j),c_j,u_j)=\col(a_{i+j-s})=\col(a_{i+j})=(\delta(v_j,w_j),c_j,v_j).$$ In particular, $v_j=u_j$ and $\delta(v_j,w_j)=\delta(u_j,v_j)$. Thus $v_j=w_j$ for all $j\in[s-i+1,d]$. Note that if $i=1$ then this conclusion is vacuously true.

Hence $v_j=w_j$ for all $j\in[1,d]$. 
Now $v$ is coloured $v_0$, and $w$ is coloured $w_0$. 
Since $v=a_i$ and $w=a_{s+i}$ receive the same colour, $v_0=w_0$.
Therefore $v_j=w_j$ for all $j\in[0,d]$. That is, $v=w$, which is the desired contradiction. 

Therefore there is no repetitively coloured path in $K_n^{(d)}$.
\end{proof}

\begin{thm}[\citep{NOW11}]
\label{CompleteGraphSubdiv}
For $d\geq2$, 
$$\left(\frac{n}{2}\right)^{1/(d+1)}\leq \pi(K_n^{(d)})\leq 9\ceil{n^{1/(d+1)}}.$$
\end{thm}

\begin{proof}
The lower bound follows from \cref{ExtremalSubdivision}. The upper bound is \cref{CompleteGraphSubdivUpperBound} with $B=(n/8)^{1/(d+1)}$ and $A=8B$. 
\end{proof}

As mentioned earlier, \citet{NOW11} showed that for a $O(1)$-colourable subdivision of $K_n$, 
$\Theta(\log n)$ division vertices per edge is best possible. 

\begin{thm}
\label{CompleteGraphSubdivLowerBound}
For every $n\in\NN$, the $(1 + 2\ceil{17\log_2 n})$-subdivision of $K_n$ has a nonrepetitive $5$-colouring. 
Conversely, if $H$ is a subdivision of $K_n$ and $\pi(H)\leq c$ then some edge of $K_n$ is subdivided at least $\log_{c+3}(\frac{n}{2})-1$ times.
\end{thm}

\begin{proof}
The upper bound follows from \cref{5Subdiv}. For the lower bound, suppose that  $H$ is a $(\leq d)$-subdivision of $K_n$ and $\pi(H)\leq c$. By \cref{GeneralCompleteSubdivision}, $(\frac{n}{2})^{1/(d+1)}-3\leq \pi(H)\leq c$.
That is, $\log_{c+3}\frac{n}{2}-1\leq d$.
Hence some edge of $H$ is subdivided at least 
$\log_{c+3}(\frac{n}{2})-1$ times.
\end{proof}

Now consider nonrepetitive colourings of the 1-subdiivsion of $K_n$. \citet{NOW11} proved the following upper bound\footnote{The proof of \cref{CompleteSubdiv1} corrects an error in the proof in \citep{NOW11}.} 

\begin{prop}[\citep{NOW11}]
\label{CompleteSubdiv1}
For every $n\in\NN$, 
$$\pi(K_n^{(1)})\leq \tfrac{5}{2} n^{2/3} + O(n^{1/3}).$$
\end{prop}

\begin{proof}
Let $N:=\ceil{n^{1/3}}$. In $K_{N^3}'$, let $\{\langle i,k\rangle: 1\leq i\leq N^2,1\leq k\leq N\}$ be the original vertices, and let $\langle i,k;j,\ell\rangle$ be the division vertex having
$\langle i,k\rangle$  and $\langle j,\ell\rangle$  as its neighbours.

Colour each original vertex \blah{i,j} by $A_i$.
Colour each division vertex \blah{i,k;j,\ell} by $B_{k,\ell}$ if $i<j$. 
Colour each division vertex \blah{i,k;i,\ell} by $C_{k,\ell}$ where $k<\ell$.

Suppose that $PQ$ is a repetitively coloured path. Since original and division vertices are assigned distinct colours, $|P|$ is even.

First suppose that $|P|\geq 4$. Then $P$ contains some transition $T$. Observe that each transition is uniquely identified by the three colours that it receives. In particular, the only transition coloured $A_iB_{k,\ell}A_j$ with $i<j$ is $\langle i,k\rangle\langle i,k;j,\ell\rangle\langle j,\ell\rangle$. And the only transition coloured $A_iC_{k,\ell}A_i$ is $\langle i,k\rangle\langle i,k;i,\ell\rangle\langle i,\ell\rangle$. Thus $T$ is repeated in $Q$, which is a contradiction.

Otherwise $|P|=2$. Thus $PQ$ is coloured $A_iC_{k,\ell}A_iC_{k,\ell}$ for some $k<\ell$. But the only edges coloured $A_iC_{k,\ell}$ are the two edges in the transition $\langle i,k\rangle\langle i,k;i,\ell\rangle\langle i,\ell\rangle$, which again is a contradiction

Hence there is no repetitively coloured path. The number of colours is $N^2+N^2+\binom{N}{2}\leq\frac{5}{2}N^2
\leq \frac{5}{2} n^{2/3}+O(n^{1/3})$.
\end{proof}

\begin{open}
What is $\pi(K_n^{(1)})$? \cref{ExtremalSubdivision,CompleteSubdiv1} imply $(\frac{n}{2})^{1/2} \leq \pi(K_n^{(1)}) \leq O(n^{2/3})$. The slightly better lower bound  $\pi(K_n^{(1)})\geq\sqrt{n}$ follows from 
the previously mentioned lower bound $\pi(K_n^{(1)})\geq\st(K_n^{(1)})\geq\sqrt{n}$ by  \citet{Wood-DMTCS05}. The best known upper bound is $O(n^{2/3})$ in \cref{CompleteSubdiv1}.
\end{open}


\begin{open}[\citep{DJKW16}]
Is there a function $f$ such that $\pi(G/M)\leq f(\pi(G))$ for every graph $G$ and for every matching $M$ of $G$, where $G/M$ denotes the graph obtained from $G$ by contracting the edges in $M$? This would generalise the results in \cref{SubdivLowerBounds} about subdivisions (when each edge in $M$ has one endpoint of degree 2).
\end{open}

\begin{open}
Does every graph have a nonrepetitively $4$-choosable subdivision? Even $3$-choosable might be possible (although this is open even for paths). 
\end{open}


%
%
%

\section{Bounded Expansion}
\label{BoundedExpansion}

\citet{Sparsity} introduced the notion of bounded expansion graph classes as a robust measure of graph sparsity. 
The main result in this section is that graphs with bounded nonrepetitive chromatic number have bounded expansion, as proved by \citet*{NOW11}. 

For $r\in \mathbb{N}$, a graph $H$ is an \emph{$r$-shallow minor} of a  graph $G$ if there is a set $\mathcal{X}$ of pairwise disjoint connected induced subgraphs of $G$, each with radius at most $r$, such that $H$ is isomorphic to a subgraph of the graph obtained from $G$ by contracting each subgraph in $\mathcal{X}$ into a vertex. For a graph $G$, \citet{Sparsity} defined $\nabla_r(G)$ to be the maximum, taken over all $r$-shallow minors $H$ of $G$, of the average degree of $H$. 

A graph class $\GG$ has \emph{bounded expansion} with \emph{bounding function} $f$ if $\nabla_r(G)\leq f(r)$ for each $G\in\GG$ and $r\in\mathbb{N}$. We say $\GG$ has \emph{linear expansion} if, for some constant $c$, for all $r\in\NN$, every graph $G\in\GG$ satisfies $\nabla_r(G) \leq cr$. Similarly, $\GG$ has \emph{polynomial expansion} if, for some constant $c$, for all $r\in\NN$, every graph $G\in\GG$ satisfies $\nabla_r(G) \leq cr^c$. For example, when $f(r)$ is a constant, $\mathcal{G}$ is contained in a proper minor-closed class. As $f(r)$ is allowed to grow with $r$ we obtain larger and larger graph classes.

Bounded expansion classes can be characterised by excluded subdivisions. A rational number $r$ is a \emph{half-integer} if $2r$ is an integer. For a half-integer $r$, a graph $H$ is an \emph{$r$-shallow topological minor} of a graph $G$ if a $(\leq 2r)$-subdivision of $H$ is a subgraph of $G$. \citet{Sparsity} defined $\widetilde{\nabla}_r(G)$ to be the maximum of $\frac{|E(H)|}{|V(H)|}$ taken over all $r$-shallow topological minors $H$ in $G$. 

Now add nonrepetitive chromatic number into the picture. Consider a graph $G$. By definition, some subgraph $G'$ of $G$ is a $(\leq 2r)$-subdivision of some graph $H$ with $\frac{|E(H)|}{|V(H)|} = \widetilde{\nabla}_r(G)$. By  \cref{ExtremalSubdivision}, 
$$|E(H)| \leq 2\pi(G')^{2r+1} ( |V(H)| - \tfrac{c+1}{2} )  < 2\pi(G')^{2r+1} |V(H)|  .$$
Since $G'$ is a subgraph of $G$, we have $\pi(G')\leq\pi(G)$. Thus
\begin{equation}
\label{TopoNablaPi}
\widetilde{\nabla}_r(G) = \frac{|E(H)|}{|V(H)|} < 2\pi(G')^{2r+1} \leq 2\pi(G)^{2r+1}. 
\end{equation}
It follows from a result of \citet{Dvorak08} that $\nabla_r(G)$ and $\widetilde{\nabla}_r(G)$ are equivalent in the sense that 
\begin{equation}
\label{NablaTopoNabla}
\widetilde{\nabla}_r(G) 
\leq
\nabla_r(G)
\leq
4(4\widetilde{\nabla}_r(G) )^{(r+1)^2}.
\end{equation}
\cref{NablaTopoNabla,TopoNablaPi} imply:
\begin{align*}
\nabla_r(G)
\leq
4(8 \pi(G)^{2r+1} )^{(r+1)^2}.
\end{align*}
This implies the following theorem:

\begin{thm}[\citep{NOW11}]
\label{PiBoundedExpansion}
For each $c\in\NN$ the class of graphs $\{G:\pi(G)\leq c\}$ has bounded expansion.
\end{thm}

\citet{NOW11} actually proved a stronger result that we now present. The idea originates in a notion of 
\citet{DujWoo05}, who defined a graph parameter $\alpha$ to be {\em topological} if for some function $f$, every graph $G$ satisfies $\alpha(G)\leq f(\alpha(G^{(1)}))$ and $\alpha(G^{(1)})\leq f(\alpha(G))$, where $G^{(1)}$ is the $1$-subdivision of $G$. For instance, tree-width and genus are topological, but chromatic number is not. Sightly more generally, \citet{NOW11} defined a graph parameter $\alpha$ to be {\em strongly topological} if for some function $f$, for every graph $G$ and every $(\leq 1)$-subdivision $H$ of $G$, we have $\alpha(G)\leq f(\alpha(H))$ and $\alpha(H)\leq f(\alpha(G))$. 
\cref{Subdiv2Graph,NonRepSub}(a) imply:

\begin{thm}[\citep{NOW11}] 
	$\pi$ is strongly topological. 
\end{thm}

\citet{NOW11} characterised bounded expansion classes as follows. A graph parameter $\alpha$ is {\em monotone} if $\alpha(H)\leq\alpha(G)$ for every subgraph  $H$ of $G$, and $\alpha$ is {\em degree-bound} if for some function $f$, every graph $G$ has a vertex of degree at most $f(\alpha(G))$. \citet{NOW11} characterised bounded expansion classes as follows:

\begin{lem}[\citep{NOW11}]
\label{TopoExp}
A graph class $\mathcal C$ has bounded expansion if and only if there exists a strongly topological, monotone, degree-bound graph parameter $\alpha$ and a constant $c$ such that $\mathcal C\subseteq\{G: \alpha(G)\leq c\}$.
\end{lem}

By definition, $\pi$ is a monotone graph parameter. By \cref{ExtremalPi}, every graph $G$ has a vertex of degree at most $2\pi(G)-2$, implying that $\pi$ is degree-bound. Thus \cref{TopoExp} is applicable with $\mathcal{C}=\{G:\pi(G)\leq c\}$ where $\alpha$ is nonrepetitive chromatic number itself. In particular, \cref{TopoExp} implies \cref{PiBoundedExpansion}. 

The following well-known folklore theorem takes this result further, and actually characterises bounded expansion classes in terms of nonrepetitive colourings. 

\begin{thm}
\label{BoundedExpansionIffNonRep}
A graph class $\GG$ has bounded expansion if and only if there is a function $f$ such that for every graph $G\in\GG$ and every  $k\in\NN$, there is an $f(k)$-colouring of $G$ with no repetitively coloured path on at most $2k$ vertices. 
\end{thm}

\begin{proof}
The following definition is useful for the proof. A coloring of a graph $G$ is \emph{$p$-centred} if for every connected subgraph $H$ of $G$, some color appears exactly once in $H$ or $H$ is assigned at least $p$ colours. Let $\chi_p(G)$ be the minimum number of colours in a $p$-centred colouring of $G$. 

$(\Longrightarrow)$ (This direction is similar to a result of \citet{Gryczuk-IJMMS07} about  weak colouring numbers and a result of \citet{Yang09} about transitive fraternal augmentations.)\ Let $\GG$ be a graph class with bounded expansion. \citet{Sparsity} proved that $\GG$ has bounded $\chi_p$ for every $p\in\NN$. That is, there exists a function $f$ (depending on the expansion function of $\GG$) such that $\chi_p(G)\leq f(p)$ for every graph $G\in\GG$. Consider a $(k+1)$-centred colouring $\phi$ of a graph $G\in\GG$ with at most $f(k+1)$ colours. Let $P$ be a path in $G$ on at most $2k$ vertices. If $P$ is repetitively coloured by $\phi$, then no colour appears exactly once in $P$, implying that $P$ is assigned at least $k+1$ colours, which is impossible for a repetitively coloured path on at most $2k$ vertices. Thus paths with at most $2k$ vertices in $G$ are nonrepetitiely coloured. 

$(\Longleftarrow)$  Let $\GG$ be a graph class and let $f$ be a function such that for every graph $G\in\GG$ and every  $k\in\NN$, there is an $f(k)$-colouring of $G$ with no repetitively coloured path on at most $2k$ vertices. 
To show that $\widetilde{\nabla}_r(G)$ is bounded for $G\in\GG$, we use the argument at the start of this section. By definition, some subgraph $G'$ of $G$ is a $(\leq 2r)$-subdivision of some graph $H$ with 
$\frac{|E(H)|}{|V(H)|} = \widetilde{\nabla}_r(G)$. Let $k:=4r+2$. By assumption, there is an $f(k)$-colouring of $G$ with no repetitively coloured path on at most $2k$ vertices. The same property holds for the subgraph $G'$ of $G$. 
By \cref{ExtremalSubdivisionExtra} with $d=2r$ and $4d+4=2k$, we have $\widetilde{\nabla}_r(G) = \frac{|E(H)|}{|V(H)|} < 2 f(k)^{2r+1}$. By \cref{NablaTopoNabla}, $\nabla_r(G)\leq4(8 f(k))^{2r+1} )^{(r+1)^2}$, 
which is a function of $r$. Hence $\GG$ has bounded expansion.
\end{proof}

Many graph classes with bounded expansion also have bounded nonrepetitive chromatic number (such as planar graphs, graphs excluding a fixed minor, graphs excluding a fixed subdivision, $(g,k)$-planar graphs, etc.) The following open problem is probably the most important direction for future research on nonrepetitive colourings. 

\begin{open}
Do graphs of linear / polynomial / single exponential expansion have bounded nonrepetitive chromatic number? 	Single exponential expansion would be best possible here, since by \cref{CompleteGraphSubdivLowerBound}, the $o(\log n)$-subdivision of $K_n$ has unbounded $\pi$. It is even possible that if a graph class $\mathcal{C}$ has bounded expansion with expansion function $f(r)$, then for some constant $c$, every graph $G\in\mathcal{C}$ satisfies
	\begin{equation}
	\label{Speculative}
	\pi(G)\leq c \sup_r f(r)^{2/r}
	\end{equation}
Note that graphs $G$ with maximum degree $\Delta$ have bounded expansion with expansion function $f(r)\leq\Delta^r$. So if \eqref{Speculative} holds, then $\pi(G)\leq c \sup_r \Delta^2$, implying \eqref{DeltaSquared}. This is the reason for the 2 in \eqref{Speculative}. This question is highly speculative. Whether graphs with linear or polynomial expansion have bounded $\pi$ is already a challenging question. Note that 
\cref{Speculative} was jointly formulated with Gwena\"el Joret.
\end{open}

\subsection*{Acknowledgements} Thanks to J{\'a}nos Bar{\'a}t, Vida Dujmovi\'c, Louis Esperet,  Gwena\"el Joret, Jakub Kozik, Jaroslav Ne{\v{s}}et{\v{r}}il, Patrice Ossona de Mendez, Bartosz Walczak, Paul Wollan and Timothy Wilson,  with whom I have collaborated on nonrepetitive graph colouring. Thanks to Louis Esperet, Kevin Hendrey, Robert Hickingbotham, Gwena\"el Joret and Piotr Micek for insightful comments about this survey. Thanks to Arseny Shur for helpful discussions about reference \citep{Shur10}. Thanks to Timothy Chan for pointing out  reference~\citep{ArithmeticoGeometricSequence}.  Thanks to St\'ephan Thomass\'e for asking a good question. 

  \let\oldthebibliography=\thebibliography
  \let\endoldthebibliography=\endthebibliography
  \renewenvironment{thebibliography}[1]{%
    \begin{oldthebibliography}{#1}%
      \setlength{\parskip}{0ex}%
      \setlength{\itemsep}{0ex}%
  }{\end{oldthebibliography}}

\bibliographystyle{../../BibTex/myNatbibStyle}
\bibliography{../../BibTex/myBibliography}

\end{document}